\documentclass[11pt, reqno]{amsart}
\usepackage{amsmath,amsthm,amscd,amssymb,amsfonts, amsbsy}

\usepackage{url}
\usepackage{latexsym, color, enumerate}
\usepackage{pxfonts}
\usepackage{mathrsfs}
\usepackage{marginnote}
\usepackage{todonotes}
\usepackage[colorlinks=true, pdfstartview=FitV, linkcolor=blue, citecolor=blue, urlcolor=blue]{hyperref}
\usepackage{mathtools}
\mathtoolsset{showonlyrefs}
\usepackage[margin=1 in]{geometry}
\usepackage{comment}

\allowdisplaybreaks[4]
\usepackage{lipsum}

\newcommand\blfootnote[1]{%
  \begingroup
  \renewcommand\thefootnote{}\footnote{#1}%
  \addtocounter{footnote}{-1}%
  \endgroup
}
\usepackage[numbers,sort]{natbib}

\usepackage{appendix}

\newtheorem{innercustomgeneric}{\customgenericname}
\providecommand{\customgenericname}{}
\newcommand{\newcustomtheorem}[2]{%
  \newenvironment{#1}[1]
  {%
   \renewcommand\customgenericname{#2}%
   \renewcommand\theinnercustomgeneric{##1}%
   \innercustomgeneric
  }
  {\endinnercustomgeneric}
}

\newcustomtheorem{customthm}{Theorem}

\usepackage{tikz}
\usetikzlibrary{decorations.pathmorphing}
\usepackage{multicol}
\usepackage{cancel}

\numberwithin{equation}{section}

\theoremstyle{plain}
\newtheorem{theorem}{Theorem}[section]
\newtheorem{lemma}[theorem]{Lemma}
\newtheorem{corollary}[theorem]{Corollary}
\newtheorem{proposition}[theorem]{Proposition}

\theoremstyle{definition}
\newtheorem{definition}[theorem]{Definition}

\newtheorem{example}[theorem]{Example}

\newtheorem{remark}[theorem]{Remark}

\newtheorem*{theorem*}{Theorem}

\newcommand{\bN}{\mathbb{N}}

\newcommand{\re}{\mathrm{Re}}
\newcommand{\im}{\mathrm{Im}}
\newcommand{\w}{\mathrm{w}}
\newcommand{\Op}{\mathrm{Op}}
\newcommand{\st}{\sup_{0 \leq t \leq T}}
\newcommand{\abs}[1]{\lvert#1\rvert}
\newcommand{\norm}[1]{\left\|#1\right\|}

\newcommand{\bracket}[1]{\left\{#1\right\}}
\newcommand{\jb}{\langle \xi \rangle}

\def\XXint#1#2#3{{\setbox0=\hbox{$#1{#2#3}{\int}$}
		\vcenter{\hbox{$#2#3$}}\kern-.5\wd0}}

\newcommand{\p}{\partial}

\newcommand{\R}{\mathbb{R}}

\begin{document}

\title[]{Smoothing effect for third order operators with variable coefficients}

\author[Federico and Tramontana]{Serena Federico and Davide Tramontana}

\address{Serena Federico
\newline \indent Dipartimento di Matematica, Universit\`a di Bologna
\newline \indent Piazza di Porta San Donato 5, 40126, Bologna, Italy}
\email{serena.federico2@unibo.it}

\address{Davide Tramontana
\newline \indent Dipartimento di Matematica, Universit\`a di Bologna
\newline \indent Piazza di Porta San Donato 5, 40126, Bologna, Italy}
\email{davide.tramontana4@unibo.it}

\subjclass[2020]{
35B45; 35B65; 35Q60; 35Q35; 35Q55; 35S10}
\keywords{Smoothing effect; Third order operators with variable coefficients; Dispersive equations.}
\blfootnote{S. Federico and D. Tramontana are members of GNAMPA.}

\begin{abstract}
In this work we study the smoothing effect of some variable coefficient operators of the form $D_t-A$, where $A$ is a Weyl-quantized pseudo-differential operator of order $m=2,3$. The class under consideration includes, among others, KdV-type and ultrahyperbolic Schr\"odinger operators.
We prove \textit{homogeneous} and \textit{inhomogeneous} smoothing estimates and use them to get well-posedness results for some NLIVPs with derivative nonlinearities. Finally, we investigate the so called non-trapping property of the bicharacteristic curves of the principal symbol of our operators.
\end{abstract}

\maketitle

\setcounter{secnumdepth}{5}
\setcounter{tocdepth}{5}
\tableofcontents

\parindent = 10pt     
\parskip = 8pt
\section{Introduction}

The aim of this work is to investigate the validity of local smoothing estimates for third order variable coefficient operators of the form 
\begin{align}\label{eqP}
    P(t,x,D_t,D_x)=D_t-A(x,D_x), 
\end{align}
$$\quad D_t:=-i\partial_t,\quad  D_x=(D_{x_1},\ldots, D_{x_n}):=(-i\partial_{x_1},\ldots,-i\partial_{x_n}),$$
where $A$ is a pseudo-differential operator of order $m=2,3$, i.e. in the class $\Psi^m_{1,0}(\R^n)$, such that its Weyl symbol $a\in S^m(\R^n):=S^m_{1,0}(\R^n)$  can be written as 
\begin{align}
    a (x,\xi)= a_m(x,\xi)+a_{m-1}(x,\xi),
\end{align}
where $a_j \in S^j(\R^n)$, $j=m-1,m,$ and $a_m$ is real valued. 
We will assume that the real part $\re (a)$ of the symbol satisfies \eqref{ellderivatives} and \eqref{smallcoeffa2a3} below, that is, we require:
\begin{itemize}
\item [(i)] There exists  $C>0$ such that 
\begin{equation}
\frac{1}{C}\abs{\xi}^{m-1}\leq \abs{\nabla_\xi \re (a)}\leq C\abs{\xi}^{m-1}, \quad \forall (x,\xi)\in \R^{n}\times \R^n;
\end{equation}
\item[(ii)] For all $\alpha, \beta \in \mathbb{N}_0^n$, $|\beta|\geq 1$, and for $\varepsilon>0$ sufficiently small, there exists $\lambda \in L^1([0,+\infty[)\cap C^0([0,+\infty[)$ strictly positive and non-increasing such that
\begin{equation}
\abs{\partial^\beta_x \partial_\xi^\alpha \re(a)} \leq \varepsilon\lambda(\abs{x})\abs{\xi}^{m-\abs{\alpha}}, \quad \forall (x,\xi)\in \R^{n}\times \R^n.
\end{equation}
\end{itemize}
We will also assume that the imaginary part of the symbol, that is $\im(a)= \im(a_{m-1}) \in S^{m-1}(\R^n)$, satisfies the following  smallness condition: for $\lambda \in L^1([0,+\infty[)\cap C^0([0,+\infty[)$ as above, there exists $c_0>0$ such that
   $$|\im(a_{m-1})(x,\xi)|\leq c_0 \lambda(\abs{x})\langle\xi\rangle^{m-1},\quad \forall (x,\xi)\in \R^n \times \R^n.$$ 

We recall that an equation exhibits a smoothing effect whenever its solution gains regularity with respect to the initial data (\textit{homogeneous smoothing effect}) and/or with respect to the inhomogeneous term of the equation (\textit{inhomogeneous smoothing effect}).
It is well-known that whenever a smoothing estimate is available for an evolution operator, then well-posedness results for the suitable corresponding nonlinear initial value problems (NLIVPs) can be established by means of standard techniques based on the contraction argument. 
The suitable NLIVPs we refer to are those with derivative and polynomial nonlinearities, where the order of derivation involved depends on the order of the operator. Let us remark that while polynomial nonlinearities can also be handled by using the so-called Strichartz estimates, when derivatives appear in the nonlinear term the use of smoothing estimates becomes crucial. 

Our interest in third-order operators of the form \eqref{eqP} is motivated by the applications to KdV-type operators with variable coefficients, possibly defined on $\R_t\times\R^n$ with $n\geq 1$. The importance of such equations comes from their many applications in water waves, plasma physics, nonlinear optics and algebraic geometry, due to the large number of conservation laws and the behavior of the solutions as solitary waves for large $t$ (see \cite{ZDVC}, \cite{Mi76} and the references therein).  Examples of operators of  type  \eqref{eqP} are, for instance, those studied in \cite{F24} and in \cite{KS},  which
also include the classical constant coefficient cases given by the (linear) KdV and ZK equations.

The problem we investigate here –  the \textit{local smoothing effect} for variable coefficient operators – was studied in \cite{CKS,D94,D96,F22, FR24,FS21,KPV04,KPRV05} for operators as in \eqref{eqP} with $A$ being an operator of order two instead of three, that is when $a=a_2+a_1$, with $a_2\in S^2(\R^n)$ real valued and elliptic, and $a_1\in S^1(\R^n)$. In this framework, \textit{homogeneous} local and microlocal smoothing estimates have been proved by Doi in \cite{D96} and by Craig, Kappeler and Strauss in \cite{CKS}, respectively. By a result in \cite{D96} known as Doi's lemma, Kenig et al. proved the general \textit{inhomogeneous} smoothing effect for variable coefficient Schr\"odinger operators in \cite{KPRV05}. In the same work, Kenig et al. established a version of Doi's lemma for ultrahyperbolic operators, a lemma which they used to get the smoothing effect of variable coefficient ultrahyperbolic Schr\"odinger operators. 
These type of operators, the ultrahyperbolic Schr\"odinger ones, have been recently studied by Ifrim, Pineau and Tataru in \cite{IPT}, who used Strichartz and bilinear estimates to get global well-posedness of the cubic problem, also weakening the regularity requirement on the initial datum with respect to previous results based on smoothing estimates. 
In the context of Shr\"odinger equations with time-dependent variable coefficients, the smoothing effect, both homogeneous and inhomogeneous, was studied in \cite{F22, FR24,FS21}. As for KdV-type equations with variable coefficients, they have attracted lots of attention recently; some results can be found in \cite{CKS2,KS,Cai}.  In particular, in \cite{KS} the smoothing effect is studied in any space dimension for operators with constant coefficient leading part, while in \cite{CKS2,Cai} the smoothing effect in presence of a variable coefficients leading  operator is studied in one-space dimension, and the gain of regularity is measured according to the degree of vanishing of the initial data at infinity, hence in suitable weighted Sobolev spaces. As for the local well-posedness of KdV-type equations with variable coefficients in one space-dimension, we wish to mention the recent paper \cite{MTZ}. Let us stress that the approach in \cite{MTZ} is completely different from ours, which is based on the use of smoothing estimates to handle derivative nonlinearities; however, on the other side, the approach in \cite{MTZ} allows to work in a low regularity regime, a problem we do not focus on here.
Other results, strictly confined to KdV-type equations with variable coefficients, hence in space dimension $n=1$, are those given in \cite{AAW, CKS2, GHGM, AAC}, where the gain of regularity and the local well-posedness of fully nonlinear equations are studied. The study of operators as in \eqref{eqP} is especially aimed at solving problems in any space dimension rather than the sole case $n=1$.  In other words, we wish to consider ZK operators with variable coefficients in any space dimension. These are a natural generalization of KdV-type operators that we often simply call of KdV-type. We finally stress that the techniques used in the present work and those used in the aforementioned papers \cite{AAW, CKS2, GHGM, AAC} are extremely different because of the different dimensions of the ambient space and the conditions on the possibly variable coefficients in the leading term.
As for operators of the form \eqref{eqP} when $A$ has any order $m\geq 2$, a \textit{weak} local smoothing effect has been proved by Chihara in \cite{C}. This result is based on the \textit{nontrapping condition}, that we discuss below, and provides a smoothing effect which is too weak – compared to that in \cite{KPRV05} for $m=2$ an to that proved here for $m=3$ – to be applied to solve NLIVPs. Our smoothing result, compared to that by Chihara when $m=2,3$, is tailored for a (suitable) gain of regularity in a Sobolev space of arbitrary order $s\in \mathbb{R}$ instead of just for $s=0$ as in \cite{C}, and yields the same derivative gains as Chihara's result: $(m-1)/2$ derivatives relative to the initial datum and $m-1$ derivatives relative to the forcing term. Nevertheless, whereas Chihara explicitly assumes nontrapping of the Hamiltonian flow of the principal symbol, we impose the directly verifiable smallness/decay condition \eqref{smallcoeffa2a3} together with the assumption \eqref{cond.ima} on the imaginary part.
Other operators of any order when the space dimension is fixed to be $n=1$, often called $m$-evolution, have been studied in \cite{AAC2} with different techniques and appeared later than our work.

The main novelty of this paper – apart from the study of third-order operators with variable coefficients in the leading part of $A$ in any space dimension – is the introduction of two conditions, \eqref{ellderivatives} and \eqref{smallcoeffa2a3}, under which we explicitly construct a function $q$ (see \eqref{explicitq} below) establishing the admissibility of $\re(a)$. This function plays the role of an escape function in the positive commutator method behind the smoothing estimates.

More precisely, we will consider pseudo-differential operators $A \in \Psi^m(\R^n)$, for $m=2,3$, whose Weyl symbol $a$ splits as $a=a_m+a_{m-1}$ and satisfies (i) and (ii) above. Under these assumptions on $\re(a)$– the splitting of the symbol and \eqref{ellderivatives} and \eqref{smallcoeffa2a3}  for $\re(a)$ – and under the smallness assumptions on $\im (a)$ introduced above, then $D_t-A$ has the desired \textit{smoothing effect}, that is the one stated in Theorem \ref{thm.LIVP}. A crucial point in the proof of Theorem \ref{thm.LIVP} is that $\re(a)$ is \textit{admissible} (see Definition \ref{def.admissible}), which we know to be true as a consequence of the validity of \eqref{ellderivatives} and \eqref{smallcoeffa2a3} (see Proposition \ref{propsuffadmiss}). In other words, \eqref{ellderivatives} and \eqref{smallcoeffa2a3} are sufficient conditions for the admissibility of a (real) symbol.

Let us point out that our notion of \textit{admissibile symbol (or operator)} applies to symbols (or operators) of any order. This allows us to apply the admissibility to get smoothing estimates for \eqref{eqP} when $A$ is of order $m=2,3,$ and to recover the well-known results in \cite{KPRV05} for Schr\" odinger and ultrahyperbolic Schr\" odinger operators with variable coefficients. Unfortunately, when $A$ has order higher than three, we do not have at our disposal strong enough a priori estimates to conclude the expected smoothing result for $D_t-A$.

We wish to highlight that in the pioneering works \cite{CKS,D96,KPRV05}, the authors use different strategies to prove their results; however, a common requirement is the validity of the so-called \textit{non-trapping condition} (see also \cite{MMT12,MMT14,MMT2021,ST02} for different smoothing-type problems for variable coefficient Schr\"odinger operators related with the non-trapping properties). The latter is a requirement on the bicharacteristic curves of the principal symbol $a_2$, which, roughly speaking, must escape from any compact set (see Definition \ref{nontrapped}). When the operator is elliptic and has variable coefficients, there are some conditions ensuring that the non-trapping condition holds.
When the operator is not elliptic and has variable coefficients, the behavior of the bicharacteristics is much more complicated and hard to predict, not to mention that these curves are difficult if not impossible to compute even in the second-order elliptic case (with variable coefficients, the constant coefficient case is trivial).
Because of the aforementioned reason, we did not rely on the strategies used in \cite{D96}, \cite{CKS} and \cite{KPRV05} based on geometric assumptions on the bicharacterics of the principal symbol, simply because we could not.
Nevertheless, even if we avoid direct assumptions on the bicharacteristics of the principal symbol of $A$ in \eqref{eqP}, of course there must be a relation between our conditions and the behavior of these curves.
It is important to point out that we work with non-elliptic operators in general, and all the results available in the literature related to the non-trapping properties hold for elliptic operators only. This motivated our investigation of the non-trapping properties of the "elliptic part" of our operators. More precisely, we will prove that the so called \textit{strongly elliptic points} (see Definition \ref{def.elliptic.cosphere}) are \textit{nontrapped} (see Definition \ref{nontrapped}). 
A direct consequence of our result, which holds under suitable conditions, is that the non-trapping property for globally elliptic operators of any order holds at any point of the phase space. As for the behavior of the bicharacteristic curves starting at non-strongly elliptic points,  it is a phenomenon that is difficult to describe and that we intend to investigate in the future.

We conclude this introduction by stating our main results followed by the plan of the paper.

Theorem \ref{thm.LIVP} below is a smoothing result for suitable operators of order $m=2,3$, while Theorem \ref{thm.NLIVP} is a local well-posedness result for evolution  operators of order three (the case $m=2$ was already solved in \cite{KPRV05}).
\begin{theorem}\label{thm.LIVP}
Let $a\in S^m(\R^n)$, with $m=2,3$ be such that \eqref{cond.a} and \eqref{cond.ima} hold with $\lambda(\abs{x})=\langle x\rangle^{-N}$, where $N\in\mathbb{N}$ and $N>1$. Assume also that $\re(a)$ satisfies  \eqref{ellderivatives} and \eqref{smallcoeffa2a3} with $\langle x\rangle^{-N}$.
Then, given $u_0 \in H^s(\R^n)$, $s\in \R$,
\begin{itemize}
    \item [(i)] If $f\in L^1([0,T];H^s(\R^n))$, the IVP \eqref{LIVP} has a unique solution $u\in C([0,T];H^s(\R^n))$ satisfying
    
    \begin{align}
    \sup_{0\leq t\leq T}\|u\|_s\lesssim e^{C_1T} \left(\|u_0\|_s+\int_0^T \|f(t)\|_sdt\right).
\end{align}
\item [(ii)]  If $f\in L^2([0,T];H^s(\R^n))$, the IVP \eqref{LIVP} has a unique solution $u\in C([0,T];H^{s}(\R^n))$ satisfying
\begin{align}
    \sup_{0\leq t \leq T}\|u\|^2_s+\int_0^T (\lambda(|x|)\Lambda ^{s+\frac{m-1}{2}}u(t,\cdot),\Lambda^{s+\frac{m-1}{2}}u(t,\cdot))_0\,dt\lesssim e^{C_1T} \left(\|u_0\|_s^2+\int_0^T \|f(t)\|_s^2 dt\right).
\end{align}
\item [(iii)]  If $\Lambda^{s-\frac{m-1}{2}}f\in L^2([0,T]\times \R^n;\lambda(\abs{x})^{-1} dx dt)$, the IVP \eqref{LIVP} has a unique solution $u\in C([0,T];H^{s}(\R^n))$ satisfying
\begin{align}
    &\sup_{0\leq t \leq T}\|u\|^2_s+\int_0^T (\lambda(|x|)\Lambda^{s+\frac{m-1}{2}}u(t,\cdot),\Lambda^{s+\frac{m-1}{2}}u(t,\cdot))_0\,dt\\
    & \quad \lesssim e^{C_1T} \left(\|u_0\|_s^2+\int_0^T (\lambda(|x|)^{-1}\Lambda^{s-\frac{m-1}{2}}f(t,\cdot),\Lambda^{s-\frac{m-1}{2}}f(t,\cdot))_0dt\right).
\end{align}
\end{itemize}
 \end{theorem}

\begin{theorem}\label{thm.NLIVP}
    Let $A=\mathrm{Op}^\w(a) \in \Psi^3(\R^n)$ be a pseudo-differential operator with symbol $a\in S^3(\R^n)$ satisfying \eqref{ellderivatives}, \eqref{smallcoeffa2a3}, \eqref{cond.a} and \eqref{cond.ima}. Let $N(u,\bar{u},D^\alpha u)$ be a nonlinearity of the form \eqref{Npqalpha}, and $\lambda(\abs{x}):=\langle x \rangle^{-2N}$ with $N\in\mathbb{N}$ such that $N>1$. Let also $s \in 2\mathbb{N}+1$, with $s> n+8N+9$, and $u_0 \in \mathscr{S}(\R^n)$. Then there exists $T=T(\norm{u_0}_{H_x^s}, \norm{\lambda(\abs{x})^{-1}u_0}_{H_x^{s-2}})$, such that \eqref{NLIVP3} has a unique solution $u$ defined in the time interval $[0,T]$ and satisfying 
\[
\begin{split}
u\in X_T^s:=\Bigl \lbrace u\in C([0,T]\times \R^n;\mathbb{C}); \ & \st\norm{u}_{ H^s_x} < +\infty, \;  \Bigl(\int_0^T \int_{\R^n}\lambda(\abs{x})\abs{\Lambda^{s+1}u}^2 \, dx \, dt\Bigr)^{1/2} < +\infty, \\
& \st\norm{\lambda (\abs{x})^{-1}u}_{ H^{s-4N-2}_x}  < +\infty, \,  \st\norm{ \lambda (\abs{x})^{-1} \partial_t u}_{ H_x^{s-4N-5}} <+\infty \Bigr \rbrace,
\end{split}
\]
where $ X_T^s$ is the Banach space with norm 
\[
\norm{u}^2_{X_{T}^s}:=\st\norm{u}_{H^s_x}^2+\int_0^T \int_{\R^n}\lambda(\abs{x})\abs{\Lambda^{s+1}u}^2 \, dx \, dt+\st\norm{\lambda (\abs{x})^{-1}u}_{ H^{s-4N-2}_x}^2+\st\norm{ \lambda (\abs{x})^{-1}\partial_t u}^2_{H_x^{s-4N-5}}.
\]
Moreover, we have the following continuity property:

for every $u_0 \in \mathscr{S}(\R^n)$ there exist a neighborhood $U_{u_0}$ of $u_0$ in $\mathscr{S}(\R^n)$, and a time $T'>0$, such that the map 
\[
U_{u_0} \ni v_0 \mapsto v \in X_{T'}^s,
\]
which associates to the initial datum $v_0$ the solution $v$ of \eqref{NLIVP3}, is continuous.
\end{theorem}
 
Finally, the organization of the paper is as follows. 
In Section \ref{sec.preliminaries}, we recall the basics of the pseudo-differential calculus that we will use throughout the paper. 

Section \ref{sec.adm} is centered on the notion of \textit{admissible} operators. 
We will start with the proof of a new version of
Doi's lemma for operators of order $m\geq 2$, which will lead to the definition of admissible operators. Once this notion is introduced, we will establish sufficient conditions for an operator to be admissible in Proposition \ref{propsuffadmiss}. Finally, we will exhibit some examples of operators that satisfy these conditions, both in the constant and variable coefficient case.

In Section \ref{sec.smoothing.estimates}, we use the a priori estimates of Lemma \ref{lem.smoothing} to prove Theorem \ref{thm.LIVP} for a class of operators of the form $D_t-A$, where $A$ is a pseudo-differential operator of order $m=2,3,$ with admissible real part (the admissibility of the real part $\re(a)$ of the symbol comes from the fact that the symbol satisfies  \eqref{ellderivatives} and \eqref{smallcoeffa2a3}) and with imaginary part satisfying an appropriate smallness assumption.
In particular, Theorem \ref{thm.LIVP} provides the local well-posedness in $H^s$ of \eqref{LIVP} under suitable regularity assumptions on the forcing term, as well as the smoothing estimates for the solution.

In Section \ref{sec.NLIVP} we prove the local well-posedness of the NLIVP \eqref{NLIVP3} stated in Theorem \ref{thm.NLIVP}. The proof is based on the smoothing estimates obtained in Section \ref{sec.smoothing.estimates} and on the standard contraction argument. 

Finally, in Section \ref{sec.nontrapping} we study the non-trapping phenomenon associated with the bicharacteristic curves of the principal symbol of the operators under discussion.

\section{Preliminaries}\label{sec.preliminaries}
\noindent \textbf{Notations:} Below we shall use the notation $(\cdot,\cdot)_0$ for the $L^2$-inner product, and $\|\cdot\|_{s}$ for the standard $H^s(\R^n)$-Sobolev norm. 
Also, given an operator $P$ and two functions $f$ and $g$, we will write $(Pf)g$ when the operator $P$ applies to $f$ only, and $Pfg$ when $P$ applies to $fg$. Moreover, we denote by $\mathbb{N}_0:=\mathbb{N}\cup \lbrace 0 \rbrace$ and by $C^\infty_b(\R^n)$ the set of bounded $C^\infty$ function on $\R^n$, with bounded derivatives at any order. Finally, for every $A,B\in\R$ positive, we will write $A\lesssim (\gtrsim) B$ whenever there exists a positive constant $C$ such that $A\leq (\geq) CB$. If $a,p \in C^\infty(\R^n\times \R^n)$ are two complex-valued functions, we sometimes write $H_a p$ to denote the Poisson bracket $\lbrace a,p\rbrace $.

In this section, we recall the main properties of (global) pseudo-differential operators on $\R^n$ and the asymptotic formulas building the Weyl pseudo-differential calculus. For a complete and detailed exposition of the topic, and for  the proofs of the results presented hereafter, we refer the interested reader to \cite{Ho3} and \cite{Le10}. 

We start with the definition of the so-called \textit{standard} symbol class $S^m(\R^n):=S_{1,0}^m(\R^n)$.

\begin{definition}
    Let $m \in \R$ and $a\in C^\infty(\R^n \times \R^n)$. We say that $a \in S^{m}(\R^n)$, that is that \textit{$a$ is a symbol of order $m$}, if for all $\alpha,\beta \in \mathbb{N}_0^n$ there exists a constant $C_{\alpha,\beta}>0$ such that
\[
|\partial_x^\beta\partial_\xi^\alpha a(x,\xi)|\leq C_{\alpha,\beta}\langle \xi\rangle^{m-|\alpha|},\quad  \forall (x,\xi)\in\R^n \times \R^n.
\]
Moreover, $S^{m}(\R^n)$ equipped with the seminorms
\begin{align}
    \gamma_{k,m}(a)=\sup_{(x,\xi)\in\R^n \times \R^n, |\alpha|+|\beta|\leq k} |(\partial_x^\beta\partial_\xi^\alpha a)(x,\xi)|\langle \xi\rangle^{-m+|\alpha|},\quad k\in\mathbb{N},
\end{align}
is a Frechét space.
\end{definition}

Given a symbol $a\in S^m(\R^n), m\in \R$, we define a pseudo-differential operator associated with $a$ by using some formulas called \textit{quantization formulas}. Different quantizations, i.e. quantization formulas, lead – in general, except for symbols independent of $x$ –  to different operators. On the other hand, different symbols can give the same operator in different quantizations. The most popular and used quantizations are the Kohn-Nirenberg and the Weyl quantization, both belonging to the following one-parameter family of quantizations called \textit{$t$-quantizations} 
\begin{align}
 \mathrm{Op}_t(a)u(x)= (2\pi)^{-n}\int_{\R^n \times \R^n}e^{i(x-y) \cdot \xi} a((1-t)x+ty,\xi)u(y)dyd\xi, \quad u \in \mathscr{S}(\R^n), \quad  t\in[0,1].   
\end{align}
When $t=0$ in the previous formula we get  the Kohn-Nirenberg quantization of $a$ (that we denote by $\mathrm{Op}(a)=\mathrm{Op}_0(a)$), while when $t=1/2$  we get the Weyl one (that we denote by $\Op^{\mathrm{w}}(a)=\Op_{1/2}(a)$). 
 In this work we will use the Weyl quantization, therefore we will deal with Weyl-quantized pseudo-differential operators. 

\begin{definition}
We call \textit{Weyl-quantized pseudo-differential operator} with symbol $a \in S^m(\R^n)$, $m\in\R$,  the oscillatory integral
\[
\mathrm{Op}^\w(a)u(x)= (2\pi)^{-n}\int_{\R^n \times \R^n}e^{i(x-y) \cdot \xi} a((x+y)/2,\xi)u(y)dyd\xi, \quad u \in \mathscr{S}(\R^n).
\]

In addition, we say that an operator \textit{$A$ is a pseudo-differential operator of order $m$}, that is $A\in \Psi^m(\R^n)$, if there exists $a \in S^m(\R^{n})$ such that $A=\mathrm{Op}^\w(a)$.
\end{definition}
Our choice to use the Weyl quantization is motivated by the convenient properties of the Weyl pseudo-differential calculus, properties that we will briefly recall below. In order to distinguish the Weyl from the Kohn-Nirenberg symbol, we will refer to the symbol of a Weyl-quantized pseudo-differential operator $A$ as the \textit{Weyl symbol}. 

\begin{theorem}\label{thm.wey.adjoint}
Let $A=\Op^\w(a) \in \Psi^m(\R^n)$, $m \in \R$. Then the adjoint $A^\ast$ of $A$ belongs to the class $\Psi^m(\R^n)$ and $A^*=\Op^\w(\bar{a})$.
In particular, if $a$ is real valued, then the corresponding Weyl-quantized operator is formally self-adjoint.
\end{theorem}

For the composition of Weyl-quantized pseudo-differential operators, the following asymptotic formula holds.

\begin{theorem}\label{thm.weyl.comp}
Let $A=\Op^\w(a) \in \Psi^{m_1}(\R^n)$ and $B = \Op^\w(b) \in \Psi^{m_2}(\R^n)$, with $m_1,m_2 \in \R$. Then there exists a uniquely determined symbol $a\# b\in S^{m_1+m_2}(\R^n)$ satisfying
$A \circ B=\Op^\w(a\# b)$ and such that
\begin{align}\label{Weyl.comp.formula}
    (a\# b)(x,\xi)\sim\sum_{\substack{k\geq 0}}2^{-k}\sum_{\substack{|\alpha|+|\beta|=k}}\frac{(-i)^{|\alpha|+|\beta|}}{\alpha!\beta!}(-1)^{|\beta|}(\partial_\xi^\alpha\partial_x^\beta a)(\partial_\xi^\beta\partial_x^\alpha b),
\end{align}
in the sense that, for all $N\in\mathbb{N}$,
\begin{align}
  (a\# b)(x,\xi)-\sum_{\substack{0\leq k\leq N}}2^{-k}\sum_{\substack{|\alpha|+|\beta|=k}}\frac{(-i)^{|\alpha|+|\beta|}}{\alpha!\beta!}(-1)^{|\beta|}(\partial_\xi^\alpha\partial_x^\beta a)(\partial_\xi^\beta\partial_x^\alpha b)  \in S^{m_1+m_2-N-1}(\R^n).
\end{align}
\end{theorem}
For the proof of Theorem \ref{thm.weyl.comp} see \cite{Le10} or \cite{FRO}.

Next, we state some continuity properties of
pseudo-differential operators and some a priori estimates, specifically the sharp G\aa rding and the Fefferman-Phong inequality, that we will use in the proof of Lemma \ref{lem.smoothing}. We will work with Weyl symbols and quantizations, but the same properties can be stated, for instance, in the Kohn-Nirenberg framework.

\begin{theorem}
    For any given $s,m\in \R$ and $a\in S^m(\R^{n})$, the operator $\mathrm{Op}^\w(a)$ is bounded from $H^{s+m}(\R^n)$ to $H^{s}(\R^n)$.
\end{theorem}

\begin{theorem}[Sharp G\aa rding and Fefferman-Phong inequality]\label{thm.SGFP}
    Let $a\in S^m(\R^n)$, $m\in \R$. Then,
    \begin{itemize}
        \item if $\re(a)(x,\xi)\geq 0$ for all $(x,\xi)\in\R^n \times \R^n$, then there exists $C>0$ such that
        \[ \re (Op^\w(a)u,u)_0\geq -C\|u\|^2_{\frac{m-1}{2}},\quad \forall u\in\mathscr{S}(\R^n) \quad (\textit{Sharp G\aa rding Inequality});\]
       \item if $a\in S^m(\R^n)$ is real valued, and $a(x,\xi)\geq 0$ for all $(x,\xi)\in\R^n \times \R^n$, then there exists $C>0$ such that
        \[ \re (Op^\w(a)u,u)_0\geq -C\|u\|^2_{\frac{m-2}{2}},\quad \forall u\in\mathscr{S}(\R^n) \quad (\textit{Fefferman-Phong Inequality}).\]
    \end{itemize}
\end{theorem}
For the proof of Theorem \ref{thm.SGFP} see \cite{Ho3}, \cite{HoUM} and \cite{FP78}.

Now, we define the class of classical pseudo-differential operators which includes differential operators.

\begin{definition}
    A function $a \in C^\infty(\R^n \times (\R^n\setminus \lbrace 0 \rbrace))$ is called \textit{positively homogeneous of order $m \in \R$} if for all $\lambda>0$ 
    \[
    a(x,\lambda\xi)=\lambda^ma(x,\xi), \quad \forall (x,\xi) \in \R^n \times (\R^n \setminus \lbrace 0 \rbrace).
    \]
\end{definition}
\begin{definition}\label{homsym}
    \textit{The space of homogeneous symbols of order $m$}, denoted $S^m_{\mathrm{hom}}(\R^n)$, is the space of all functions $a \in C^\infty(\R^n \times (\R^n\setminus \{0\}))$, positively homogeneous of order $m$, such that, for all $\alpha,\beta \in \mathbb{N}_0^n$, 
\[
|\partial_x^\beta\partial_\xi^\alpha a(x,\xi)|\leq C_{\alpha,\beta}\abs{ \xi}^{m-|\alpha|},\quad  \forall (x,\xi)\in \R^n \times (\R^n \setminus \lbrace 0 \rbrace),
\]
for some constant $C_{\alpha,\beta}>0$.
\end{definition}

\begin{definition}\label{classical}
A symbol $a\in S^m(\R^n)$ is called \textit{a classical symbol of order $m \in \R$} if, for all $j \in \mathbb{N}_0$, there exists a homogeneous symbol $a_{m-j}\in S^{m-j}_{\mathrm{hom}}(\R^n) $ such that
\begin{equation}\label{assum}
a \sim \sum_j (1-\chi)a_{m-j},
\end{equation}
that is, for all $N \in \mathbb{N}_0$,
\[
a-\sum_{j=0}^{N}(1-\chi)a_{m-j} \in S^{m-N-1}(\R^n),
\]
where, $\chi=\chi(\xi)$ is a fixed cut-off function which is identically $1$ near 0, that is,  $\chi \in C_c^\infty(\R_\xi^n)$, $0 \leq \chi \leq 1$ and, for instance, $\chi \equiv 1$ when $\abs{\xi}\leq 1/2$ and $\chi \equiv 0$ when $\abs{\xi}\geq 1$. If $a \in S^m(\R^n)$ is a classical symbol we write $a \in S^m_{\mathrm{cl}}(\R^n)$. Moreover, when $a \in S^m_{\mathrm{cl}}(\R^n)$ we say that  $A=\Op^\w(a)$ is a \textit{classical pseudo-differential operator}, and write $A=Op^\w(a) \in \Psi_{\mathrm{cl}}^m(\R^n)$.
\end{definition}
\begin{remark}\label{nocutoff}
    The use of the cut-off function $\chi$ in the previous definition is necessary. For a function to be a symbol we need the function to be $C^\infty$ while, in general, a homogeneous function is not. However, since we are interested in the behavior of such functions for $\abs{\xi} \geq 1$, we sometimes omit the cut-off function and, instead of \eqref{assum}, we write
    \[
    a \sim \sum_j a_{m-j}.
    \]
\end{remark}

Our next aim is to define the principal symbol of a pseudo-differential operator, which will lead to the notion of elliptic symbol and elliptic operator. 
To define the principal symbol  we need to recall that for all $m \in \mathbb{R}$ there exists a map
\[
\sigma_m:\Psi^{m}(\R^n)\rightarrow S^m(\R^n)/S^{m-1}(\R^n)
\]
such that 
\[
0 \rightarrow \Psi^{m-1}\rightarrow\Psi^{m}\xrightarrow{\sigma_m} S^m(\R^n)/S^{m-1}(\R^n)\rightarrow 0
\]
is a short exact sequence. Therefore,  the image $\sigma_m(A)=[a]$, where $a$ is the Weyl symbol of $A$ (see for instance H\"ormander \cite{Ho3}, p. 86).

\begin{definition}
Let $A \in \Psi^m(\R^n)$. \textit{The principal symbol of $A$} is defined as $\sigma_m(A) \in S^m(\R^n)/S^{m-1}(\R^n)$.
\end{definition}
\begin{remark}\label{classprinc}
    If $A=\Op^\w(a) \in \Psi^m_{\mathrm{cl}}(\R^n)$ is a classical pseudo-differential operator, and
    \[
    a \sim \sum_{j}a_{m-j},
    \]
    (here we have used the convention in Remark \ref{nocutoff}), then we identify the principal symbol of $A$ with the leading part $a_m\in S^m(\R^n)$ in the asymptotic expansion.
\end{remark}

\begin{definition}
     A symbol $a \in S^m(\R^n)$ is said to be \textit{(uniformly) elliptic} if there exist some constants $c,C>0$ such that 
\[
\abs{\xi}\geq C \Longrightarrow \abs{a(x,\xi)} \geq c \langle \xi\rangle^m. 
\]
\end{definition}
\begin{definition}\label{def.ellipt.op}
Let $A \in \Psi^m(\R^n)$. The operator $A$ is said to be \textit{(uniformly) elliptic} if its principal symbol $\sigma_m(A)$ is elliptic.
\end{definition}
Since we will refer to symbols in different quantizations, we want to stress that Definition \ref{def.ellipt.op} does not depend on the choice of the quantization.

We conclude this section with a few more definitions that will appear throughout the paper.

\begin{definition}
    \textit{The Hamilton vector field} of $a\in C^\infty(\R^n \times \R^n;\R)$ is defined as the unique vector field $H_{a}\in T(\R^n \times \R^n)$ such that
    \[
    d a(V)=\sigma(V,H_{a}), \quad \forall V \in T(\R^n \times \R^n),
    \]
where $\sigma$ is the standard symplectic form on $\R^n \times \R^n$.
\end{definition} 
The standard symplectic form $\sigma$, in coordinates
$(x_1,...,x_n,\xi_1,...,\xi_n)$ on $\R^n \times \R^n$, is given by $\sigma=\sum_j d\xi_j \wedge dx_j$, therefore the Hamilton vector field of $a$  takes the form
$H_{a}=\sum_{j=1}^n \partial_{\xi_j} a\cdot \partial_{x_j}-\partial_{x_j} a\cdot \partial_{\xi_j}$. Hence, for every  symbol $p$
$$H_{a}p(x,\xi)=\{a,p\}(x,\xi),$$
where $\{\cdot,\cdot\}$ is the so-called \textit{Poisson bracket} defined as
$$\{a_1,a_2\}(x,\xi):= \nabla_\xi a_1(x,\xi)\cdot \nabla_x a_2(x,\xi)-\nabla_x a_1(x,\xi)\cdot \nabla_\xi a_2(x,\xi), \quad (x,\xi) \in \R^n \times \R^n.$$

\begin{definition}\label{biccurve}
    Let $a \in C^\infty(\R^n \times \R^n;\R)$ and $H_a$ be the Hamilton vector field of $a$. We call \textit{bicharacteristic curves} of $a$ the integral curves of $H_a$.
\end{definition}
\begin{definition}\label{prtype}
    Let $A\in \Psi_{\mathrm{cl}}^m(\R^n)$ be a classical pseudo-differential operator of order $m$ with real principal symbol $a_m$, and let $V=\xi \cdot \partial_\xi \in T(\R^n \times \R^n)$ be the radial vector field. The operator $A$ is said to be of \textit{real principal type} if the Hamilton vector field $H_{a_m}$ never vanishes on $\lbrace (x,\xi) \in \R^n \times (\R^n \setminus \lbrace 0 \rbrace); \ a_m(x,\xi)=0 \rbrace$ and does not have the radial direction there. In other words, there is no point $(x_0,\xi_0) \in \lbrace (x,\xi) \in \R^n \times (\R^n \setminus \lbrace 0 \rbrace); \ a_m(x,\xi)=0 \rbrace$ such that $H_{a_m}|_{(x_0,\xi_0)}=cV$, for some $c \in \R$.
\end{definition}
\section{Admissible operators}\label{sec.adm}
The purpose of this section is to investigate the so-called \textit{admissible operators} (see Definition \ref{def.admissible}), that is, operators having symbols to which Lemma \ref{newDoilem} below is applicable. 
As we will see in Section \ref{sec.smoothing.estimates}, in order to prove the smoothing effect for the evolution operator $D_t-A$, a fundamental role is played by the admissibility of the real part of the operator $A$.
This fact motivates our interest in these operators, especially in establishing sufficient conditions in order for an operator to be admissibile. The aforementioned sufficient conditions are suggested by the type of operators we are interested in, that is, operators having (possibly) a \textit{real valued symbol} $a \in S^m(\R^n)$ satisfying the following two conditions:
\begin{itemize}
\item [(i)] There exists  $C>0$ such that 
\begin{equation}\label{ellderivatives}
\frac{1}{C}\abs{\xi}^{m-1}\leq \abs{\nabla_\xi a}\leq C\abs{\xi}^{m-1}, \quad \forall (x,\xi)\in \R^{n}\times \R^n.
\end{equation}
\item[(ii)] For all $\alpha,\beta \in \mathbb{N}_0^n$, $\abs{\beta}\geq 1$, and for $\varepsilon>0$ sufficiently small
\begin{equation}\label{smallcoeffa2a3}
\abs{\partial_x^\beta \partial_\xi^\alpha a} \leq \varepsilon\lambda(\abs{x})\abs{\xi}^{m-\abs{\alpha}}, \quad \forall (x,\xi)\in \R^{n}\times \R^n,
\end{equation}
where $\lambda \in L^1([0,+\infty[)\cap C^0([0,+\infty[)$ is a strictly positive and non-increasing function.
\end{itemize}

When the symbol is complex-valued, then \eqref{ellderivatives} and \eqref{smallcoeffa2a3} are required on the real part, while the imaginary part will be assumed suitably small.

To make sense of \textit{admissible} operators, we first focus on Lemma \ref{newDoilem} in Subsection \ref{sub.fundlemma}. Then, in Subsection \ref{subs.adm}, we define the admissible operators and prove that \eqref{ellderivatives} and \eqref{smallcoeffa2a3} are sufficient conditions to have admissibility.

\subsection{A new version of Doi's lemma}
\label{sub.fundlemma}
As already mentioned, here
 we give a fundamental result to prove the smoothing effect of operators of the form \eqref{eqP} when the latter satisfy suitable assumptions, that is Lemma \ref{newDoilem}.  This will be a new version of the so-called Doi's lemma (Lemma 2.3 in \cite{D96}), targeted at operators of order $m\geq 2$ rather than of order $m=2$. Even if we will focus on the smoothing effect when $A$ in \eqref{eqP} is of order two or three, hence when $a\in S^m(\R^n), m=2,3$, the result of this section is given for  symbols $a\in S^m(\R^n)$ with $m\geq 2$.
 The very first version of Lemma \ref{newDoilem} for elliptic operators of order two is due to Doi in \cite{D96}, who used it to get the homogeneous smoothing estimates and the local well-posedness of the LIVP for variable coefficient Schr\"odinger operators. Doi's lemma was then used in \cite{KPRV05} to deal with the NLIVP, specifically to prove inhomogeneous smoothing estimates. In \cite{KPRV05} the authors also derived a version of Doi's lemma for second order ultrahyperbolic operators, which was used to get smoothing estimates for ultrahyperbolic Schr\"odinger operators.  Our version of Doi's lemma includes both the two aforementioned known cases.

\begin{lemma}\label{newDoilem}
Let $a \in S^m(\R^n)$, with $m\geq 2$, be real valued, and let $\lambda \in L^1([0,+\infty[)\cap C^0([0,+\infty[)$ be a positive non-increasing function. Assume also that there exists a function $q\in C^\infty(\R^n \times \R^n; \R)$ satisfying the following two conditions.

(i) For each $\alpha,\beta \in \mathbb{N}_0^n$,  there exists a constant $C_{\alpha,\beta}>0$ such that  
\begin{align}
    \abs{ \partial_x^\beta \partial_\xi^\alpha q(x,\xi)}\leq 
   \left\{ \begin{array}{ll}
    C_{\alpha\beta} \langle x\rangle \langle \xi \rangle^{-\abs{\alpha}},     & \text{if}\,\,|\beta|=0,  \\ 
     C_{\alpha\beta}  \langle \xi \rangle^{-\abs{\alpha}},    &  \text{if}\,\,|\beta|\geq 1 , 
    \end{array}\right.
\end{align}
for all $(x,\xi) \in \R^n \times \R^n$.

(ii) There exist two constants $C_1,C_2>0$  such that
\begin{align}
    H_{a}q(x,\xi)\geq C_1\abs{\xi}^{m-1}-C_2,
\end{align}
for all $(x,\xi) \in \R^n \times \R^n$.

Then, for all $C>0$ there exists a real valued symbol $p \in S^0(\R^n)$ and a constant $C'>0$ such that 
\begin{equation}\label{admp}
    H_{a}p(x,\xi)\geq C\lambda(\abs{x})\abs{\xi}^{m-1}-C',\quad \mathrm{for} \ \mathrm{all} \ (x,\xi)  \in \R^n \times \R^n .
\end{equation}
\end{lemma}
\proof
Let us first take $K\geq1$ such that $\abs{q(x,\xi)} \leq K \langle x \rangle$ for $(x,\xi) \in \R^n\times \R^n$, extend $\lambda(t)=\lambda(0)$ for $t \leq 0$, and define $\tilde{\lambda}(t):=\lambda(K^{-1}t-10)$ (note that $\tilde{\lambda}(t)\leq \lambda(t)$ for all $t\leq 0$). 

By Lemma 3.1 in \cite{D96} we can find a nonnegative function $f \in C^\infty([0,+\infty[)$ such that, for all $t \geq 0$, $\tilde{\lambda}(t)\leq f'(t)$  and
\begin{equation}\label{f}
\abs{f^{(m)}(t)}\leq C_m\Bigl(\tilde{\lambda}(0)+\int_0^t \tilde{\lambda}(s)ds\Bigr) \, (1+t)^{-m}\leq C_m\Bigl(\lambda(0)+\int_0^t \lambda(s)ds\Bigr) \, (1+t)^{-m}, 
\end{equation}
for some new suitable constant $C_m>0$.
Moreover, by the properties of $f$ we have 
\begin{equation}\label{f2}
f'(\abs{q}) \geq \lambda(K^{-1}\abs{q}-10) \geq \lambda (\langle x \rangle -10) \geq \lambda(\abs{x}), \quad \forall (x,\xi) \in \R^n\times \R^n.
\end{equation}

Now, we take $\varepsilon$ sufficiently small to be fixed later, and $\phi \in C^\infty(\R)$ such that $\phi(t)=0$, if $t\leq 1$, $\phi(t)=1$ if $t \geq 2$, and $\phi'(t) \geq 0$ on $\R$. We set $\phi_+(t):=\phi(t/\varepsilon)$, $\phi_-(t):=\phi(-t/\varepsilon)$ and $\phi_0:=1-\phi_+-\phi_-$, and define $\psi_+,\psi_-,\psi_0 \in S^0(\R^n)$ as
\[
\psi_0(x,\xi):=\phi_0(q(x,\xi)/\langle x \rangle), \quad \psi_+(x,\xi):=\phi_+(q(x,\xi)/\langle x \rangle) \quad \mathrm{and} \quad \psi_-(x,\xi):=\phi_-(q(x,\xi)/\langle x \rangle). 
\]
Since $\lambda \in L^1([0,+\infty[)$, by \eqref{f} and hypothesis $(i)$ on $q$, for some $C_{\alpha,\beta}>0$
\[
\abs{\partial_x^\beta \partial_\xi^{\alpha}f(\abs{q(x,\xi)})}\leq C_{\alpha,\beta}\langle \xi \rangle^{-\abs{\alpha}}
\quad \text{on} \,\,\, \mathrm{supp} \, \psi_+ \cup \mathrm{supp} \, \psi_-.\] 

Next, we define $p \in S^0(\R^n)$ as
\begin{equation}\label{Hp}
p(x,\xi):=\frac{q(x,\xi)}{\langle x \rangle}\psi_0(x,\xi)+(f(\abs{q(x,\xi)})+2\varepsilon)(\psi_+(x,\xi)-\psi_-(x,\xi)),
\end{equation}
and estimate $H_{a}p$ from below. 

First observe that, by hypotheses (i) and (ii) on $q$, and by choosing $\varepsilon>0$ sufficiently small, we have 
\begin{equation}\label{eq.2010}
H_{a}\Bigl(\frac{q}{\langle x \rangle}\Bigr)=\frac{H_{a}q}{\langle x \rangle}-\frac{q}{\langle x \rangle}\frac{x \cdot \nabla_\xi a}{\langle x \rangle^2}\geq C_1' \frac {|\xi|^{m-1}}{\langle x \rangle}-C_2' \quad \text{on}\,\, \mathrm{supp} \, \psi_0,
\end{equation}
 for some $C_1',C_2'>0$. Therefore, by the properties of $f$ and the inclusions  $\mathrm{supp} \, \phi'_+, \mathrm{supp} \, \phi'_- \subseteq \mathrm{supp}\,\phi_0$, we have
\begin{align}
H_{a}p&=H_{a}\Bigl(\frac{q}{\langle x \rangle}\Bigr)\psi_0+  f'(\abs{q})(H_{a}q)(\psi_+ +\psi_-)+\Bigl(f(\abs{q})+2\varepsilon-\frac{\abs{q}}{\langle x \rangle}\Bigr)\Bigl(\phi_+'(\frac{q}{\langle x \rangle})-\phi_-'(\frac{q
    }{\langle x \rangle})\Bigr)H_{a}\Bigl(\frac{q}{\langle x \rangle}\Bigr) \nonumber\\
    &\geq H_{a}\Bigl(\frac{q}{\langle x \rangle}\Bigr)\psi_0+  f'(\abs{q})(H_{a}q)(\psi_+ +\psi_-) -C_3,\label{Hp.1}
    \end{align}
for some $C_3>0$.

To estimate the first term on the RHS (right-hand side) of \eqref{Hp.1} we use again \eqref{eq.2010}, while for the second term on the RHS of \eqref{Hp.1} we use hypothesis (ii) and \eqref{f2}. 
 Putting everything together we obtain 
 \[
H_{a}p\geq C_4\Big(\langle x \rangle^{-1} \psi_0+\lambda(\abs{x})(\psi_++\psi_-)\Big)\abs{\xi}^{m-1}-C_5,
\]
for some new constants $C_4,C_5>0$.

 Finally, since $\lambda(t)=O(t^{-1})$ as $t \rightarrow +\infty$, we have $\langle x\rangle^{-1}\gtrsim \lambda(\abs{x})$, which yields 
\[
H_{a}p\geq C_6 \lambda(\abs{x})\abs{\xi}^{m-1}-C_5
\]
for some $C_6>0$. In conclusion, possibly rescaling the symbol $p$, we showed that for all $C>0$ there exists a constant $C'>0$ such that 
\[
H_{a}p\geq C\lambda(\abs{x})\abs{\xi}^{m-1}-C', \quad \forall x, \xi \in \R^n.
\] 
\endproof
\begin{remark}
    Note that when $m=2$ our conditions in Lemma \ref{newDoilem} are related with (A6) in \cite{D96}, but with a slightly different choice of $q$. As already mentioned, such a $q$, or more precisely the symbol $p \in S^0(\R^n)$ constructed through $q$ in Lemma \ref{newDoilem}, will play a central role to get the smoothing estimates in Section \ref{sec.smoothing.estimates}.
\end{remark}
\endproof

\textcolor{blue}{
\textbf{Fact:}
Our condition will be: there exists $C_1,C_2>0$ such that 
\begin{align}
    H_{a_3}q+H_{a_2}q+\re a_1\geq C_1|\xi|^2    -C_2,\quad \forall (x,\xi)\in \R^n\times\R^n.
\end{align}
assuming the coefficients of $a_2,a_1$ very small wrt the coefficients of $a_3$.}

\textcolor{red}{Let us assume now that there exists a function $q\in C^\infty(\R^n \times \R^n; \R)$ satisfying the following two conditions.}

\textcolor{red}{(i) For each $\alpha,\beta \in \mathbb{N}_0^n$ there exists a constant $C_{\alpha,\beta}>0$ such that  
\begin{align}
    \abs{\partial_\xi^\alpha \partial_x^\beta q(x,\xi)}\leq C_{\alpha\beta} \langle x \rangle \langle \xi \rangle^{-\abs{\alpha}},
\end{align}
for all $(x,\xi) \in \R^n \times \R^n$.}

\subsection{Admissible operators and sufficient conditions for admissibility}\label{subs.adm}
We are finally ready to give a characterization of the operators to which Lemma \ref{newDoilem} is applicable. 
According to the following definition, such operators will be called {\it admissible}. 

\begin{definition}\label{def.admissible}
    Let $a\in S_{}^m(\R^n)$ be a real valued symbol of order $m\geq 2$. We say that $a$ is {\it admissible} if there exists a real valued $q\in C^\infty(\R^n\times \R^n)$ such that

(i) For each $\alpha,\beta \in \mathbb{N}_0^n$,  there exists a constant $C_{\alpha,\beta}>0$ such that  
\begin{align}
    \abs{ \partial_x^\beta \partial_\xi^\alpha q(x,\xi)}\leq 
   \left\{ \begin{array}{ll}
    C_{\alpha\beta} \langle x\rangle \langle \xi \rangle^{-\abs{\alpha}},     & \text{if}\,\,|\beta|=0,  \\ 
     C_{\alpha\beta}  \langle \xi \rangle^{-\abs{\alpha}},    &  \text{if}\,\,|\beta|\geq 1 , 
    \end{array}\right.
\end{align}
for all $(x,\xi) \in \R^n \times \R^n$.

(ii) There exist $C_1,C_2>0$ such that
\begin{align}
    H_{a}q(x,\xi)\geq C_1\abs{\xi}^{m-1}-C_2,
\end{align}
for all $(x,\xi) \in \R^n \times \R^n$.

    We will call a function $q$ satisfying (i) and (ii) above a {\it Gårding weight} for $a$ and the constant $C_1$ the corresponding \textit{Gårding constant}.
    Moreover, given a pseudo-differential operator $A$, we say that $A$ is \textit{admissible} if its symbol $a$ is admissible.
\end{definition}
In Proposition \ref{propsuffadmiss} below we will prove that \eqref{ellderivatives} and \eqref{smallcoeffa2a3} are sufficient conditions for a pseudo-differential operator to be admissible.
  These conditions will allow us to consider different situations and to prove that several operators of interest are admissible. In particular, the operators we have in mind are those used to define dispersive equations. The reader can notice that typical constant coefficient operators $A$ on $\R_x^n$ of order $m=2,3,$ used to define dispersive equations $D_t-A$, fall into the cases covered by Proposition \ref{propsuffadmiss}.

\begin{remark}
Condition \eqref{smallcoeffa2a3} corresponds to a smallness assumption on the symbol, which, in the case of partial differential operators, amounts to a sort of smallness of the operator's coefficients, specifically of their derivatives. This requirement has to do with the behavior of the bicharacteristics of the principal symbol, as can be seen when $a$ is of order two and is either elliptic (see \cite{D96,D00,CKS}) or ultrahyperbolic (see \cite{KPRV05}). 
Of course, if the operator is a Fourier multiplier or, equivalently, its symbol is independent of the space variable $x$, then the smallness condition is trivially satisfied. 
\end{remark}
\begin{remark}
    We wish to point out that operators satisfying condition \eqref{ellderivatives} (in the classical case) are examples of operators of  \textit{real principal type} (see Definition \ref{prtype}). Indeed, when $A \in \Psi^m_{\mathrm{cl}}(\R^n)$ has a real principal symbol $a_m$ that satisfies \eqref{ellderivatives}, it is  easy to check that $H_{a_m}$ never vanishes on $\lbrace (x,\xi) \in \R^n \times (\R^n \setminus \lbrace 0 \rbrace); \ a_m(x,\xi)=0 \rbrace$ and does not have radial direction. 
    These operators play a key role in the study of propagation of singularities and in solvability problems (see, for instance, \cite{Le10}).
\end{remark}

\begin{remark}\label{rmk.admiss.const.coeff}
    If $A$ is an operator with constant coefficients, then \eqref{smallcoeffa2a3} is trivially satisfied, so \eqref{ellderivatives} guarantees the admissibility of $A$. Moreover, in such a case, Theorem \ref{thm.LIVP} holds if the symbol of $A$ is real and admissible. In turn, for operators with real symbol, the admissibility is a sufficient condition to have the smoothing effect.
\end{remark}

\begin{proposition}\label{propsuffadmiss}
If $a \in S^m(\R^n)$, $m\geq 2$, is a \textit{real valued} symbol such that \eqref{ellderivatives} and \eqref{smallcoeffa2a3} are satisfied, then $a$ is admissible, i.e. there exists a function $q\in C^\infty(\R^n \times \R^n; \R)$ such that the following two conditions hold:

(i) For each $\alpha,\beta \in \mathbb{N}_0^n$ there exists a constant $C_{\alpha,\beta}>0$ such that  
    \begin{align}\label{prop.q}
    \abs{ \partial_x^\beta \partial_\xi^\alpha q(x,\xi)}\leq 
   \left\{ \begin{array}{ll}
    C_{\alpha\beta} \langle x\rangle \langle \xi \rangle^{-\abs{\alpha}},     & \text{if}\,\,|\beta|=0,  \\ 
     C_{\alpha\beta}  \langle \xi \rangle^{-\abs{\alpha}},    &  \text{if}\,\,|\beta|\geq 1 , 
    \end{array}\right.
\end{align}
for all $(x,\xi) \in \R^n \times \R^n$.

(ii) There exist $C_1, C_2>0$ such that
\begin{align}
    H_{a}q(x,\xi)\geq C_1\abs{\xi}^{m-1}-C_2,
\end{align}
for all $(x,\xi) \in \R^n \times \R^n$.
\end{proposition}
\proof
To prove the result it suffices to consider $q \in C^\infty(\R^n \times \R^n)$ defined as 
\begin{equation}\label{explicitq}
q(x,\xi):=2C_1C^2\langle \xi \rangle ^{-(m-1)}\sum_{j=1}^n x_j\partial_{\xi_j}a(x,\xi),
\end{equation}
where $C$ is the constant in \eqref{ellderivatives}.
 Condition $(i)$ in the statement is trivially satisfied by $q$ as in \eqref{explicitq}, therefore we focus on showing that $(ii)$ holds too with the same function $q$.

By $\eqref{ellderivatives}$ and $\eqref{smallcoeffa2a3}$ 
\begin{align}
H_{a} q(x,\xi)&=\nabla_\xi a\cdot \nabla_x q  -  \nabla_x a\cdot\nabla_\xi q  \\
&=2C_1C^2\abs{\nabla_\xi a(x,\xi)}^2\langle \xi \rangle ^{-(m-1)}+2C_1C^2\langle \xi \rangle ^{-(m-1)}\sum_{j,k=1}^n x_j\partial_{\xi_k}a(x,\xi)\partial_{\xi_j}\partial_{x_k}a(x,\xi) - \nabla_x a\cdot\nabla_\xi q  \\
&\geq2C_1C^2\abs{\nabla_\xi a(x,\xi)}^{2}\langle \xi \rangle ^{-(m-1)} -2 C_1C^2\langle \xi \rangle ^{-(m-1)}\sum_{j,k=1}^n \abs{x_j\partial_{\xi_k}a(x,\xi)\partial_{\xi_j}\partial_{x_k}a(x,\xi)} - \abs{\nabla_x a}\abs{\nabla_\xi q} \\
& \geq 2C_1\abs{\xi}^{2(m-1)}\langle \xi \rangle ^{-(m-1)}  -\varepsilon C'\langle \xi \rangle^{m-1}\\
&\geq (2C_1-\varepsilon C'')\abs{\xi}^{m-1}-C_2, \label{ellderimplyadmiss}
\end{align}
for all $(x,\xi)\in \R^n\times \R^n$ and for some $C',C'',C_2>0$. Finally, for $\varepsilon$ sufficiently small
\[
H_{a} q(x,\xi) \geq C_1 \abs{\xi}^{m-1}-C_2\quad \forall(x,\xi)\in\R^{n}\times \R^n,
\]
which shows the result.
\endproof

\begin{remark}\label{rmk.garding.constant}
    Note that, if $q$ is a G\aa rding weight  for $a$ with corresponding G\aa rding constant $C_1>0$, then, for all $C_2>0$, $q'=\frac{C_2}{C_1}q$  is a G\aa rding weight  for $a$ with associated G\aa rding constant $C_2$.
\end{remark}

The following proposition provides a relation between $a$ and its Gårding weight $q$.

\begin{proposition}\label{prop.admi.princ.symb}
    Let $a\in S^m(\R^n)$ be a real valued symbol and $m\geq 2$. If $a$ is admissible with Gårding weight $q$, then
    \begin{align}
        \nabla_x a\cdot \nabla_\xi q\lesssim \langle \xi\rangle^{m-1}, \quad \forall(x,\xi)\in \R^n\times \R^n.
    \end{align}
\end{proposition}
\proof
The proof follows from the fact that 
\[
\nabla_xa \cdot \nabla_\xi q=\nabla_\xi a\cdot \nabla_x q -H_{a}q \leq C \langle \xi \rangle^{m-1}+D \leq C' \langle \xi \rangle^{m-1}, \quad \forall (x,\xi) \in \R^n\times \R^n,
\]
for some $C,D,C'>0$.
\endproof

   In the next result, we establish a connection between the admissibility of the leading part of the symbol and that of the whole symbol.

\begin{proposition}\label{prop.NC2}
    Let $a\in S_{}^m(\R^n)$ be a symbol of order $m\geq 2$ of the form $a=a_m+a_{m-1}$, with $a_j\in S^j(\R^n)$ real valued for $j=m-1,m$, and with $a_{m-1}$ satisfying
    \begin{align}
       |\nabla_x\partial^\beta_\xi a_{m-1}| \lesssim \langle x\rangle^{-1}\langle \xi\rangle^{m-1-|\beta|},\quad \forall (x,\xi)\in \R^n \times \R^n.
    \end{align}
    If $a$ is admissible with Gårding weight $q\in C^\infty(\R^n \times \R^n)$, then $a_m$ is admissible with respect to the same Gårding weight $q$.
\end{proposition}
\proof

Let $q$ be a Gårding weight for $a$ and $C$ a Gårding constant relative to $q$. 
Let also $D>0$ be the constant such that
$$H_{a}q\geq C|\xi|^{m-1}-D,\quad \forall (x,\xi)\in \R^n\times
 \R^n.$$
Due to the properties of the symbol $a_{m-1}$, there exist $C',D'>0$ such that $H_{a_{m-1}}q\lesssim \langle \xi\rangle^{m-2}\leq C' |\xi|^{m-2}+ D'$.
This property and the admissibility of $a$ finally give
\begin{align}
    H_{a_m}q=H_aq-H_{a_{m-1}}q\geq C|\xi|^{m-1}- C'|\xi|^{m-2}-(D+D')\geq C''|\xi|^{m-1}-D'',
\end{align}
for some $C'',D''>0$, with $C''<C$, which concludes the proof.

 \endproof
 \begin{remark}
    Proposition \ref{prop.NC2} shows that the admissibility of a real valued symbol is strictly related to the admissibility of its principal part. However, the latter is not a sufficient condition for the admissibility. 
\end{remark}

Next, we will show that if $a\in S^m_{}(\R^n)$ is such that $a=a_m+a_{m-1}$, with $a_j\in S^j(\R^n)$ real valued for all $j=m-1,m$, and if the symbol $a_m$ is admissible, then $a$ is also admissible provided that suitable smallness conditions are satisfied. This is proved in the following result.

\begin{proposition}\label{prop.admiss.a}
    Let $a\in S_{}^m(\R^n)$ be a symbol of order $m\geq 2$ of the form $a=a_m+a_{m-1}$, with $a_j\in S^j(\R^n)$ real valued for all $j=m-1,m$. If 
    \begin{itemize}
        \item[(i)] $a_m$ is admissible with Gårding weight $q\in C^\infty(\R^n \times \R^n)$;
        \item [(ii)] $a_{m-1}$ satisfies  $$|\nabla_x\partial^\beta_\xi a_{m-1}(x,\xi)|\lesssim \langle x\rangle^{-1} \langle \xi\rangle^{m-1-|\beta|}, \quad \forall (x,\xi)\in \R^n \times \R^n,$$
    \end{itemize}
    then $a$ is admissible with respect to the same Gårding weight $q$.
\end{proposition}   
\proof
Recall that, by definition, a Gårding weight $q$ satisfies the following:

\noindent for each $\alpha,\beta \in \mathbb{N}_0^n$ there exists a constant $C_{\alpha,\beta}>0$ such that  
    \begin{align}
    \abs{ \partial_x^\beta \partial_\xi^\alpha q(x,\xi)}\leq 
   \left\{ \begin{array}{ll}
    C_{\alpha\beta} \langle x\rangle \langle \xi \rangle^{-\abs{\alpha}},     & \text{if}\,\,|\beta|=0,  \\ 
     C_{\alpha\beta}  \langle \xi \rangle^{-\abs{\alpha}},    &  \text{if}\,\,|\beta|\geq 1 , 
    \end{array}\right.
\end{align}
for all $(x,\xi) \in \R^n \times \R^n$.
Hence, by hypothesis (ii),
\begin{align}
    |H_{a_{m-1}}q|&=|\nabla_\xi a_{m-1}\cdot\nabla_xq-\nabla_x a_{m-1}\cdot\nabla_\xi q|\\
    &\leq C_2\langle \xi\rangle^{m-2},\quad \forall (x,\xi)\in \R^n \times \R^n,
\end{align}
for some $C_2>0$.
Then, since $a_m$ is admissible by (i),  we have
\begin{align}
    H_aq=H_{a_m}q+H_{a_{m-1}}q\geq C|\xi|^{m-1}-C_2 |\xi|^{m-2}-D-D'\geq C''|\xi|^{m-1}-D'', \quad \forall (x,\xi)\in \R^n \times \R^n,
\end{align}
for some $C',D',C'',D''>0$ ($C''< C$), which shows the admissibility of $a$ and concludes the proof.
\endproof

\begin{remark}
    Proposition \ref{prop.admiss.a} shows that \textit{small} perturbations of admissibile symbols are still admissible. Here \textit{small} refers to the asymptotic smallness in space of lower order perturbations. Of course, one can also consider sufficiently small perturbations of order $m$.
\end{remark}
\subsection{Examples of admissible operators and dispersive operators with variable coefficients}\label{subs.example}
In this section, we provide several examples of operators satisfying \eqref{ellderivatives} and \eqref{smallcoeffa2a3}, which are, in particular, admissible operators and are used to define dispersive equations. Later, in Section \ref{sec.smoothing.estimates}, specifically in Theorem \ref{thm.LIVP}, we will see that the admissibility of the real part of an operator $A$ on $\R^n_x$ is determinant to prove the smoothing effect of the corresponding dispersive operator $D_t-A$ on $\R_t\times\R^n_x$ whenever $A$ is or order $m=2,3$. 
\begin{example}[Operators defining constant coefficient dispersive equations] The easiest example of admissible operators on $\R^n$ is given by those of the form $A=Op^\mathrm{w}(a)$ with $a(\xi)=|\xi|^{2m}$, $m\geq 1$.  They trivially satisfy both \eqref{ellderivatives} and \eqref{smallcoeffa2a3}, therefore they are admissible by Proposition \ref{propsuffadmiss}. It is well known that dispersive operators $D_t-A$ enjoy smoothing estimates. In particular, when $m=1$, thanks to admissibility, we reobtain the classical smoothing effect of $D_t-A$  by Theorem \ref{thm.LIVP} (see also Remark \ref{rmk.admiss.const.coeff} and Remark \ref{rmk.conf.thm.LIN}).
  \end{example}  

\begin{example}[Operators defining constant coefficients KdV-type operators]\label{ex.1}
    \label{toymodels}
    Other examples of admissible operators with constant coefficients on $\R^n$ are given by
    \begin{align}
        A_1(D)&=D^3,\quad \text{when}\,\, n=1,\\
    A_2(D)&=D_1(D_1^2+D_2^2),\quad \text{when}\,\, n=2,\\
        A_n(D)&=D_1\sum_{j=1}^n D_j^2,\quad\text{when}\,\, n\geq 3.
    \end{align}
    Once again, the admissibility of $A_1,A_2$ is due to the validity of \eqref{ellderivatives} and \eqref{smallcoeffa2a3}.
    These operators define the (linear) KdV operator $D_t-A_1(D_x)$ on $\R_t\times \R_x$, and the ZK operator $D_t-A_2(D_x)$ on $\R_t\times \R_x^2$.
    Since $A_1$ and $A_2$ are admissible by Proposition \ref{propsuffadmiss}, we get the classical smoothing estimates for $D_t-A_j(D_x), j=1,2,$ by Theorem \ref{thm.LIVP} below (see also Remark \ref{rmk.admiss.const.coeff} and Remark \ref{rmk.conf.thm.LIN}).  For $A_n$, which can be viewed as the higher-dimensional generalization of $A_1$ and $A_2$, the same considerations lead to admissibility and the presence of the smoothing effect.
    \end{example}

    \begin{example}  [Operators defining variable-coefficients KdV-type operators] 
    More generally,  Proposition \ref{propsuffadmiss} successfully applies to show the admissibility of operators on $\R^{n}$ (specifically to the real part of the symbol) of order three of the form of $A_3$ below, that are used to define the KdV-type operators on $\R^{n+1}$ (see \cite{F24})
    \begin{equation} \label{kdv-type}  
A_3+X_0=X_1\sum_{j=1}^nX_j^\ast X_j+X_0,
    \end{equation}
    where $iX_k$, $k=0,...,n$, form an elliptic system of smooth real vector fields satisfying some \textit{smallness} and ellipticity conditions.  Note that these operators can have variable coefficients, so the smoothing effect of the possibly dispersive operator $A_3+X_0$ (consider $X_0=D_t$) is not guaranteed by classical results in the constant coefficients case. Consequently, to be able to say something about the smoothing effect, if we assume \eqref{smallcoeffa2a3} for $A_3=Op^\w(a_3+a_2)$, by Theorem \ref{thm.LIVP} it suffices to have the admissibility of $\re(a_3+a_2)$ and the smallness of $\im(a_2)$, which are provided by the following proposition.
\end{example}

    \begin{proposition}\label{prop.kdv-type}
        Let $A(x,D)$ a partial differential operator of the same form as $A_3$ in \eqref{kdv-type} defined on $\R^n$, with $n>1$, such that
    $$iX_j=i\sum_{i=1}^na_{ij}(x)D_i, \quad j=1,\ldots,n,$$  form an elliptic system of smooth real vector fields. Let also $a=a_3+a_2\in S^3(\R^n \times \R^n)$ be the Weyl symbol of $A$, where $a_j\in S^j(\R^n \times \R^n)$, $j=2,3$, and $a_3$ is real valued. If
    \begin{align}
        a_{ij}\in C^\infty_b(\R^n)\,\, \text{and}\,\,|\partial_x^{\alpha} a_{ij}(x)|\leq \frac{\varepsilon}{\langle x\rangle^{1+|\alpha|}},\quad \forall x\in \R^n, \forall i,j=1,\ldots,n, \,\, |\alpha|\geq 1,\label{cond.kdv.small.}
    \end{align}
    for $\varepsilon$ sufficiently small, and there exist $k\in\{1,\ldots,n\}$ and $C>0$ such that
    \begin{align}\label{cond.Xk}
\abs{\partial_{\xi_k}X_1(x,\xi)}\geq C, \quad \forall (x,\xi)\in\R^n \times \R^n,
    \end{align}
    and
    \begin{align}\label{cond.akj}
        \sup_{x\in\R^n}|a_{kj}(x)|^2< 3C^2/(n-1), \forall j=2,\ldots,n,
    \end{align}
    then $\re (a)$ is admissible. Moreover, for $\varepsilon$ sufficiently small, there exists a positive constant $c_0$ such that
    \begin{align}\label{cond.kdv.im.prop}
        |\im(a_2)(x,\xi)|\leq c_0\langle x\rangle^{-2}\langle \xi\rangle ^2, \quad \forall (x,\xi)\in \R^n \times \R^n.
    \end{align} 
    \end{proposition}
\proof
The proof is based on the use of Proposition \ref{propsuffadmiss} and Proposition \ref{prop.admiss.a}. We will first prove that the principal symbol $a_3$ of $A$ is admissible. Afterwards, we will show that $\re(a)$, viewed as a perturbation  of $a_3$, is admissible too, and that \eqref{cond.kdv.im.prop} holds.

Observe that, due to \eqref{cond.kdv.small.}, we have that \eqref{smallcoeffa2a3} is trivially satisfied by $a_3$. Hence, if we prove that $|\nabla_\xi a_3(x,\xi)|\gtrsim |\xi|^2$ for all $(x,\xi)\in\R^n \times \R^n$, which amounts to \eqref{ellderivatives}, then by Proposition \ref{propsuffadmiss} we can conclude the admissibility of $a_3$ as desired. Because of that, we now focus on showing the above property of $\nabla_\xi a_3$.

We denote by $X_j(x,\xi)$ and $X_j^\w(x,\xi)$ the Kohn-Nirenberg and the Weyl symbol of the homogeneous first order operator $X_j(x,D)$, respectively. In particular, by the change of quantization formula (see, for instance, \cite{Fo89} Theorem 2.4.1. and \cite{FRO} Theorem 4.12) we have
\begin{align}
    X_j(x,\xi)=\sum_{i=1}^na_{ij}(x)\xi_i\quad \text{and}\quad X_j^\w(x,\xi)=X_j(x,\xi)+d_{X_j}, \quad \text{where}\quad d_{X_j}:=-\frac 1 2\sum_{i=1}^nD_ia_{ij}(x),
\end{align}
therefore the principal symbol $a_3$ of $A$, which is the same in both quantizations, is given by
\begin{align}
a_3(x,\xi)=X_1(x,\xi)\sum_{j=1}^nX_j(x,\xi)^2.
\end{align}
\indent Now, by using \eqref{cond.Xk} and the smallness hypotheses on the coefficients of the vector fields $X_j$, and choosing $\delta>0$ such that $(n-1)/(3C)<\delta<C /\sup_x\{|a_{kj}(x)|^2, j=2,\ldots ,n\}$ (which is possible due to \eqref{cond.akj}), we obtain
\begin{align}
    |\nabla_\xi a_3(x,\xi)|&\geq |\partial_{\xi_k}a_3|\\
    &=\abs{3\partial_{\xi_k}X_1(x,\xi)|X_1(x,\xi)|^2 +\partial_{\xi_k}X_1(x,\xi)\sum_{j=2}^nX_j(x,\xi)^2+2\sum_{j=2}^nX_1(x,\xi)\mathrm{Re}[(\partial_{\xi_k}X_j)(x,\xi)\overline{X}_j(x,\xi)]}\\
    &\geq \Bigl(3C-\frac{n-1}{\delta}\Bigr)X_1(x,\xi)^2+ \sum_{j=2}^n(C-\delta|a_{kj}(x)|^2)X_j(x,\xi)^2\\
    & \gtrsim \sum_{j=1}^nX_j(x,\xi)^2\gtrsim |\xi|^2,\quad \forall(x,\xi)\in \R^n \times \R^n,\label{eq.grad}
\end{align}
which gives \eqref{ellderivatives} and hence the admissibility of $a_3$. Note that when $n=1$ we have $a_3(x,\xi)=a(x)\xi^3$ and $\nabla_\xi a_3(x,\xi)=3a(x)\xi^2$. Hence,  \eqref{ellderivatives} is trivially satisfied without the need of any condition of type \eqref{cond.akj} (which we do not assume in this case), so $a_3$ is admissible.

Next, to prove the admissibility of $\re(a)$, it suffices to check the behavior of $\mathrm{Re}(a_2)$, where $a_2\in S^2(\R^n)$ is the symbol of order two satisfying
$$a=a_3+a_2,$$
where $a$ is the Weyl symbol of $A$.

Below, we shall write $X_j^\w, X_j$, in place of $X_j^\w(x,\xi), X_j(x,\xi)$, when no confusion arises.
Then, dropping the dependence on $(x,\xi)$, the Weyl composition formula gives that
\begin{align}
   a&= \sum_{j=1}^nX_1^\w\# \overline{
   X_j^\w}\#X_j^\w\\
   &=X_1^\w\#\sum_{j=1}^n(|{X_j^\w}|^2-\frac{i}{2}\{\overline{X_j^\w},X_j^\w\}+r_0)\\
   &=X_1^\w\#\left(\sum_{j=1}(X_j^2-d_{X_j}^2)+r_0'\right)\\
   &= \sum_{j=1}\left((X_1+d_{X_1})\#X_j^2- (X_1+d_{X_1})\#d_{X_j}^2\right)+(X_1+d_{X_1})\#r_0'\\
   &= a_3-\frac{i}{2}\sum_{j=1}^n(\{X_1,X_j^2\}+2id_{X_1}X_j^2)+r_1+r_0\\
   &=a_3+a_2,
\end{align}
  where $r_1\in S^1_{1,0}(\R^n)$ is real valued, and $r_0\in S^0_{1,0}(\R^n)$ is  a symbol of order 0 possibly different at any appearance. Note that both $r_1$ and $r_0$ are symbols of differential operators having coefficients depending on the derivatives of the vector field's coefficients, that is, depending on $\partial^\alpha_xa_{ij}, i,j=1,\ldots,n,$ and $\alpha\in \mathbb{N}^n$. 

Now, since $d_{X_j}$, for all $j=1,\ldots,n$, is either 0 or purely imaginary, using \eqref{cond.kdv.small.} we get, for all $(x,\xi)\in\R^n \times \R^n$,
\begin{align}
    \abs{\nabla_x\partial_\xi^\beta \re( a_2)}=\abs{\nabla_x\partial_\xi^\beta r_1}\lesssim \langle x\rangle^{-1}\langle \xi\rangle^{1-|\beta|},\quad \forall \beta \in \mathbb{N}^n, \label{eq.a2}
\end{align}
and 
\begin{align}
    \abs{\mathrm{Im}(a_2)}=\left|-\frac{1}{2}\sum_{j=1}^n\Big(\{X_1,X_j^2\}-2\im (d_{X_1})X_j^2\Big)+\im(r_0)\right|.\label{eq.ima2.}
\end{align}
We use \eqref{eq.a2} and \eqref{eq.ima2.} to prove the admissibility of $\re (a)$ and \eqref{cond.kdv.im.prop}, respectively (recall that $r_0$ is a symbol of a differential operator  having coefficients depending on the derivatives $\partial^\alpha_xa_{ij}$).

Inequality \eqref{eq.a2} together with \eqref{eq.grad} give the admissibility of $\re(a)$ by Proposition \ref{prop.admiss.a}.

Identity \eqref{eq.ima2.} and the hypotheses satisfied by $a_{ij}$ give \eqref{cond.kdv.im.prop}.
To see this, it suffices to expand the Poisson brackets $\{X_1,X_j\}$ in \eqref{eq.ima2.} and note that $\im(a_2)$ is a combination of symbols  whose coefficients are space-derivatives of the vector field's coefficients $a_{ij}$.
As for the other terms in $\im(a_2)$, they also include derivatives of the $a_{ij}$'s, since $\im (d_{X_j})=\frac 1 2\sum_{i=1}^n\partial_i a_{ij}(x)$, for all $j=1,\ldots,n$. Therefore, for $\varepsilon$ as in \eqref{cond.kdv.small.} sufficiently small, we reach
\begin{align}
    \abs{\mathrm{Im}(a_2)}\leq c_0 \langle x\rangle ^{-2}\langle \xi\rangle ^2,\quad \forall(x,\xi)\in \R^n \times \R^n,\label{eq.im}
    \end{align} 
for some $c_0>0$, which proves \eqref{cond.kdv.im.prop} and concludes the proof. 
\endproof
    
\begin{remark}
   Let us stress that \eqref{cond.kdv.im.prop} in Proposition \ref{prop.kdv-type} is not necessary for the admissibility of $\re(a)$, but is required to obtain the smoothing estimates in Theorem \ref{thm.LIVP}.
\end{remark}

\begin{example}[Operators defining ultrahyperbolic Schr\"odinger equations]
    Other interesting examples of admissible operators are given by ultrahyperbolic differential operators of order two, which were used in \cite{KPRV05} to define ultrahyperbolic Schr\"odinger operators. A toy model of such operators is $\Delta_1-\Delta_2$, where, for $k\leq n$, $\Delta_1$ is the (positive) Laplacian in the first group of $k$ variables and $\Delta_2$ is the (positive) Laplacian in the second group of $n-k$ variables, that is the operator whose principal part's coefficients are determined by the matrix
    \[
    \begin{pmatrix}
        \mathrm{I_k} & 0 \\
        0 & -\mathrm{I_{n-k}}
    \end{pmatrix}.
    \]
    More generally, ultrahyperbolic differential operators of order two with variable coefficients are defined as 
    \[
    A=\sum_{i,j=1}^n -\partial_{x_i} a_{ij}(x)\partial_{x_j},
    \]
    where the coefficient matrix of the principal part $\mathsf{A}_2(x)=(a_{ij}(x))_{i,j=1, \dots, n}$, for all $x \in \R^n$, is real, symmetric and satisfies the non-degeneracy condition
    \begin{equation}\label{nondegeneracy}
    \exists \nu>0\quad \text{s.t.}\quad \nu^{-1}\abs{\xi} \leq \abs{\mathsf{A}_2(x) \cdot \xi}\leq \nu \abs{\xi}, \quad \forall x,\xi \in \mathbb{R}^n.
    \end{equation}
     Note that the Weyl symbol $a$ of $A$, which belongs to $S^2(\R^n)$, has the form 
\[
a(x,\xi)=a_2(x,\xi)+a_0(x,\xi)=\sum_{i,j=1}^n a_{ij}(x) \xi_i \xi_j+\frac{1}{4}\sum_{\ell,k=1}^n\partial_{x_\ell}\partial_{x_k}a_{\ell k}(x),
\]
    where $a_j\in S^j(\R^n)$ for $j=0,2$. In addition, we assume the following smallness condition: there exists $\varepsilon>0$ sufficiently small such that, for all $i,j=1,\dots,n$, for all $\beta \in \bN_0^n$, $|\beta|\geq 1$,
\begin{equation}\label{small.ultrahyp}
\abs{\partial_{x}^\beta a_{ij}(x)} \leq C_{\beta} \varepsilon \langle x \rangle^{-2}, \quad \forall x \in \R^n.
\end{equation}
Since $\partial_{\xi_j}a_2(x,\xi)=2 \sum_{i=1}^na_{ij}(x)\xi_i=2(\mathsf{A}_2(x) \cdot \xi)_j$, for all $j=1,\dots, n$, by \eqref{nondegeneracy} we have 
\[
2\nu^{-1} \abs{\xi}\leq \abs{\nabla a_2(x,\xi)} \leq 2 \nu \abs{\xi}, \quad \forall \xi \in \mathbb{R}^n,
\]
which, in turn, shows that $a_2$ is admissible by Proposition \ref{propsuffadmiss}. Finally, since $a_0$  satisfies the smallness condition above, we get that $a$, viewed as a perturbation of $a_2$, is admissible by Proposition \ref{prop.admiss.a}.

By the results in Section \ref{sec.smoothing.estimates}, the admissibility of $A$ and the smallness assumption $\eqref{small.ultrahyp}$ imply the smoothing estimates for the ultrahyperbolic Schr\"odinger operator $D_t-A$ (see Remark \ref{rmk.conf.thm.LIN}). Specifically, we retrieve the results established in \cite{KPRV05} simply by verifying the admissibility of the ultrahyperbolic operator $A$, which is guaranteed by the smallness assumptions. 
\end{example}
\begin{example}[Operators defining  other constant coefficient linear KdV-type equations]\label{Kdvtype2}
Let now $0 \neq b \in \R^n$ and consider on $\R^n$ the operator
\[
A=\sum_{\ell=1}^n b_\ell D_{x_\ell} \Delta
\]
with symbol
\[
a(x,\xi)=a(\xi)=\sum_{\ell=1}^n b_\ell\xi_\ell \abs{\xi}^2=(b\cdot \xi)|\xi|^2, \qquad \xi \in \R^n.
\] 
In this case, we have 
\begin{equation}\label{eq.ia1}
\nabla_\xi a=b|\xi|^2+2(b\cdot \xi)\xi,
\end{equation}
and, consequently, 
\begin{equation}\label{eq.ia2}
b \cdot \nabla_\xi a=|b|^2|\xi|^2+2(b\cdot \xi)^2.
\end{equation}
Therefore, by \eqref{eq.ia1} we have 
\[
|\nabla_\xi a(\xi)|\leq 3 |b||\xi|^2, \quad \forall \xi \in \R^n,
\]
and by \eqref{eq.ia2}
\[
|b|^2|\xi|^2\leq b \cdot \nabla_\xi a(\xi) \leq |b||\nabla_\xi a(\xi)|, \quad \forall \xi \in \R^n,
\]
which leads to 
\[
|\nabla_\xi a(\xi)| \geq |b||\xi|^2, \quad \forall \xi \in \R^n.
\]
This proves \eqref{ellderivatives} and the admissibility of $A$ by Proposition \ref{propsuffadmiss} (note also that the operators in question are operators with constant coefficients).

Observe that the operator $D_t-A$ describes a (linear) KdV-type operator and that the classical smoothing estimates for $D_t-A$ (see, for instance, \cite{CS89}, \cite{RS12} or its generalization \cite{FR24}, and references therein) can be recovered by using the admissibility of $A$ and
Theorem \ref{thm.LIVP} (see Remark \ref{rmk.admiss.const.coeff} and Remark \ref{rmk.conf.thm.LIN}).
Finally, note that with a suitable choice of $b \in \R^n$, we obtain
\[
A=\sum_{\ell=1}^n D_{x_\ell}\Delta,
\]
or, more generally, 
\[
A=\sum_{\ell=1}^{n-k}D_{x_\ell}\Delta,
\]
with $0 \leq k \leq n-1$.
\end{example}

\section{Smoothing estimates}\label{sec.smoothing.estimates}
In this section we prove the local smoothing effect for evolution equations with possibly variable coefficients defined through a class of operators of order $m=2,3,$ having admissible real part. These operators include, for instance, those discussed in Subsection \ref{subs.example}. 
More precisely, we focus on the study of the IVP
\begin{equation}\label{LIVP}
\begin{cases}
\partial_t u -iAu=f \quad \quad \mathrm{on} \quad \R^n \times ]0,T[, \quad T>0, \\
u(0,\cdot)=u_0 \quad \quad \mathrm{on} \quad \R^n,
\end{cases}
\end{equation}
 where $u_0 \in H^s(\R^n)$, $f$ has suitable regularity (see Theorem \ref{thm.LIVP}), and $A=Op^\w(a)\in \Psi^m(\R^n)$ is a pseudo-differential operator of order $m=2,3,$ satisfying the following conditions:
\begin{itemize}
    \item The Weyl symbol $a \in S^m(\mathbb{R}^n)$ of $A$ is of the form 
\begin{align}
    a (x,\xi)= a_m(x,\xi)+a_{m-1}(x,\xi),\label{cond.a}
\end{align}
where $a_j \in S^j(\R^n)$ for all $j=m-1, m$, and $a_m$ is real valued; 
\item The real part $\re(a) \in S^m(\R^n)$
satisfies \eqref{ellderivatives} and \eqref{smallcoeffa2a3};
\item The imaginary part $\im(a)= \im (a_{m-1}) \in S^{m-1}(\R^n)$ satisfies 
\begin{align}  
|\im(a_{m-1})(x,\xi)|\leq c_0\lambda(\abs{x})\langle\xi\rangle^{m-1},\quad \forall (x,\xi)\in \R^n \times \R^n, \label{cond.ima}
\end{align}
for some $c_0>0$ and for $\lambda(\abs{x})$ as in \eqref{smallcoeffa2a3}.
   \end{itemize}

Note that conditions \eqref{ellderivatives} and \eqref{smallcoeffa2a3} ensure that $\re(a)=a_m+\re(a_{m-1})$ is admissible with Gårding weight $q$ and Gårding constant $C$, for some $C>0$ and some $q\in  C^\infty(\R^n \times \R^n)$ satisfying, for all $\alpha,\beta\in\mathbb{N}_0^n$, 
   \begin{align}\label{qSG}
    \abs{ \partial_x^\beta \partial_\xi^\alpha q(x,\xi)}\leq 
   \left\{ \begin{array}{ll}
    C_{\alpha\beta} \langle x\rangle \langle \xi \rangle^{-\abs{\alpha}},     & \text{if}\,\,|\beta|=0,  \\ 
     C_{\alpha\beta}  \langle \xi \rangle^{-\abs{\alpha}},    &  \text{if}\,\,|\beta|\geq 1 , 
    \end{array}\right.
\end{align}
for every $(x,\xi)\in\R^n\times\R^n$.

We emphasize that we are not assuming $a_m, a_{m-1}$ to be homogeneous symbols here. The key feature of the decomposition of $a$ is that the pseudo-differential operator of order $m$ in $A$, that is $Op^\w(a_m)$, has a\textit{real valued} admissible symbol.

We also stress that, although the case $m=2$ is already known in the literature (see \cite{CKS, D96,D00, KPRV05}),  we  provide a new method of proof in this case.  Our method relies on considering operators for which the real part of symbol, $\re(a)$, satisfies \eqref{ellderivatives} and \eqref{smallcoeffa2a3} (see Remark \ref{rmk.m=2}) and such that \eqref{cond.a} and \eqref{cond.ima} hold. In other words, some of these conditions replace the nontrapping condition in \cite{CKS, D96,D00, KPRV05}) (smallness conditions can be found in these works as well). Our method applies successfully also to the case $m=3$.

One may wonder why we did not consider the general case $m\geq 2$ here. The reason is that there is a non-trivial difference between the cases $m=2,3,$ and $m>3$. This distinction is due to the fact that when $m>3$ we are only able to prove smoothing estimates with a big "error term", which means that the estimates include an additional unpleasant term with respect to the standard ones. In contrast, when $m=2,3$, we arrive at the expected standard smoothing estimates. 
A brief discussion of this topic can be found in Remark \ref{rmk.Gard-FP} below.

To prove Theorem \ref{thm.LIVP} on the smoothing effect of $D_t-A$, with $A$ satisfying the aforementioned properties, we need Lemma \ref{lem.smoothing} which provides (i)-(iii) of Theorem \ref{thm.LIVP}.

\begin{lemma}\label{lem.smoothing}
    Let $s\in\R$, $m=2,3$, and let $\lambda$ and $A$ be as in Theorem \ref{thm.LIVP}. Then, there exists a positive constant $C_1=C_1(A,n)$ such that, for all $u\in C([0,T];H^{s+m}(\R^n))\cap C^1([0,T];H^{s}(\R^n))  $,
    \begin{itemize}
    \item [(i)] 
    \begin{align}
    \sup_{0\leq t\leq T}\|u\|_s\lesssim  e^{C_1T} \left(\|u(0,\cdot)\|_s+\int_0^T \|(\partial_t-iA)u\|_s dt\right),
\end{align}
\item [(i')] 
    \begin{align}
    \sup_{0\leq t\leq T}\|u\|_s\lesssim  e^{C_1T} \left(\|u(T,\cdot)\|_s+\int_0^T \|(\partial_t-iA)^*u\|_s dt\right),
\end{align}
\item [(ii)] 
\begin{align}
    \sup_{0\leq t \leq T}\|u\|^2_s+\int_0^T(\lambda(|x|) \Lambda^{s+\frac{m-1}{2}}u(t,\cdot),\Lambda^{s+\frac{m-1}{2}}u(t,\cdot))_0\,dt\lesssim e^{C_1T} \left(\|u(0,\cdot)\|_s^2+\int_0^T \|(\partial_t-iA)u\|_s^2 dt \right),
\end{align}
\item [(ii')] 
\begin{align}
    &\sup_{0\leq t \leq T}\|u\|^2_s+\int_0^T(\lambda(|x|) \Lambda^{s+\frac{m-1}{2}}u(t,\cdot),\Lambda^{s+\frac{m-1}{2}}u(t,\cdot))_0\,dt\lesssim e^{C_1T} \left(\|u(T,\cdot)\|_s^2+\int_0^T \|(\partial_t-iA)^*u\|_s^2 dt\right),
\end{align}
\item [(iii)]  
\begin{align}
    &\sup_{0 \leq t \leq T}\|u\|^2_s+\int_0^T (\lambda(|x|)\Lambda^{s+\frac{m-1}{2}}u(t,\cdot),\Lambda^{s+\frac{m-1}{2}}u(t,\cdot))_0\,dt\\
    &\lesssim e^{C_1T} \left(\|u(0,\cdot)\|_s^2+\int_0^T \big(\lambda(|x|)^{-1} \Lambda^{s-\frac{m-1}{2}}(\partial_t-iA)u,\Lambda^{s-\frac{m-1}{2}}(\partial_t-iA)u\big)_0 dt\right),
\end{align}
\item [(iii')]  
\begin{align}
    &\sup_{0\leq t\leq T}\|u\|^2_s+\int_0^T(\lambda(|x|)  \Lambda^{s+\frac{m-1}{2}}u(t,\cdot),\Lambda^{s+\frac{m-1}{2}}u(t,\cdot))_0\,dt\\
    &\lesssim e^{C_1T} \left(\|u(T,\cdot)\|_s^2+\int_0^T \big(\lambda(|x|)^{-1}\Lambda^{s-\frac{m-1}{2}}(\partial_t-iA)^*u,\Lambda^{s-\frac{m-1}{2}}(\partial_t-iA)^*u\big)_0dt\right).
\end{align}
\end{itemize}

\end{lemma}

\begin{proof}

Here we focus on proving (i), (ii), and (iii) only, since (i'), (ii'), and (iii') can be derived analogously (see the end of the proof for more details about these cases).

We define, for $s \in \R$, the operators 
\[
\Lambda^s:=Op^{\w}(\langle \xi\rangle^s) \in \Psi^s(\R^n) \quad \mathrm{and} \quad E:=Op^{\w}(e^p) \in \Psi^0(\R^n), 
\]
where $p \in S^0(\R^n)$ is  such that 
$$ H_{\re(a)}p(x,\xi)\geq \lambda(|x|)|\xi|^{m-1}-D,\quad \forall(x,\xi)\in \R^n\times \R^n,  $$
for some $D>0$. Note that such a $p$ exists due to the admissibility of $\re(a)$ and Lemma \ref{newDoilem}.
We also introduce the quantity
\begin{align}
    N(u):=(\|E\Lambda^s u\|^2_0+\|u\|_{s-2}^2)^{1/2},
\end{align}
which defines a norm equivalent to the standard $H^s$-Sobolev norm. To see this it suffices to observe that 
\begin{align}
 \widetilde{E} E =I+R_{-2}, \label{equiv.norm2}  
\end{align}
where $R_{-2}\in \Psi^{-2}(\R^n)$ and 
\begin{align}
 \widetilde{E}:=Op^\w(e^{-p}) \in \Psi^{0}(\R^n).\label{equiv.norm3}
\end{align}
By \eqref{equiv.norm2}, \eqref{equiv.norm3}, and the boundedness properties of pseudo-differential operators, we have
\begin{align}
\|v\|^2_s&=\|\Lambda^sv\|_0^2\leq \|\widetilde{E} E\Lambda^s v-R_{-2}\Lambda^s v\|^2_0+\|v\|^2_{s-2}\\
&\leq C(\| E\Lambda^s v\|_0^2+\|v\|_{s-2}^2 )\leq  C'\|v\|_{s}^2,
\end{align}
which shows the equivalence of $N(\cdot)$ and $\|\cdot\|_s$.

Now, we set $f:=(\partial_t-iA)u$ and estimate $\partial_t N(u)^2= \partial_t \, \norm{E\Lambda^s u}^2_0+\partial_t \, \norm{ u}^2_{s-2}$. To begin with, we calculate $\partial_t \, \norm{ u}^2_{s-2}$. Using the equivalence of $N(\cdot)$ and $\|\cdot\|_s$ and the boundedness property of pseudo-differential operators, we obtain
\begin{align}
    \partial_t\norm{u}_{s-2}^2&=2\re(\Lambda^{s-2}\partial_tu,\Lambda^{s-2}u)_0\\
    &=2\re(\Lambda^{s-2}iAu, \Lambda^{s-2}u)+2\re(\Lambda^{s-2}f,\Lambda^{s-2}u)_0\\
    &\leq C_0 N(u)^2+2\re(\Lambda^{s-2}f,\Lambda^{s-2}u)_0, \label{estim.norm1}
\end{align}
where $C_0$ is a suitable positive constant.

Next, we compute $\partial_t\norm{E\Lambda^su}_0^2$ and get
\begin{align}
\partial_t \, \norm{E\Lambda^s u}^2_0 &=\partial_t (E\Lambda^s u, E\Lambda^s u)_0 \\
&=2 \, \re(E\Lambda^siAu,E\Lambda^su)_0+2\re(E\Lambda^s f,E\Lambda^s u )_0 \\
&=2 \, \re (iAE\Lambda^su,E\Lambda^su)_0 +2 \, \re([E\Lambda^s,iA]u,E\Lambda^su)_0+2\re(E\Lambda^s f,E\Lambda^s u )_0 \\
&=(I)+(II)+2\re(E\Lambda^s f,E\Lambda^s u )_0. \label{eqn19033-2}
\end{align}

For (I) we have
\[
\begin{split}
(I)&=2 \, \re (Op^\w(ia)E\Lambda^su,E\Lambda^su)_0 \\
&=2 \, \re (Op^\w(ia_m+ia_{m-1})E\Lambda^su,E\Lambda^su)_0 \\
&=-2 \, \re (Op^\w(\im(a_{m-1}))E\Lambda^su,E\Lambda^su)_0,\\
\end{split}
\]
where the last equality follows from the fact that $Op^\w(a_m+\re(a_{m-1}))$ is self-adjoint (since the symbol is real valued).

For (II) we have
\[
(II)=2 \, \re ([E,iA]\Lambda^su,E\Lambda^su)_0+2 \, \re  (E[\Lambda^s,iA]u,E\Lambda^su)_0=(II_1)+(II_2).
\]
 By the Weyl composition formula (see \eqref{Weyl.comp.formula}), since $a=a_m+a_{m-1}$ and $a_m$ is real valued, the following identities hold
\[
\begin{split}
[E,iA]&=Op^\w(-\bracket{a,p}e^p)+\mathrm{Op}^\w(ir_{m-2})+Op^\w(r_{m-3}),  \\
-\bracket{a,p}e^p&=-(H_{a_m}p+H_{a_{m-1}}p)e^p,
\end{split}
\]
where $r_{m-2} \in S^{m-2}(\R^n)$ is real valued and $r_{m-3} \in S^{m-3}(\R^n)$, which yield 
\begin{align}(II_1)=-2\re(Op^\w(H_{a_m}p+H_{a_{m-1}}p)E\Lambda^su,E\Lambda^su)_0+2\re(Op^\w(r'_{m-3})E\Lambda^su,E\Lambda^su)_0,
\end{align}
with $r'_{m-3} \in S^{m-3}(\R^n)$. As for $(II_2)$, we use the identity 
\[
(II_2)= 2\re(E[\Lambda^s,iA](\Lambda^{-s}\widetilde{E}E\Lambda^s+R'_{-2}) u, E\Lambda^s u)_0,
\]
where $R'_{-2}\in \Psi^{-2}(\R^n)$. We then call $b=b(x,\xi)$ the Weyl symbol of the operator $E[\Lambda^s,iA]\Lambda^{-s}\widetilde{E}$, and have, due to condition  \eqref{smallcoeffa2a3}, that
\begin{equation}\label{symbolb}
\abs{\re(b)(x,\xi)}\lesssim \varepsilon \lambda(\abs{x}) \langle \xi\rangle^{m-1}\lesssim \varepsilon \lambda(\abs{x}) |\xi|^{m-1}+\varepsilon c,
\end{equation}
for some $c>0$.

Therefore, since $E[\Lambda^s,iA]R'_{-2}\in \Psi^{s}(\R^n)$, 
\begin{align}
    (I)+(II)&=-2 \re (Op^\w(H_{a_m}p+H_{a_{m-1}}p)E\Lambda^su,E\Lambda^su)_0-2 \, \re (Op^\w(\im (a_{m-1}))E\Lambda^su,E\Lambda^su)_0+\\
    &\quad + 2\re(Op^\w(r'_{m-3})E\Lambda^su,E\Lambda^su)_0 + 2\re(Op^\w (b) E\Lambda^s u, E\Lambda^s u)_0+\re (E[\Lambda^s,iA]R_{-2} u, E\Lambda^s u)_0\\
   &\leq\!\! -2 \re (Op^\w(H_{a_m+\re(a_{m-1})}p+\!\im(a_{m-1})-\!\re(b))E\Lambda^su,E\Lambda^su)_0 \label{C2-1}\\
   &\quad \!\!\!-2\underbrace{\re(Op^\w(iH_{\im(a_{m-1})}p)E\Lambda^s u, E\Lambda^s u)_0}_{=0}+ C\|u\|_s^2.\label{C2-2}
\end{align}
Note that, by the admissibility  of $\re(a)$, the choice of $p$,  the hypothesis on $\im(a_{m-1})$,
and  the properties of the symbol $b$, we have, possibly rescaling $p$, that for $C>0$ sufficiently large
\begin{equation}\label{C2-4}
H_{a_m+\re(a_{m-1})}p+\im(a_{m-1})-\re(b)\geq (C-c_0-\varepsilon)\lambda(|x|)|\xi|^{m-1} -(D+\varepsilon c)\geq C'(C-c_0-\varepsilon)\lambda(|x|)\langle\xi\rangle^{m-1} -(D'+\varepsilon c),
\end{equation}
where $c_0>0$ and $0<C'<1$. By taking the suitable smallness condition, that is for $\varepsilon$ sufficiently small, we get, for two new positive constants that we keep denoting $C$ and $D$,
\begin{equation}\label{C2-3}
    H_{a_m+\re(a_{m-1})}p+\im (a_{m-1})-\re(b)-C\lambda(|x|)\langle \xi\rangle ^{m-1} \geq-D.
    \end{equation}
Then, we apply  the Fefferman-Phong inequality when $m=3$ and the sharp G\aa rding inequality when $m=2$ on $Op^\w(H_{a_m+\re(a_{m-1})}p+\im (a_{m-1})-\re(b)-C\lambda(|x|)\langle \xi\rangle ^{m-1}+D)$, and obtain
\begin{align}
  & -2 \, \re (Op^\w(H_{a_m+\re(a_{m-1})}p+\im (a_{m-1})-\re(b))E\Lambda^su,E\Lambda^su)_0\leq
       -C(Op^\w(\lambda(|x|)\langle \xi\rangle^{m-1}) E\Lambda^s u,E\Lambda^s u)+C\norm{u}_{s}^{2}.\label{eq.FPSG}
\end{align}
Since $m=2,3$, we have
\begin{align}
    \Lambda^{\frac{m-1}{2}}\lambda(|x|) \Lambda^{\frac{m-1}{2}}&=
    Op^\w(\langle \xi\rangle^\frac{m-1}{2}\#\lambda(|x|) \#\langle \xi\rangle^\frac{m-1}{2})\\
    &=Op^\w\Bigl(\lambda(|x|)\langle \xi\rangle^{m-1} -\cancel{\frac i 2 \left\{\langle \xi\rangle^{\frac{m-1}{2}},\lambda(\abs{x})\langle \xi\rangle^{\frac{m-1}{2}}\right\}}-\cancel{\frac{i}{2}\left\{\lambda(\abs{x}),\langle \xi\rangle^{\frac{m-1}{2}}\right\}\langle\xi\rangle^{\frac{m-1}{2}} }+r_0 \Bigr)\\
&=Op^\w(\lambda(|x|)\langle \xi\rangle^{m-1}+r_0),
\end{align}
for some $r_0\in S^0(\R^n)$. Moreover, since  for every $P\in \Psi^0(\R^n)$ the operator $[\lambda(|x|)^{1/2}\Lambda^{\frac{m-1}{2}}, P]\in \Psi^0(\R^n)$  is  bounded on $L^2(\R^n)$, we have
\begin{align}
    (\Lambda^{\frac{m-1}{2}}\lambda(|x|) \Lambda^{\frac{m-1}{2}} v,v)_0&=\|\lambda(|x|)^{1/2} \Lambda^{\frac{m-1}{2}}v\|_0^2=\|\lambda(|x|)^{1/2} \Lambda^{\frac{m-1}{2}}(\widetilde{E} E-R_{-2})v\|_0^2\\
    &\lesssim\|\lambda(|x|)^{1/2} \Lambda^{\frac{m-1}{2}} Ev\|_0^2+O(\| v\|_0^2),
\end{align}
where $\widetilde{E}:=Op^\w(e^{-p})\in \Psi^0(\R^n)$ is such that $\widetilde{E} E=I+R_{-2}$ with $R_{-2}\in \Psi^{-2}(\R^n)$.
Therefore, from \eqref{eq.FPSG} and the previous considerations, we get
\begin{align} 
(I)+(II)&=-2 \, \re (Op^\w(H_{a_m+\re(a_{m-1})}p+\im (a_{m-1})-\re(b))E\Lambda^su,E\Lambda^su)_0+O(\norm{u}_s^2)\\
&\leq -C(\lambda(|x|) \Lambda^{\frac{m-1}{2}}E\Lambda^su,\Lambda^{\frac{m-1}{2}}E\Lambda^su)+O(\norm{u}_{s}^2) \\
&= -C \| \lambda(\abs{x})^{1/2}\Lambda^{\frac{m-1}{2}}E \Lambda^s u\|_0^2+O(\|u\|^2_s)\\
&\leq -C \|\lambda(\abs{x})^{1/2}\Lambda^{\frac{m-1}{2}} \Lambda^s u\|_0^2+O(\|u\|^2_s), \label{eqn19032}
\end{align}
where the constant $C$ in the last line is a new suitable positive constant.

Now, since $N(\cdot)$ is equivalent to $\|\cdot\|_s$, then, denoting $C_1$ a positive constant possibly different at any appearance, from \eqref{estim.norm1},\eqref{eqn19033-2} and \eqref{eqn19032} we have
\begin{align}
    \partial_{t}N(u)^2&\leq  -C(\lambda(|x|) \Lambda^{s+\frac{m-1}{2}}u,\Lambda^{s+\frac{m-1}{2}}u)_0+C_1\norm{u}_s^2+C_0N(u)^2+2\re(\Lambda^{s-2}f,\Lambda^{s-2}u)_0+2\re( E\Lambda^s f,E\Lambda^su)_0 \\
    &\leq  -C(\lambda(|x|) \Lambda^{s+\frac{m-1}{2}}u,\Lambda^{s+\frac{m-1}{2}}u)_0 +C_1N(u)^2 +2\re(\Lambda^{s-2}f,\Lambda^{s-2}u)_0+2\re( E\Lambda^s f,E\Lambda^su)_0\label{eq06061102}\\
     &\leq-C(\lambda(|x|) \Lambda^{s+\frac{m-1}{2}}u,\Lambda^{s+\frac{m-1}{2}}u)_0+C_1N(u)^2+C_2N(f)N(u). \label{eq0702}
\end{align} 
The last inequality implies that
\begin{align}
    2N(u)\partial_tN(u)
\leq C_1N(u)^2+C_2N(u)N(f),
\end{align}
which is equivalent to
\begin{align}
    \partial_{t}(e^{-C_1t}N(u))\leq e^{-C_1t} C_2N(f).
\end{align}
Therefore, integrating in time and taking the supremum in $[0,T]$ we get (i).

To prove (ii), we go back to \eqref{eq0702} and we rewrite it as
\begin{align}
    \partial_{t}(e^{-C_1t}N(u)^2)\leq e^{-C_1t} \Big(-C(\lambda(|x|) \Lambda^{s+\frac{m-1}{2}}u,\Lambda^{s+\frac{m-1}{2}}u)_0+C_2N(u)N(f)\Big).
\end{align}
Then, integrating in time and using once more the equivalence of Sobolev norms, we get
\begin{align}
\sup_{0\leq t \leq T}\|u\|^2_s+\int_0^T(\lambda(|x|) \Lambda^{s+\frac{m-1}{2}}u(t,\cdot),\Lambda^{s+\frac{m-1}{2}}u(t,\cdot))_0\,dt\lesssim e^{C_1T} \left(\|u(0,\cdot)\|_s^2+ \int_0^T \|f(t)\|_s^2 dt\right),
\end{align}
which proves part (ii) of the result.

To prove (iii) we use \eqref{eq06061102} and estimate $\re(E\Lambda^s f,E\Lambda^s u)$ as follows
\begin{align}
\re(E\Lambda^s f,E\Lambda^s u)&=\re(\Lambda^s f,E^\ast E\Lambda^su)_0 \\
&=\re(\Lambda^{s-(\frac{m-1}{2})}f,\Lambda^{\frac{m-1}{2}}E^\ast E\Lambda^s u)_0\\
&=\re(\lambda(|x|)^{-1/2}\Lambda^{s-\frac{m-1}{2}}f, \lambda(|x|)^{1/2}\Lambda^{\frac{m-1}{2}}E^\ast E\Lambda^s u)_0\\
&\leq \frac{1}{2\delta}\|\lambda(|x|)^{-1/2}\Lambda^{s-\frac{m-1}{2}}f\|_0^2+\frac \delta 2\|\lambda(|x|)^{1/2}\Lambda^{\frac{m-1}{2}}E^\ast E\Lambda^s u\|_0^2\\
& \leq \frac{1}{2\delta}\|\lambda(|x|)^{-1/2}\Lambda^{s-\frac{m-1}{2}}f\|_0^2+ \delta \|\lambda(|x|)^{1/2}E^\ast E\lambda(|x|)^{-1/2}\lambda(|x|)^{1/2}\Lambda^{s+\frac{m-1}{2}} u\|_0^2\\
&+\delta \|\lambda(|x|)^{1/2}[\Lambda^{\frac{m-1}{2}},E^\ast E]\Lambda^s u\|_0^2\\
&\lesssim \frac{1}{\delta} \|\lambda(|x|)^{-1/2}\Lambda^{s-\frac{m-1}{2}}f\|_0^2+ \delta \|\lambda(|x|)^{1/2}\Lambda^{s+\frac{m-1}{2}} u\|_0^2+C_\delta \|u\|_s^2
 \label{eq06061118}
\end{align}
Using \eqref{eq06061118} into \eqref{eq06061102} and the equivalence of $N(\cdot)$ and $\|\cdot\|_s$, and choosing $\delta$ sufficiently small, for $C, C_1$ being new positive constants we obtain
\begin{align}
    \partial_{t}N(u)^2
    \leq&  -C(\lambda(|x|) \Lambda^{s+\frac{m-1}{2}}u,\Lambda^{s+\frac{m-1}{2}}u) +C_1N(u)^2\\
    &+C_2\norm{\lambda(|x|)^{-\frac 1 2}\Lambda^{s-\frac{m-1}{2}} f}_0^2+2\re(\Lambda^{s-2} f,\Lambda^{s-2} u). \label{eqn19033}
\end{align}
Moreover, since $s+\frac{m-1}{2}-4<s$, we have
\begin{align}
  2\re(\Lambda^{s-2} f,\Lambda^{s-2} u)=&2\re(\lambda(\abs{x})^{-\frac 1 2}\Lambda^{s-\frac{m-1}{2}} f,\lambda(\abs{x})^{\frac 1 2}\Lambda^{s+\frac{m-1}{2}-4} u)\lesssim \|\lambda(\abs{x})^{-\frac{1}{2}}  \Lambda^{s-\frac{m-1}{2}} f\|_0^2+N(u)^2,
\end{align}
consequently, for $C_1,C_2$ new suitable positive constants, we get
\begin{align}
    \partial_{t}N(u)^2
    \leq&  -C(\lambda(|x|) \Lambda^{s+\frac{m-1}{2}}u,\Lambda^{s+\frac{m-1}{2}}u) +C_1N(u)^2+C_2\norm{\lambda(|x|)^{-\frac 1 2}\Lambda^{s-\frac{m-1}{2}} f}_0^2.\label{eqn19034}
\end{align}
From \eqref{eqn19034}, integrating in time and using once again the equivalence of $N(\cdot)$ and $\|\cdot\|_s$,  we finally reach
\begin{align}
&\sup_{0\leq t \leq T}\|u\|^2_s+\int_0^T(\lambda(|x|) \Lambda^{s+\frac{m-1}{2}}u(t,\cdot),\Lambda^{s+\frac{m-1}{2}}u(t,\cdot))_0\,dt\\
&\quad\lesssim e^{C_1T} \left(\|u(0,\cdot)\|_s^2 +\int_0^T (\lambda(|x|)^{-1}\Lambda^{s-\frac{m-1}{2}}f(t,\cdot),\Lambda^{s-\frac{m-1}{2}}f(t,\cdot))_0dt\right),
\end{align}
which concludes the proof (iii).

As for the proof of (i'), (ii') and (iii'), we set $f=(\partial_t-iA)^*u=-\partial_tu+iA^*u$ and repeat all the calculations above with the suitable sign adjustment and with $u(t,x)$ replaced by $u(T-t,x)$. Due to the flipped signs, we have to ensure that Lemma \ref{newDoilem} still applies, that is, that $-\re(\bar{a})$ satisfies condition \eqref{ellderivatives} and \eqref{smallcoeffa2a3} (recall that we are working with Weyl symbols, hence $-A^*=\mathrm{Op}^\w(-\bar{a})$). Since this is always the case due to the fact that $\re(a)$ satisfies such conditions, we can conclude the proof as in the cases (i), (ii), (iii).
\end{proof}

\begin{remark}\label{rmk.conf.thm.LIN}
    We remark that  Lemma \ref{lem.smoothing} still holds if we make the following replacement in the assumptions: 
    \[\hypertarget{C1}{(C1)}=\eqref{smallcoeffa2a3} \text{ on } \re(a)\underset{\text{replaced by}}{
    \longrightarrow} \hypertarget{C2}{(C2)}=\eqref{smallcoeffa2a3} \text{ on } a_m \text{ and }  |\p^\alpha_x\p^\beta_\xi \re(a_{m-1})|\lesssim \langle x\rangle^{-N} \langle \xi\rangle^{m-1-|\beta|}. \] 
    In order to show that such a change still gives the result, we explain where \hyperlink{C1}{(C1)} is used in the proof and how \hyperlink{C2}{(C2)} gives the same result. 

    Condition \hyperlink{C1}{(C1)} is used in the proof of Lemma \ref{lem.smoothing} to have 
    \eqref{symbolb} verified and to ensure that \eqref{C2-3} is satisfied (the latter is proved by using \eqref{symbolb}).

    Condition \hyperlink{C2}{(C2)} still implies \eqref{symbolb}, since
    \[|\re(b)(x,\xi)|\lesssim \varepsilon \lambda(|x|)|\xi|^{m-1}+\varepsilon c + \lambda(|x|)\langle \xi\rangle ^{m-2}\lesssim \varepsilon \lambda(|x|)|\xi|^{m-1}+c',\]
    for some $c'=c'(\varepsilon)>0$.
Moreover,  possibly rescaling $p$ so that the constant $C'$ below is large enough, \hyperlink{C2}{(C2)} and \eqref{symbolb} give \eqref{C2-3}, that is,
\begin{align}
    &H_{a_m}p+H_{\re(a_{m-1})}p+\im(a_{m-1})-\re(b)-C\lambda(|x|)\langle \xi\rangle^{m-1}\\
    &\geq C'\lambda(|x|)|\xi|^{m-1}-C_1\varepsilon \lambda(|x|)|\xi|^{m-1}-c_0\lambda(|x|)|\xi|^{m-1}-c''-C\lambda(|x|)\langle \xi\rangle^{m-1}\geq -D,
\end{align}
for some $D,c''>0$ and $c_0$ as in \eqref{ellderivatives}.
This finally shows that \hyperlink{C2}{(C2)} can be used in place of \hyperlink{C1}{(C1)} in the hypotheses of Lemma \ref{lem.smoothing}.

Due to the above observation, Theorem \ref{thm.LIVP} can also be formulated in terms of \hyperlink{C2}{(C2)} instead of \hyperlink{C1}{(C1)}.
\end{remark}
We now use the a priori estimates of Lemma \ref{lem.smoothing} to prove Theorem \ref{thm.LIVP}.

\proof[Proof of Theorem \ref{thm.LIVP}]
\textit{Step 1: Uniqueness}.
To prove the uniqueness it suffices to note that, if $u_1,u_2 \in C([0,T]; H^s(\mathbb{R}^n))$ are solutions to \eqref{LIVP}, then $v:=u_1-u_2\in C([0,T]; H^s(\mathbb{R}^n))\cap C^1([0,T];H^{s-m}(\R^n))$ solves \eqref{LIVP} with $f=0$ and $u_0=0$. Thus, by Lemma \ref{lem.smoothing}-(i), we have $v=0$. 

\textit{Step 2: Existence and smoothing estimates}. To prove (i), (ii), and (iii) we start with the analysis of the case  $u_0 \in \mathscr{S}(\mathbb{R}^n)$ and $f \in \mathscr{S}(\mathbb{R}^{n+1})$.
Under these hypotheses, we take $P:=\partial_t-iA$ and define the subspace $W\subseteq L^1([0,T];H^{-s}(\R^n))$ as
\[
W= \lbrace \psi \in L^1([0,T];H^{-s}(\R^n)); \; \exists \, \varphi \in C_c^\infty([0,T[ \times \R^n), \; \mathrm{s.t.} \; \psi=P^\ast \varphi \rbrace.
\]
We consider the linear functional $\ell: W \rightarrow \mathbb{C}$ given by 
\[
\ell(\psi):= \int_0^T (f(t,\cdot),\varphi(t,\cdot))_0 \, dt+ (u_0,\varphi(0, \cdot))_0,
\]
where $\varphi$ satisfies $\psi=P^\ast \varphi$. 
Note that $\varphi$ is uniquely determined due to point (i)' of Lemma \ref{lem.smoothing} and the fact that $\varphi_1(T,\cdot)=\varphi_2(T,\cdot)=0$. 
To see that $\ell$ is continuous, we use Lemma \ref{lem.smoothing}-(i)' on $\varphi$ with $s$ replaced by $-s$, which gives
\[
\begin{split}
\abs{\ell (\psi)}&\leq \norm{f}_{L^1([0,T]);H_x^s)}\sup_{t \in [0,T]}\norm{\varphi(t,\cdot)}_{H_x^{-s}}+\norm{u_0}_{H^s_x}\norm{\varphi(0,\cdot)}_{H_x^{-s}} \\
&\leq C_2e^{C_1T}\Bigl(\norm{f}_{L^1([0,T]);H_x^s)}+\norm{u_0}_{H^s_x}\Bigr)\|\psi\|_{L^1([0,T]);H_x^{-s})}.
\end{split}
\]
By Hahn-Banach's theorem we can extend the functional to $L^1([0,T];H^{-s}(\R^n))$. Then, by Riesz's theorem we have the existence of $u \in (L^1([0,T];H^{-s}(\R^n)))^\ast=L^\infty([0,T]; H^s(\R^n))$ such that, for $\psi=P^\ast \varphi$,
\begin{equation}\label{eqfunctell}
\ell(\psi)=(u,P^\ast \varphi)_0=\int_0^T (f(t,\cdot),\varphi(t,\cdot))_0 \, dt+ (u_0,\varphi(0, \cdot))_0,
\end{equation}
and thus $Pu=f$ in $\mathscr{D'}(]0,T[ \times \mathbb{R}^n)$. Moreover, since $Pu=\partial_tu-iAu=f$ and $f \in \mathscr{S}(\R^{n+1})$, we have that $\partial_t u \in L^\infty([0,T]; H^{s-m}(\R^n))$ and so $u \in C([0,T];H^{s-m}(\R^n))$. Using the equation once more, we also get $u \in C^1([0,T]; H^{s-2m}(\R^n))$. Since by \eqref{eqfunctell} 
\[
\int_0^T (u(t,\cdot),(\partial_t-iA)^*\varphi(t,\cdot))_0 \, dt =\int_0^T (f(t,\cdot),\varphi(t,\cdot))_0 \, dt+ (u_0,\varphi(0, \cdot))_0,
\]
we have $u(0,\cdot)=u_0$. Finally,  using the previous argument with $s+2m$ in place of $s$, we obtain a solution of the $IVP$ $\eqref{LIVP}$ for which the estimates in Lemma \ref{lem.smoothing} hold under the hypotheses $u_0 \in \mathscr{S}(\R^n)$ and $f \in \mathscr{S}(\R^{1+n})$. 

Now, we use the result in the case $u_0 \in \mathscr{S}(\R^n)$ and $f \in \mathscr{S}(\R^{1+n})$ to prove the cases (i), (ii), and (iii).

\textit{Proof of (i).} Let $u_0 \in H^s(\R^n)$ and $f \in L^1([0,T];H^s(\R^n))$.  In this case, by density, there exist two sequences $(f_j)_{j \in \mathbb{N}} \subseteq \mathscr{S}(\R^{n+1})$ and $(v_j)_{j \in \mathbb{N}} \subseteq \mathscr{S}(\R^n)$ such that $f_j \rightarrow f$ and $v_j \rightarrow u_0$ as $j \rightarrow +\infty$. Therefore, due to what we have proved above in the case when the initial datum and the source terms are in $\mathscr{S}(\R^n)$, for all $j \in \mathbb{N}$ there exists a solution $u_j$ of $\eqref{LIVP}$ with initial datum $v_j$ and source term $f_j$ which satisfies Lemma \ref{lem.smoothing}-(i). Since $(u_j)_{j \in \mathbb{N}}$ is a Cauchy sequence, then $u:=\lim_{j \rightarrow +\infty} u_j$ solves \eqref{LIVP} with initial datum $u_0$ and source term $f$ and satisfies estimate (i) of Lemma \ref{lem.smoothing}. This proves point (i) of the theorem.

\textit{Proof of (ii)}. Let $u_0 \in H^s(\R^n)$ and $f \in L^2([0,T];H^s(\R^n))$. To prove this part we proceed as in the proof of (i), but considering $(f_j)_j \in \mathscr{S}(\R^{n+1})$ such that $f_j\rightarrow f$ in $L^2([0,T];H^s(\R^n))$ as $j \rightarrow +\infty$, and using (ii) of Lemma \ref{lem.smoothing} instead of (i). 

\textit{Proof of (iii)} Let $u_0 \in H^s(\R^n)$ and $\Lambda^{s-\frac{m-1}{2}}f\in L^2([0,T]\times \R^n;\lambda(\abs{x})^{-1} dx dt)$. By the boundedness theorem for pseudo-differential operators, it is possible to prove that there exists a sequence $(g_j)_{j \in \mathbb{N}}\in\mathscr{S}(\R^{1+n})$ such that $g_j \rightarrow \Lambda^{s-\frac{m-1}{2}}f$ in $L^2([0,T]\times \R^n;\lambda(\abs{x})^{-1} dx dt)$ as $j \rightarrow + \infty$. Then,  applying the same argument used in the proof of (i) with $f_j$ replaced by $\Lambda^{-s+\frac{m-1}{2}}g_j$, and using part (iii) of Lemma $\ref{lem.smoothing}$, we finally get the proof. 
\endproof

\begin{remark}[Smoothing estimates for operators of order  $m=2$]\label{rmk.m=2}
We wish to emphasize once more that our smoothing estimates hold for pseudo-differential operators $D_t-A$ with $A$ satisfying the hypotheses of Theorem \ref{thm.LIVP}, hence not necessarily differential. Moreover, when $m=2$ and $A$ satisfies such conditions, then the IVP \eqref{LIVP} includes the one for ultrahyperbolic variable coefficient Schr\"odinger operators treated in \cite{KPRV05}, and that for variable coefficient Schr\"odinger operators considered in \cite{D96},\cite{KPRV05} and \cite{CKS}, whenever our hypotheses are satisfied. Therefore, our result includes those already known in the literature when $m=2$ and \eqref{ellderivatives} and \eqref{smallcoeffa2a3} hold.
\end{remark}

\begin{remark}[Smoothing estimates for operators of order $m>3$]\label{rmk.Gard-FP}
    Given a pseudo-differential operator $A$ of order $m>3$ satisfying the hypotheses in Theorem \ref{thm.LIVP}, one hopes to prove the same smoothing effect for the solution of the corresponding IVP \eqref{LIVP}. However, when $m>3$ one faces some obstructions due to the absence of a priori estimates stronger than the Fefferman-Phong inequality, i.e. estimates that return $L^2$-error terms when applied to operators of order $m>3$. In turn, when $m>3$, the previous strategy gives smoothing estimates with a $H^{\frac{m-3}{2}}$- error, which, unfortunately, we cannot eliminate with our technique.  
\end{remark}

\section{Well-posedness for the NLIVP for KdV-type operators with variable coefficients}\label{sec.NLIVP}
In this section we focus on the well-posedness of the NLIVP 
\begin{equation}\label{NLIVP3}
\begin{cases}
\partial_t u -iAu-N(u,\bar{u},  D^\alpha u)=0 \quad \quad \mathrm{on} \quad \R^n \times ]0,T[, \quad T>0, \\
u(0,\cdot)=u_0 \quad \quad \mathrm{on} \quad \R^n,
\end{cases}
\end{equation}
where $u_0 \in \mathscr{S}(\R^n)$, $N=N(u,\bar{u}, D^\alpha u)$ is a polynomial with no linear or constant terms, and $A=Op^\w(a)$ is a pseudo-differential operator of order $3$ whose Weyl symbol $a\in S^3(\R^n)$ satisfies \eqref{cond.a} and \eqref{cond.ima}, and such that $\re(a)$ satisfies \eqref{ellderivatives} and \eqref{smallcoeffa2a3}.
  Moreover, we will consider nonlinear terms of the form
\begin{equation}\label{Npqalpha}
N(u,\bar{u},D^\alpha u)=u^p\bar{u}^q D^\alpha u,
\end{equation}
with $p,q\in \mathbb{N}_0$ such that $(p,q)\neq (0,0)$, and $\alpha \in \mathbb{N}_0^n$ such that $\abs{\alpha}\leq 2$. 
\begin{remark}
Even if we focus on $N(u,\bar{u}, D^\alpha u)$ as in \eqref{Npqalpha}, it is possible to consider nonlinearities defined as a finite linear combination of terms of the form \eqref{Npqalpha}. We chose not to focus on this more general case, as it does not introduce any additional difficulties in the proof; rather, it would merely require repeating the arguments used for a single term of the form \eqref{Npqalpha}. 
\end{remark}
\begin{remark}
Here we devote our analysis to the case $m=3$ instead of $m=2,3$.
However, by repeating the same classical techniques employed below for the case $m=3$, all the results continue to hold for operators of order two, provided that the nonlinear terms can inlude derivatives of $u$ up to order $|\alpha|\leq 1$. Since there is no difference in the proofs when $m=2$ and $m=3$, and since the case $m=2$ has been extensively treated in \cite{KPRV05}, we focus on the open problem $m=3$.
\end{remark}

\begin{remark}
Since we work with operators of order three, in view of the \textit{inhomogeneous smoothing effect} it is possible to consider nonlinearities that involve derivatives of order two. This is due to the fact that
by Theorem \ref{thm.LIVP}  the solution of dispersive equations of order $m=3$, with possibly variable coefficients, gains  $m-1=2$ derivatives (in the suitable sense) with respect to the inhomogeneous term of the equation. This fact is well-known in the third-order constant coefficients case (see \cite{CS89,KPV93}) and when the third-order principal part has constant coefficients (see \cite{KS}), but it is new when the variable coefficients appear in the principal part. 
Note finally that nonlinearities of the form \eqref{Npqalpha} include the classical ones $N(u)=\pm u \abs{u}^{2k}$ and  $N(u,\bar{u},\nabla u, \nabla \bar{u})=\pm \nabla u \cdot  u^{2k}$, with $k\geq 1$, both when $m=2$ and $m=3$. These nonlinearities were treated in \cite{FS21} for Schr\"odinger equations with time-dependent coefficients. Even if we do not examine the space-time dependent coefficients case in this work, it can be approached by combining our techniques with those in \cite{FS21}. 
\end{remark}

\begin{remark}
    We remark that our general nonlinearity allows us to include in our class of problems  the \textit{k-generalized KdV equation} on $[0,T]\times \R$, that is
  \begin{align}   
 &\left\{ \begin{array}{ll}
        \partial_t u +\partial_x^3u+u^k\partial_x u=0, & \mathrm{on} \quad [0,T]\times \R \,\,\,\,(k\in \mathbb{Z}_+) \quad (\textit{k-generalized KdV}) \\
        u(0,\cdot)=u_0 &\mathrm{on} \quad \R,
    \end{array}\right.
    \end{align}
    where the admissible operator $A$ here is $A=D_{x}^3$, and the \textit{$k$-generalized ZK equation}, that is
 \begin{align}   
 &\left\{ \begin{array}{ll}
        \partial_t u +\partial_{x_1}\Delta_{x}u+u^k\partial_{x_1} u=0, & \mathrm{on} \quad [0,T]\times \R^2 \,\,\,\,(k\in \mathbb{Z}_+) \quad (\textit{k-generalized ZK})\\
        u(0,\cdot)=u_0 &\mathrm{on} \quad \R^2,
    \end{array}\right.
    \end{align}
 where the admissible operator $A$ here is $A=-D_{x_1}\Delta_x$. When $k=1$ the equations above give the KdV and the ZK equation, respectively.
\end{remark}

In order prove the well-posedness of the NLIVP \eqref{NLIVP3}, we use the standard contraction argument.

\begin{proof}[Proof of Theorem \ref{thm.NLIVP}]
The first step is to reduce our problem
\[
\begin{cases}
\partial_t u -iAu-N(u,\bar{u},D^\alpha u)=0 \quad \quad \mathrm{on} \quad \R^n \times ]0,T[, \quad T>0, \\
u(0,\cdot)=u_0 \quad \quad \mathrm{on} \quad \R^n,
\end{cases}
\]
to an equivalent one. 
To do that, we rewrite $N$ as
\begin{align}
N=N(u,\bar{u},D^\alpha u)&=u^p\bar{u}^qD^\alpha u \\
&=(u^p\bar{u}^q-u_0^p\bar{u}_0^q)D^\alpha u + u_0^p\bar{u}_0^qD^\alpha u  \\
&=:\widetilde{N}(u,\bar{u},D^\alpha u)+u_0^p\bar{u}_0^qD^\alpha u,\label{eq.pf.1}
\end{align}
incorporate the last term on the RHS of \eqref{eq.pf.1} to the operator $A$, and call the resulting operator $\tilde{A}$. With the above notations, our NLIVP becomes equivalent to 
\begin{equation} \label{IVP3.1}
    \left\{
\begin{array}{ll}
 \partial_t u -i\tilde{A}u-\widetilde{N}(u,\bar{u},D^\alpha u)=0    &   \mathrm{on} \quad \R^n \times ]0,T[, \quad T>0,\\
 u(0,\cdot)=u_0    & \mathrm{on} \quad \R^n,
\end{array}
    \right.
\end{equation}
where $\tilde{A}=\tilde{A}_{u_0}=A-iu_0^p\bar{u}_0^qD^\alpha $, due to the property $u_0 \in \mathscr{S}(\mathbb{R}^n)$, is still an operator of order three with an admissible real part and satisfying \hyperlink{C2}{(C2)} in Remark \ref{rmk.conf.thm.LIN} and 
\eqref{cond.ima}.
Note that the real part of the symbol of $\tilde{A}$ is admissible by Proposition \ref{prop.admiss.a}. Moreover, since $\tilde{A}$ satisfies condition \hyperlink{C2}{(C2)} (but not \hyperlink{C1}{(C1)}) in Remark \ref{rmk.conf.thm.LIN}, then Theorem \ref{thm.LIVP} still holds with $\tilde{A}$ in place of $A$ and there exists a unique solution of the linear initial value problem associated with $D_t-\tilde{A}$ (see Remark \ref{rmk.conf.thm.LIN}).

Next, we want to solve the integral equation \begin{equation}
\label{eqintNLIVP2}
u=\widetilde{W}(t)\, u_0+\int_0^t \widetilde{W}(t-t')\tilde{N}(u,\bar{u},D^\alpha u) \, dt',
\end{equation}
where here $\widetilde{W}(t)u_0$ is the solution at time $t$ of the linear homogeneous equation    \begin{equation}\label{homLIVP} 
    \left\{
\begin{array}{ll}
 \partial_t u -i\tilde{A}u=0 \quad \quad & \mathrm{on} \quad \R^n \times ]0,T[, \quad T>0,\\
 u(0,\cdot)=u_0,  & \mathrm{on} \quad \R^n,
\end{array}
    \right.
\end{equation}
Of course $u$ solves \eqref{eqintNLIVP2} if and only if it solves \eqref{NLIVP3}.
To solve \eqref{eqintNLIVP2} and show that the solution has the desired properties, for $s \in \mathbb{R}$, $s> n+8N+9$, $s\in 2\bN+1$,  $\lambda(\abs{x})=\langle x \rangle^{-2N}$, $N \in \mathbb{N}$, $N>1$, and $T>0$ to be chosen, we define the metric space  
\[
\begin{split}
X_T^s:=\Bigl \lbrace u\in C([0,T] \times \R^n ; \mathbb{C}); &\st\norm{u}_{H^s_x} < +\infty, \;  \Bigl(\int_0^T \int_{\R^n}\lambda(\abs{x})\abs{\Lambda^{s+1}u}^2 \, dx \, dt\Bigr)^{1/2} < +\infty, \\
& \st\norm{\lambda (\abs{x})^{-1}u}_{ H^{s-4N-2}_x}  < +\infty, \,  \st\norm{ \lambda (\abs{x})^{-1} \partial_t u}_{ H_x^{s-4N-5}} <+\infty \Bigr \rbrace,
\end{split}
\]
with norm 
\[
\norm{u}^2_{X_{T}^s}:=\st\norm{u}_{ H^s_x}^2+\int_0^T \int_{\R^n}\lambda(\abs{x})\abs{\Lambda^{s+1}u}^2 \, dx \, dt+\st\norm{\lambda (\abs{x})^{-1}u}_{ H^{s-4N-2}_x}^2+\st\norm{ \lambda (\abs{x})^{-1}\partial_t u}^2_{H_x^{s-4N-5}}.
\]
Also, we define the map $\Phi_{u_0}:X^s_T\longrightarrow L^\infty_tH^s$  as 
\[
\Phi_{u_0}(u)(t):=\widetilde{W}(t)\, u_0+\int_0^t \widetilde{W}(t-t')\tilde{N}(u,\bar{u},D^\alpha u) \, dt',
\]
and call $B_R\subset X_T^s$ the complete metric space
\[B_R:=\{u\in X^s_T;\ u(0,\cdot)=u_0,\ \|u\|_{X^s_T}\leq R\}.\]
To prove the result, it suffices to use the fixed point theorem on the map $\Phi_{u_0}$ restricted to $B_R$, where $R$ will be fixed later. 
To apply the contraction argument, the first step is to prove that the map $\Phi_{u_0}$ sends $B_R$ into itself, that is
that $\norm{\Phi_{u_0}(u)}_{X^s_T}\leq R$ for all $u\in B_R$, while the second step will be to prove that the map is a contraction on $B_R$. This will allow us to conclude the first part of the theorem, while the continuity property in the second part of the theorem will be proved separately.
Since $v:=\Phi_{u_0}(u)$ solves the linear problem
\begin{equation} \label{IVP3.2}
    \left\{
\begin{array}{ll}
 \partial_t v -i\tilde{A}v=\widetilde{N}(u,\bar{u},D^\alpha u) \quad \quad & \mathrm{on} \quad \R^n \times ]0,T[, \quad T>0, \\
 v(0,\cdot)=u_0 & \mathrm{on} \quad \mathbb{R}^n,  
\end{array}
    \right.
\end{equation}
then to estimate $\norm{v}_{X^s_T}:=\norm{\Phi_{u_0}(u)}_{X^s_T}$ we can use the linear smoothing estimates in Theorem \ref{thm.LIVP}. 

We first prove that the map $\Phi_{u_0}$ sends $B_R$ into itself. Since
\[
\begin{split}
\norm{v}^2_{X_{T}^s}&:= \st\norm{v}_{H^s_x}^2+\int_0^T \int_{\R^n}\lambda(\abs{x})\abs{\Lambda^{s+1}v}^2 \, dx \, dt+\st\norm{\lambda (\abs{x})^{-1}v}_{ H^{s-4N-2}_x}^2+ \st\norm{\lambda(\abs{x})^{-1} \partial_t v}_{H_x^{s-4N-5}}^2 \\
&=(I)+(II)+(III)+(IV),
\end{split}
\]
we estimate  terms (I)-(IV) separately.

For $(I)+(II)$, by the smoothing effect in Theorem \ref{thm.LIVP}, we have
\[
(I)+(II) \lesssim \norm{u_0}_{ H^s_x}^2+\int_0^T \int_{\R^n}\lambda(\abs{x})^{-1}\abs{\Lambda^{s-1}\tilde{N}(u,\bar{u},D^\alpha u)}^2 \, dx \, dt.
\]
Since $s-1 \in 2\mathbb{N}$, then $\Lambda^{s-1}$ is a differential operator of order $s-1$ and by Leibniz rule
we have 
\[
\begin{split}
\Lambda^{s-1}\tilde{N}(u,\bar{u},D^\alpha u)&=\Lambda^{s-1}(u^p\bar{u}^q-u_0^p\bar{u}_0^q) (D^\alpha u) \\
&=(\Lambda^{s-1} D^\alpha u) (u^p\bar{u}^q -u_0^p\bar{u}_0^q) \\
& \quad + \sum_{\substack{\abs{\gamma_1}+\abs{\gamma_2}\leq s-1 \\ s/2-1/2 \leq \abs{\gamma_1} < s-1, \; \abs{\gamma_2}\leq s/2-1/2}}C_{\gamma_1,\gamma_2,s}(D^{\gamma_1} D^\alpha u)D^{\gamma_2}(u^p\bar{u}^q-u_0^p\bar{u}_0^q) \\ 
& \quad +\sum_{\substack{\abs{\gamma_1}+\abs{\gamma_2}\leq s-1 \\ \abs{\gamma_1} < s/2-1/2, \; \abs{\gamma_2}\geq s/2-1/2}}C_{\gamma_1,\gamma_2,s}(D^{\gamma_1} D^\alpha u)D^{\gamma_2}(u^p\bar{u}^q-u_0^p\bar{u}_0^q)\\
& \quad 
+\sum_{\substack{\abs{\gamma_1}+\abs{\gamma_2}\leq s-1 \\ \abs{\gamma_1} < s/2-1/2, \; \abs{\gamma_2}<s/2-1/2}}C_{\gamma_1,\gamma_2,s}(D^{\gamma_1} D^\alpha u)D^{\gamma_2}(u^p\bar{u}^q-u_0^p\bar{u}_0^q),
\end{split}
\]
which gives
\[
\begin{split}
&(II) \lesssim \|u_0\|_{H_x^s}^2+ \int_0^T \int_{\R^n}\lambda(\abs{x})^{-1}\abs{(\Lambda^{s-1} D^\alpha u) (u^p\bar{u}^q-u_0^p\bar{u}_0^q)}^2 \, dx \, dt \\
& + \sum_{\substack{\abs{\gamma_1}+\abs{\gamma_2}\leq s-1 \\ s/2-1/2 \leq \abs{\gamma_1} < s-1, \; \abs{\gamma_2}\leq s/2-1/2}}C_{\gamma_1,\gamma_2,s}\int_0^T \int_{\R^n}\lambda(\abs{x})^{-1}\abs{(D^{\gamma_1} D^\alpha u)(D^{\gamma_2}(u^p\bar{u}^q-u_0^p\bar{u}_0^q))}^2 \, dx \, dt \\
& + \sum_{\substack{\abs{\gamma_1}+\abs{\gamma_2}\leq s-1 \\ \abs{\gamma_1} < s/2-1/2, \; \abs{\gamma_2}\geq s/2-1/2}}C_{\gamma_1,\gamma_2,s}\int_0^T \int_{\R^n}\lambda(\abs{x})^{-1}\abs{(D^{\gamma_1} D^\alpha u)(D^{\gamma_2}(u^p\bar{u}^q-u_0^p\bar{u}_0^q))}^2 \, dx \, dt \\
& + \sum_{\substack{\abs{\gamma_1}+\abs{\gamma_2}\leq s-1 \\ \abs{\gamma_1} < s/2-1/2, \; \abs{\gamma_2}< s/2-1/2}}C_{\gamma_1,\gamma_2,s}\int_0^T \int_{\R^n}\lambda(\abs{x})^{-1}\abs{(D^{\gamma_1} D^\alpha u)(D^{\gamma_2}(u^p\bar{u}^q-u_0^p\bar{u}_0^q))}^2 \, dx \, dt \\
&:=\|u_0\|_{H_x^s}^2+(II)_a+(II)_b+(II)_c+(II)_d.
\end{split}
\]
For $(II)_a$ we have 
\[
\begin{split}
(II)_a&=\int_0^T \int_{\R^n}\lambda(\abs{x}) \abs{\Lambda^{s-1}D^\alpha u}^2 \cdot \abs{\lambda(\abs{x})^{-1}(u^p\bar{u}^q-u_0^p\bar{u}_0^q)}^2 \, dx \, dt \\
&\lesssim T^2 \Bigr(\int_0^T \int_{\R^n} \lambda(\abs{x})\abs{\Lambda^{s-1}D^\alpha u}^2 dx \, dt \Bigl) \norm{\lambda(\abs{x})^{-1}\partial_tu^p\bar{u}^q }^2_{L_t^\infty L_x^\infty} \\
& \lesssim T^2 \Bigr(\int_0^T \int_{\R^n} \lambda(\abs{x})\abs{\Lambda^{s-1}D^\alpha u}^2 dx \, dt \Bigl) \norm{\lambda(\abs{x})^{-1}\partial_tu^p\bar{u}^q }^2_{L_t^\infty H_x^{n/2+\varepsilon}},
\end{split}
\]
which gives, by the algebraic properties of $H_x^{n/2+\varepsilon}(\mathbb{R}^n)$, 
\[
(II)_a\lesssim T^2 \Biggl(\int_0^T \int_{\R^n} \lambda(\abs{x})\abs{\Lambda^{s-1}D^\alpha u}^2 dx \, dt \Biggr) \st\norm{\lambda(\abs{x})^{-1}\partial_t u}^2_{H_x^{n/2+\varepsilon}} \st\norm{u}^{2(p+q-1)}_{H_x^{n/2+\varepsilon}}.
\]
Since $s-4N-5>n/2+\varepsilon+1$ for $\varepsilon>0$ small enough, we can conclude
\[
(II)_a\lesssim T^2\norm{u}_{X_T^s}^{2(p+q+1)}.
\]
For $(II)_b$ we have 
\begin{align}
(II)_b&=\sum_{\substack{\abs{\gamma_1}+\abs{\gamma_2}\leq s-1 \\ s/2-1/2 \leq \abs{\gamma_1} < s-1, \; \abs{\gamma_2}\leq s/2-1/2}}C_{\gamma_1,\gamma_2,s}\int_0^T\int_{\R^n}\lambda(\abs{x})^{-1}\abs{(D^{\gamma_1} D^\alpha u)(D^{\gamma_2}(u^p\bar{u}^q-u_0^p\bar{u}_0^q))}^2 \, dx \, dt \\
&\lesssim T^2\; \sum_{\substack{\abs{\gamma_1}+\abs{\gamma_2}\leq s-1 \\ s/2-1/2 \leq \abs{\gamma_1} < s-1, \; \abs{\gamma_2}\leq s/2-1/2}}\int_0^T \int_{\R^n}\abs{D^{\gamma_1}D^\alpha u}^2 \, dx \, dt \cdot \norm{\lambda(\abs{x})^{-1/2}D^{\gamma_2}\partial_t u^p\bar{u}^q}^2_{L^\infty_t L^\infty_x} \\
& \lesssim T^2 \; \sum_{\substack{\abs{\gamma_1}+\abs{\gamma_2}\leq s-1 \\ s/2-1/2 \leq \abs{\gamma_1} < s-1, \; \abs{\gamma_2}\leq s/2-1/2}}\int_0^T \int_{\R^n}\abs{D^{\gamma_1}D^\alpha u}^2 \, dx \, dt \cdot \norm{\lambda(\abs{x})^{-1}D^{\gamma_2}\partial_t u^p\bar{u}^q}^2_{L^\infty_t H_x^{n/2+\varepsilon}} \\
&\lesssim \; T^3 \sum_{\substack{\abs{\gamma_1}+\abs{\gamma_2}\leq s-1 \\ s/2-1/2 \leq \abs{\gamma_1} < s-1, \; \abs{\gamma_2}\leq s/2-1/2}} \norm{u}_{L^\infty_t H_x^{|\gamma_1|+2}}^2 \cdot \norm{\lambda(\abs{x})^{-1}D^{\gamma_2}\partial_tu^p\bar{u}^q}^2_{L^\infty_t H_x^{n/2+\varepsilon}}.\label{1411}
\end{align}
since, recall, $\abs{\alpha}\leq 2$.  To estimate the norm in \eqref{1411} we use Lemma \ref{lemmatec1} on $\tilde{u}:=\partial_t u^p\bar{u}^q$, so that, for $|\gamma_2|\leq s/2-1/2$, we have
\begin{align}
\norm{\lambda(\abs{x})^{-1}D^{\gamma_2}\tilde{u}}^2_{L^\infty_t H_x^{n/2+\varepsilon}} 
& \lesssim \norm{D^{\gamma_2}\lambda(\abs{x})^{-1}\tilde{u}}^2_{L^\infty_t H_x^{n/2+\varepsilon}}+ \sum_{j=1}^n \norm{Op(p_{\abs{\gamma_2}-1})x_j \langle x \rangle ^{2N-2}\tilde{u}}^2_{L^\infty_tH_x^{n/2+\varepsilon}} \\
    & + \sum_{\substack{\abs{\alpha+\beta} \leq 2N, \\ \abs{\alpha} \geq 2, \abs{\beta}\leq 2N-2}}\norm{Op(p_{\abs{\gamma_2}-\abs{\alpha}})x^\beta \tilde{u}}^2_{L^\infty_tH_x^{n/2+\varepsilon}} \\
    & \lesssim \norm{\lambda(\abs{x})^{-1}\tilde{u}}^2_{L^\infty_t H_x^{n/2+\varepsilon+\abs{\gamma_2}}}+ \sum_{j=1}^n \norm{x_j \langle x \rangle ^{2N-2}\tilde{u}}^2_{L^\infty_tH_x^{n/2+\varepsilon+\abs{\gamma_2}-1}} \\
    & + \sum_{\substack{\abs{\alpha+\beta} \leq 2N, \\ \abs{\alpha} \geq 2, \abs{\beta}\leq 2N-2}}\norm{x^\beta \tilde{u}}^2_{L^\infty_tH_x^{n/2+\varepsilon+\abs{\gamma_2}-\abs{\alpha}}} \\
    & \lesssim \norm{\lambda(\abs{x})^{-1}\tilde{u}}^2_{L^\infty_t H_x^{n/2+\varepsilon+\abs{\gamma_2}}}+ \sum_{j=1}^n \norm{\frac{x_j}{\langle x \rangle^2} \langle x \rangle^{2N}\tilde{u}}^2_{L^\infty_tH_x^{n/2+\varepsilon+\abs{\gamma_2}-1}} \\
    & + \sum_{\substack{\abs{\alpha+\beta} \leq 2N, \\ \abs{\alpha} \geq 2, \abs{\beta}\leq 2N-2}}\norm{\frac{x^\beta}{\langle x \rangle^{2N}}\langle x \rangle^{2N} \tilde{u}}^2_{L^\infty_tH_x^{n/2+\varepsilon+\abs{\gamma_2}-\abs{\alpha}}} \\
    &\lesssim \st\norm{\lambda(\abs{x})^{-1}\tilde{u}}^2_{H_x^{n/2+\varepsilon+\abs{\gamma_2}}}=\st\norm{\lambda(\abs{x})^{-1}\partial_t u^p\bar{u}^q}^2_{H_x^{n/2+\varepsilon+\abs{\gamma_2}}}
\end{align}
where $p_r$,  for $r \in \mathbb{R}$, stands for a symbol of order $r$. 
Since $n/2+s/2-1/2+\varepsilon\leq s-4N-5$ for $\varepsilon$ small enough, we can conclude the estimate
\[
(II)_b \lesssim T^3 \norm{u}_{X^s_T}^{2(p+q+1)}.
\]
Finally, by Sobolev's theorem and Lemma \ref{lemmatec1}, for $(II)_c$ we get
\[
\begin{split}
(II)_c&=\sum_{\substack{\abs{\gamma_1}+\abs{\gamma_2}\leq s-1 \\ \abs{\gamma_1} < s/2-1/2, \; \abs{\gamma_2}\geq s/2-1/2}}C_{\gamma_1,\gamma_2,s}\int_0^T \int_{\R^n}\lambda(\abs{x})^{-1}\abs{(D^{\gamma_1}D^\alpha u)(D^{\gamma_2}(u^p\bar{u}^q-u_0^p\bar{u}_0^q))}^2 \, dx \, dt \Bigr) \\
&\lesssim T\sum_{\substack{\abs{\gamma_1}+\abs{\gamma_2}\leq s-1 \\ \abs{\gamma_1} < s/2-1/2, \; \abs{\gamma_2}\geq s/2-1/2}}\|\lambda(\abs{x})^{-1}D^{\gamma_1}D^\alpha u\|_{L^\infty_tL^\infty_x}^2\norm{D^{\gamma_2}(u^p\bar{u}^q-u_0^p\bar{u}_0^q)}_{L^\infty_tL^2_x}^2\\
&\lesssim T \|\lambda(\abs{x})^{-1}u\|_{L^\infty_t H_x^{n/2+s/2-3/2+\varepsilon+2}}^2(\| u\|^{2(p+q)}_{L^\infty_tH_x^{s-1}}+\| u_0\|^{2(p+q)}_{H_x^{s-1}} )\\
&\lesssim T \|u\|_{X_T^s}^2(\| u\|^{2(p+q)}_{X_T^s}+\| u_0\|^{2(p+q)}_{H_x^{s}} )\\
&\lesssim T (1+\| u_0\|^{2(p+q)}_{H_x^{s}})(\| u\|^{2(p+q+1)}_{X_T^s} +\|u\|_{X_T^s}^2).
\end{split}
\]
For $(II)_d$, we can proceed as in $(II)_c$ and have 
\[(II)_d\lesssim T (1+\| u_0\|^{2(p+q)}_{H_x^{s}})(\| u\|^{2(p+q+1)}_{X_T^s} +\|u\|_{X_T^s}^2).\]
Putting together all the estimates, for $T\leq 1$ and some $C>0$ we have
\[
(I)+(II) \leq C \left(\norm{u_0}^2_{H^s_x}+T (1+\| u_0\|^{2(p+q)}_{H_x^{s}})(\| u\|^{2(p+q)}_{X_T^s}+\| u\|^{2(p+q+1)}_{X_T^s} +\|u\|_{X_T^s}^2)\right).
\]

Now, we proceed with the estimate of
\[
(III)+(IV)=\st\norm{\lambda (\abs{x})^{-1}v}_{ H^{s-4N-2}_x}^2+ \st\norm{\lambda(\abs{x})^{-1} \partial_t v}^2_{H_x^{s-4N-5}}.
\]
The idea is to use Lemma \ref{lemmatec2} as follows. Let us start  with $(III)$, for which we have
\begin{align}
(III)&=\st\norm{\lambda (\abs{x})^{-1}v}_{ H^{s-4N-2}_x}^2 \\
& \lesssim \st\norm{\lambda (\abs{x})^{-1}\widetilde{W}(t)u_0}_{H^{s-4N-2}_x}^2+\st\norm{\int_0^t\lambda (\abs{x})^{-1}\widetilde{W}(t-t') 
\tilde{N}(u,\bar{u},D^\alpha u) dt'}_{H^{s-4N-2}_x}^2. \label{eqn1409.1} 
\end{align}
Then, for $T\leq 1$, $s> n+8N+9$, and $\varepsilon<1$, using Lemma \ref{lemmatec1}, Lemma \ref{lemmatec2}, and Lemma 6.0.1. in \cite{FS21}, by \eqref{eqn1409.1} we obtain
\begin{align}
(III)
&\lesssim (1+T^{2N})^2\norm{\lambda(\abs{x})^{-1}u_0}^2_{H_x^{s-2}}+\Bigl(T\sup_{0\leq t'\leq t\leq  T}\norm{\lambda (\abs{x})^{-1}\widetilde{W}(t-t')(u^p\bar{u}^q-u_0^p\bar{u}_0^q)D^\alpha u}_{H^{s-4N-2}_x}\Bigr)^2 \\
&\lesssim (1+T^{2N})^2\norm{\lambda(\abs{x})^{-1}u_0}^2_{H_x^{s-2}}+\Bigl(T(1+T^{2N})^{1/2}\st\norm{\lambda(\abs{x})^{-1}(u^p\bar{u}^q-u_0^p\bar{u}_0^q)D^\alpha u}_{H_x^{s-2}}\Bigr)^2\\
&\lesssim
(1+T^{2N})^2\norm{\lambda(\abs{x})^{-1}u_0}_{H_x^{s-2}}^2+
T^2(1+T^{2N})  \Bigl( \st\norm{ \lambda(\abs{x})^{-1}D^\alpha u}_{H_x^{n/2+\varepsilon}}\st\norm{u^p\bar{u}^q-u_0^p\bar{u}_0^q}_{H_x^{s-2}}\\
&\quad +  \st\norm{D^\alpha u}_{H_x^{s-2}}\st\norm{\lambda(\abs{x})^{-1}(u^p\bar{u}^q-u_0^p\bar{u}_0^q)}_{H_x^{n/2+\varepsilon}}\Bigr)^2\\
&\lesssim
(1+T^{2N})\norm{\lambda(\abs{x})^{-1}u_0}_{H_x^{s-2}}^2+
T^2(1+T^{2N}) \Bigl(  \st\|\lambda(\abs{x})^{-1}u\|_{ H_x^{n/2+2+\varepsilon}}^2\Big(\st\|u\|^{2(p+q)}_{ H_x^{s-2}}+\|u_0\|^{2(p+q)}_{ H_x^{s-2}}\Big)\\
&\quad + \st\|u\|_{H_x^s}^2 \Big( \st\| \lambda(\abs{x})^{-1} u \|^2_{H_x^{n/2+\varepsilon}}\st\|u\|_{ H^{n/2+\varepsilon}_x}^{2(p+q-1)}+ \| \lambda(\abs{x})^{-1} u_0\|^2_{H_x^{n/2+\varepsilon}}\|u_0\|_{ H^{n/2+\varepsilon}_x}^{2(p+q-1)}
\Bigr)\\
&\lesssim \norm{\lambda(\abs{x})^{-1}u_0}^2_{H_x^{s-2}}+T \Big(1+ \norm{u_0}^{2(p+q)}_{H_x^{s-2}}+ \| \lambda(\abs{x})^{-1} u_0\|^{2}_{H_x^{n/2+\varepsilon}}\|u_0\|_{ H^{n/2+\varepsilon}_x}^{2(p+q-1)}\Big)
\Bigg(\|u\|_{X^s_T}^{2(p+q+1)}+\|u\|_{X^s_T}^{2}\Bigg).
\end{align}

As for term $(IV)$, it can be estimated by using the estimate for term $(III)$ as follows.
Since $s> n+8N+9$, by Lemma \ref{lemmatec1} we have
\begin{align}
(IV)&=\st\norm{\lambda(\abs{x})^{-1} \partial_t v}^2_{H_x^{s-4N-5}}\\
    &\lesssim \st\|\lambda(\abs{x})^{-1}\tilde{A}v\|_{H_x^{s-4N-5}}^2+\st\|\lambda(\abs{x})^{-1}\tilde{N}(u,\bar{u},D^\alpha u)\|_{H_x^{s-4N-5}}^2\\
&\lesssim\st\norm{\lambda(\abs{x})^{-1}v}_{H_x^{s-4N-2}}^2+\st\|\lambda(\abs{x})^{-1}(D^\alpha u)(u^p\bar{u}^q-u_0^p\bar{u}_0^q)\|_{H_x^{s-4N-5}}^2\\
&\lesssim (III)+\st\|\lambda(\abs{x})^{-1}(D^\alpha u)(\int_0^t\partial_t u^p\bar{u}^qdt')\|_{H_x^{s-4N-5}}^2.
\end{align}
Using the algebra property of $H_x^{s-4N-5}(\R^n)$ ($s-4N-5>n/2$),  along with Lemma \ref{lemmatec1} and Minkowski's inequality, we get
\begin{align}
(IV) &\lesssim (III) + \st\norm{D^\alpha u}_{H_x^{s-4N-5}}^2 \st\norm{\lambda(\abs{x})^{-1}(\int_0^t\partial_t u^p\bar{u}^qdt')}^2_{H_x^{s-4N-5}} \\
&\lesssim  (III) +T^2\st\norm{u}_{H_x^{s-4N-3}}^2 \st\norm{\lambda(\abs{x})^{-1}\partial_t u^p\bar{u}^q}^2_{H_x^{s-4N-5}}. \\
&\lesssim (III)+T^2\st\norm{u}^{2(p+q)}_{H_x^{s-4N-3}}\st\norm{\lambda(\abs{x})^{-1}\partial_t u}^2_{H_x^{s-4N-5}} \\
&\lesssim (III)+T^2\norm{u}^{2(p+q+1)}_{X_T^s}.
\end{align}

Therefore, collecting all the estimates, for $T\leq 1$, we obtain
\begin{align}
    \|\Phi_{u_0}(u)\|_{X^s_T}^2&=(I)+(II)+(III)+(IV)\\
    &\lesssim (\|u_0\|^2_{ H^s_x}+\|\lambda(\abs{x})^{-1}u_0\|^2_{H^{s-2}_x})\\
    &\quad + T (1+\|u_0\|^{2(p+q)}_{ H^s_x}+
    \|\lambda(\abs{x})^{-1}u_0\|^2_{H_x^{s-2}}\norm{u_0}_s^{2(p+q-1)}
    )(\|u\|_{X_T^s}^{2(p+q)}+\|u\|_{X_T^s}^{2(p+q+1)}+\|u\|_{X_T^s}^{2})\\
    &\leq C_1 (\|u_0\|^2_{ H^s_x}+\|\lambda(\abs{x})^{-1}u_0\|^2_{H^{s-2}_x})\\
    &\quad+ C_2T\Big(1+\Big(\|u_0\|^{2}_{H^s_x}+\|\lambda(\abs{x})^{-1}u_0\|^{2}_{H^{s-2}_x}\Big)^{p+q}\Big)\Big(\|u\|_{X_T^s}^{2(p+q)}+\|u\|_{X_T^s}^{2(p+q+1)}+\|u\|_{X_T^s}^{2}\Big),\label{eqn.contraction}
\end{align}
for some $C_1,C_2>0$.
Now, by choosing $R\geq \sqrt{2C_1(\|u_0\|^2_{H^s_x}+\|\lambda(\abs{x})^{-1}u_0\|^2_{H^{s-2}_x})}$ and 
\begin{align}
    T=\min\left\{1 , \frac{R^2}{2 C_2(1+(2C_1)^{-(p+q)}R^{2(p+q)})(R^{2(p+q)}+R^{2(p+q+1)}+R^2)}\right\},
\end{align}
we conclude that
\[\|\Phi_{u_0}(u)\|_{X^s_T}^2\leq R^2,\]
hence $\Phi_{u_0}$ sends $B_R$ into itself, which concludes the first step of the proof.

The second step of the proof is to show that the map $\Phi_{u_0}$ is a contraction on $B_R$. Given $u_1,u_2\in B_R$, we denote $v_1:=\Phi_{u_0}(u_1)$ and $v_2:=\Phi_{u_0}(u_2)$. The application of the linear smoothing estimates on $v_1-v_2$ yields 
\begin{equation}\label{eqcont}
\begin{split}
\norm{\Phi_{u_0}(u_1)-\Phi_{u_0}(u_2)}^2_{X_T^s}&=\norm{v_1-v_2}_{X_T^s}^2\\
&= \st\norm{v_1-v_2}_{H^s_x}^2+\int_0^T \int_{\R^n}\lambda(\abs{x})\abs{\Lambda^{s+1}(v_1-v_2)}^2 \, dx \, dt+ \\
    &\quad + \st\norm{\lambda (\abs{x})^{-1}(v_1-v_2)}_{H^{s-4N-2}_x}^2+\st\norm{ \lambda (\abs{x})^{-1}\partial_t (v_1-v_2)}^2_{H_x^{s-4N-5}}\\
    &\lesssim \int_0^T\int_{\R^n}\lambda(\abs{x})^{-1}\abs{\Lambda^{s-1}(\tilde{N}(u_1)-\tilde{N}(u_2))}^2 \, dx \, dt +\\
    &\quad + \st\norm{\lambda (\abs{x})^{-1}(v_1-v_2)}_{ H^{s-4N-2}_x}^2+\st\norm{ \lambda (\abs{x})^{-1}\partial_t (v_1-v_2)}^2_{H_x^{s-4N-5}}\\
    &=(V)+(VI)+(VII),
\end{split}
\end{equation}
where, for $j=1,2$,
\begin{align}
    \tilde{N}(u_j):= (u_j^p\bar{u}_j^q -u_0^p\bar{u}_0^q)D^\alpha u_j.
\end{align}

In what follows, we estimate only the term $(V)$, since the terms $(VI)$ and $(VII)$ can be handled similarly. 
First we write
\[
\begin{split}
\tilde{N}(u_1)-\tilde{N}(u_2)&=(u_1^p\bar{u}_1^q -u_0^p\bar{u}_0^q)D^\alpha u_1- (u_2^p\bar{u}_2^q -u_0^p\bar{u}_0^q)D^\alpha u_2\\
&= u_1^p\bar{u}_1^qD^\alpha u_1- u_2^p\bar{u}_2^qD^\alpha u_2 -u_0^p\bar{u}_0^q(D^\alpha u_1- D^\alpha u_2) \\
&=u_1^p\bar{u}_1^qD^\alpha u_1-u_2^p\bar{u}_1^qD^\alpha u_1+u_2^p\bar{u}_1^qD^\alpha u_1+\\
&\quad -u_2^p\bar{u}_2^qD^\alpha u_2 -u_0^p\bar{u}_0^q(D^\alpha u_1- D^\alpha u_2)\\
&=(u_1^p-u_2^p)\bar{u}_1^qD^\alpha u_1+u_2^p\bar{u}_1^qD^\alpha u_1 + \\
&\quad -u_2^p\bar{u}_2^qD^\alpha u_1 +u_2^p\bar{u}_2^qD^\alpha u_1 + \\
&\quad -u_2^p\bar{u}_2^qD^\alpha u_2 -u_0^p\bar{u}_0^q(D^\alpha u_1- D^\alpha u_2)\\
&=(u_1^p-u_2^p)\bar{u}_1^qD^\alpha u_1+(\bar{u}_1^q-\bar{u}_2^q)u_2^pD^\alpha u_1 +\\
&\quad +(D^\alpha u_1-D^\alpha u_2)u_2^p\bar{u}_2^q -u_0^p\bar{u}_0^q(D^\alpha u_1- D^\alpha u_2)\\
&= (u_1^p-u_2^p)\bar{u}_1^qD^\alpha u_1+(\bar{u}_1^q-\bar{u}_2^q)u_2^pD^\alpha u_1 +\\
&\quad +(D^\alpha u_1-D^\alpha u_2)(u_2^p\bar{u}_2^q -u_0^p\bar{u}_0^q),\\
\end{split}\]
and note that 
\begin{align}
 u_1^p-u_2^p&= \int_{0}^t \partial_s(u_1^p-u_2^p)(s) \, ds\\
 u_2^p\bar{u}_2^q -u_0^p\bar{u}_0^q&= \int_{0}^t \partial_s u_2^p\bar{u}_2^q(s) \, ds.
\end{align}
Then, splitting the nonlinear term as the sum of the previous three terms, we get that $(V)$ can be estimated, up to a constant, by the sum of three terms too, that we call $(V)'$,$(V)''$ and $(V)'''$. 
The first term of the sum in $(V)$, that is $(V)'$,  is then given by
\[
(V)':=\int_0^T\int_{\R^n}\lambda(\abs{x})^{-1}\abs{\Lambda^{s-1}(u_1^p-u_2^p)\bar{u}_1^qD^\alpha u_1}^2 \, dx \, dt.
\]
To estimate $(V)'$ we repeat the steps in the estimate of $(II)_a,(II)_b$ and $(II)_c$. When $p=0$ we have $(V)'=0$, therefore we focus on the case $p\geq 1$. When $p>1$ we obtain

\begin{align}
(V)'&=\int_0^T\int_{\R^n}\lambda(\abs{x})^{-1}\abs{\Lambda^{s-1}(u_1^p-u_2^p)D^\alpha u_1 \bar{u}_1^q}^2 \, dx \, dt \\
&\lesssim  \int_0^T \int_{\R^n}\lambda(\abs{x})\abs{\Lambda^{s-1} D^\alpha u_1}^2\cdot \abs{\lambda(\abs{x})^{-1} (u_1^p-u_2^p)\bar{u}_1^q}^2 \, dx \, dt \\
& + \sum_{\substack{\abs{\gamma_1}+\abs{\gamma_2}\leq s-1 \\ s/2-1/2 \leq \abs{\gamma_1} < s-1, \; \abs{\gamma_2}\leq s/2-1/2}}C_{\gamma_1,\gamma_2,s}\int_0^T \int_{\R^n}\lambda(\abs{x})^{-1}\abs{(D^{\gamma_1} D^\alpha u_1)(D^{\gamma_2}(u_1^p-u_2^p)\bar{u}_1^q)}^2 \, dx \, dt \\
& + \sum_{\substack{\abs{\gamma_1}+\abs{\gamma_2}\leq s-1 \\  \abs{\gamma_1} < s/2-1/2, \; \abs{\gamma_2}\geq s/2-1/2}}C_{\gamma_1,\gamma_2,s}\int_0^T \int_{\R^n}\lambda(\abs{x})^{-1}\abs{(D^{\gamma_1} D^\alpha u_1)(D^{\gamma_2}(u_1^p-u_2^p)\bar{u}_1^q)}^2 \, dx \, dt\\
& + \sum_{\substack{\abs{\gamma_1}+\abs{\gamma_2}\leq s-1 \\  \abs{\gamma_1} < s/2-1/2, \; \abs{\gamma_2}< s/2-1/2}}C_{\gamma_1,\gamma_2,s}\int_0^T \int_{\R^n}\lambda(\abs{x})^{-1}\abs{(D^{\gamma_1} D^\alpha u_1)(D^{\gamma_2}(u_1^p-u_2^p)\bar{u}_1^q)}^2 \, dx \, dt\\
&\lesssim T^2\|u_1\|_{X^s_T}^2 \|\lambda(\abs{x})^{-1}\partial_t(u_1^p-u_2^p)\|^2_{L^\infty_tH_x^{n/2+\varepsilon}}\|u_1\|^{2q}_{L^\infty_tH_x^{n/2+\varepsilon}}+(V)'_b+(V)'_c+(V)'_d\label{11091}\\
\end{align}
Since
\[\p_t(u_1^{p}-u_2^p)=pu_1^{p-1}\p_t(u_1-u_2)+p(\p_tu_2)(u_1^{p-1}-u_2^{p-1})\]
and
\[u_1^{p-1}-u_2^{p-1}=(u_1-u_2)\sum_{k=0}^{p-2}u_1^{p-2-k}u_2^k,\, \]
we have
\begin{equation}
  \begin{aligned}
    \|\lambda(|x|)^{-1}\partial_t(u_1^p-u_2^p)\|_{L_t^\infty H_x^{n/2+\varepsilon}}
    &\le C_p\Bigl(
      R^{p-1}\|\lambda(|x|)^{-1}\partial_t(u_1-u_2)\|_{L_t^\infty H_x^{n/2+\varepsilon}}\\
      &\quad 
      +R^{p-2}\|u_1-u_2\|_{L_t^\infty H_x^{n/2+\varepsilon}}
      \|\lambda(|x|)^{-1}\partial_tu_2\|_{L_t^\infty H_x^{n/2+\varepsilon}}
      \Bigr)\\
    &\le C_pR^{p-1}\|u_1-u_2\|_{X^s_T},
  \end{aligned}
  \label{eq:stima-derivata}
\end{equation}
which yields 
\begin{equation}\label{eq:stima (V)_a}
    (V)'\lesssim T^2 R^{2(p+q)} \|u_1-u_2\|^2_{X^s_T}+(V)'_b+(V)'_c+(V)'_d.
\end{equation}

When $p=1$, we can repeat the above computations up to \eqref{11091} and get
\begin{align}
    (V)'&\lesssim T^2\|u_1\|_{X^s_T}^2 \|\lambda(\abs{x})^{-1}\partial_t(u_1-u_2)\|^2_{L^\infty_tH_x^{n/2+\varepsilon}}\|u_1\|^{2q}_{L^\infty_tH_x^{n/2+\varepsilon}}+(V)'_b+(V)'_c+(V)'_d\\
    &\lesssim T^2R^{2(q+1)}\|u_1-u_2\|_{X^s_T}^{2}+(V)'_b+(V)'_c+(V)'_d
\end{align}
Therefore, for all $p\geq 0$
\begin{align} \label{est(V)'_a}
    (V)'\lesssim  T^2 R^{2(p+q)} \|u_1-u_2\|^2_{X^s_T}+(V)'_b+(V)'_c+(V)'_d.
\end{align}
Now, for $\varepsilon$ sufficiently small as above, when $q\geq 1$, 
\begin{align}
 (V)'_b&\lesssim T \sum_{\substack{|\gamma_2'|+|\gamma_2''|\leq s/2-1/2}}\|u_1\|_{L^\infty_tH_x^{s}}^2\|\lambda(\abs{x})^{-1}D^{\gamma_2'}(u_1^p-u_2^p) \|_{L^\infty_tH_x^{n/2+\varepsilon}}^2\|D^{\gamma_2''}\bar{u}_1^q \|_{L^\infty_tH_x^{n/2+\varepsilon}}^2\\
 &\lesssim T \|u_1\|_{L^\infty_tH_x^{s}}^2\|\lambda(\abs{x})^{-1}(u_1^p-u_2^p) \|_{L^\infty_tH_x^{s/2-1/2+n/2+\varepsilon}}^2\|u_1 \|_{L^\infty_tH_x^{s/2-1/2+n/2+\varepsilon}} ^{2q}\\
 &\lesssim T \sum_{j=0}^{p-1}\|u_1\|_{L^\infty_tH_x^{s}}^2\|\lambda(\abs{x})^{-1}(u_1-u_2) \|_{L^\infty_tH_x^{s-4N-5}}^2\|u_1\|^{2(p-j-1)}_{L^{\infty}_tH^{s-4N-5}}\|u_2\|^{2j}_{L^{\infty}_tH_x^{s-4N-5}}\|u_1 \|_{L^\infty_tH_x^{s-4N-5}} ^{2q}\\
 &\lesssim T  \|u_1-u_2\|_{X^s_T}^2\|u_1\|_{X^s_T}^{2(q+1)}(\|u_1\|_{X^s_T}+\|u_2\|_{X^s_T})^{2(p-1)}\\
  &\lesssim T R^{2(p+q)}\|u_1-u_2\|^2_{X^s_T}.
\end{align}
When $q=0$ the estimate above is still true without the use of Leibniz rule on $D^{\gamma_2}$ in $(V)'_b$.

Analogously,
\begin{align}
   (V)'_c&\lesssim T \|\lambda(\abs{x})^{-1} u_1\|^2_{L^\infty_t H_x^{s/2+1/2+n/2+\varepsilon}} \|u_1-u_2\|^2_{L^\infty_tH_x^{s-1}} (\|u_1\|_{L^\infty_tH_x^{s-1}}+\|u_2\|_{L^\infty_tH_x^{s-1}})^{2(p-1)}\|u_1\|^{2q}_{L^\infty_tH_x^{s-1}}\\
   &\lesssim T  \|u_1-u_2\|_{X^s_T}^2\|u_1\|_{X^s_T}^{2(q+1)}(\|u_1\|_{X^s_T}+\|u_2\|_{X^s_T})^{2(p-1)}\\
   &\lesssim T R^{2(p+q)}\|u_1-u_2\|^2_{X^s_T},\\
   (V)'_d &\lesssim T R^{2(p+q)}\|u_1-u_2\|^2_{X^s_T}.
\end{align}
Hence, for $T<1$,
\begin{align}
    (V)' &\lesssim T R^{2(p+q)}\|u_1-u_2\|^2_{X^s_T}.
\end{align}
By similar calculations, when $q=0$ we have $(V)''=0$, while when $q\geq 1$ we have
\begin{align}
    (V)''&:=\int_0^T\int_{\R^n}\lambda(\abs{x})^{-1}\abs{\Lambda^{s-1}(\bar{u}_1^q-\bar{u}_2^q)u_2^pD^\alpha u_1}^2 \, dx \, dt \\
    &\lesssim T \|u_2\|_{X^s_T}^{2p}\|u_1\|_{X^s_T}^2\|u_1-u_2\|_{X^s_T}^{2}(\|u_1\|_{X^s_T}+\|u_2\|_{X^s_T})^{2(q-1)}\\
     &\lesssim T R^{2(p+q)}\|u_1-u_2\|^2_{X^s_T}.
\end{align}
To estimate $(V)'''$ we observe that, when for $q\neq 0, p\neq0$, we have
\begin{align}
(V)'''&:= \int_0^T\int_{\R^n}\lambda(\abs{x})^{-1}\abs{\Lambda^{s-1}(D^\alpha u_1-D^\alpha u_2)(u_2^p\bar{u}_2^q -u_0^p\bar{u}_0^q)}^2 \, dx \, dt \\
&\lesssim  \int_0^T \int_{\R^n}\lambda(\abs{x})\abs{\Lambda^{s-1} D^\alpha (u_1-u_2)}^2\cdot \abs{\lambda(\abs{x})^{-1} (u_2^p\bar{u}_2^q-u_0^p\bar{u}_0^q)}^2 \, dx \, dt \\
& + \sum_{\substack{\abs{\gamma_1}+\abs{\gamma_2}\leq s-1 \\ s/2-1/2 \leq \abs{\gamma_1} < s-1, \; \abs{\gamma_2}\leq s/2-1/2}}C_{\gamma_1,\gamma_2,s}\int_0^T \int_{\R^n}\lambda(\abs{x})^{-1}\abs{(D^{\gamma_1} D^\alpha (u_1-u_2))(D^{\gamma_2}(u_2^p\bar{u}_2^q-u_0^p\bar{u}_0^q)}^2 \, dx \, dt \\
& + \sum_{\substack{\abs{\gamma_1}+\abs{\gamma_2}\leq s-1 \\  \abs{\gamma_1} < s/2-1/2, \; \abs{\gamma_2}\geq s/2-1/2}}C_{\gamma_1,\gamma_2,s}\int_0^T \int_{\R^n}\lambda(\abs{x})^{-1}\abs{(D^{\gamma_1} D^\alpha (u_1-u_2))(D^{\gamma_2}(u_2^p\bar{u}_2^q-u_0^p\bar{u}_0^q)}^2 \, dx \, dt\\
& + \sum_{\substack{\abs{\gamma_1}+\abs{\gamma_2}\leq s-1 \\  \abs{\gamma_1} < s/2-1/2, \; \abs{\gamma_2}< s/2-1/2}}C_{\gamma_1,\gamma_2,s}\int_0^T \int_{\R^n}\lambda(\abs{x})^{-1}\abs{(D^{\gamma_1} D^\alpha (u_1-u_2))(D^{\gamma_2}(u_2^p\bar{u}_2^q-u_0^p\bar{u}_0^q)}^2 \, dx \, dt\\
&\lesssim T^2\Bigg( \int_{0}^T\int_{\R^n}\lambda(\abs{x})|\Lambda^{s+1}(u_1-u_2)|^2\,dx\,dt\,\Bigg) \,\|\lambda(\abs{x})^{-1}\partial_t(u_2^p\Bar{u}_2^q)\|_{L^\infty_tH_x^{n/2+\varepsilon}}^2+(V)'''_b+(V)'''_c+(V)'''_d\\
&\lesssim T^2 \|u_1-u_2\|_{X^s_T}^2
\big(\|u_2^{p-1}\bar{u}_2^q\lambda(\abs{x})^{-1}\partial_tu_2\|^2_{L^\infty_tH_x^{n/2+\varepsilon}}+ \|u_2^{p}\bar{u}_2^{q-1}\lambda(\abs{x})^{-1}\partial_t\bar{u}_2\|^2_{L^\infty_tH_x^{n/2+\varepsilon}}\big)\label{eqn1409.2}\\
&+(V)'''_b+(V)'''_c+(V)'''_d \\
&\lesssim T^2 \|u_1-u_2\|_{X^s_T}^2\|\lambda(\abs{x})^{-1}\partial_tu_2\|_{L^\infty_t H_x^{s-4N-5}}^2\|u_2\|_{L^\infty_t H^s_x}^{2(p+q-1)}+(V)'''_b+(V)'''_c+(V)'''_d\\
&\lesssim T^2 \|u_1-u_2\|_{X^s_T}^2\|u_2\|_{X^s_T}^{2(p+q)}+ (V)'''_b+(V)'''_c+(V)'''_d.
\end{align}
On the other hand, when $p=0$ or $q=0$, we will simply have one term in \eqref{eqn1409.2}, so the same estimate as above holds for $(V)'''$.

Now, for $(V)'''_b$ we have
\begin{align}
    (V)'''_b&\lesssim T^2\|u_1-u_2\|^2_{X^s_T}\|\lambda(\abs{x})^{-1}\partial_t u_2\|^2_{L^\infty_tH_x^{n/2+s/2-1/2+\varepsilon}}\|u_2\|_{L^\infty_tH_x^{n/2+s/2-1/2+\varepsilon}}^{2(p+q-1)}\\
    &\lesssim T^2 \|u_1-u_2\|_{X^s_T}^2\|u_2\|_{X^s_T}^{2(p+q)}\\
     &\lesssim T R^{2(p+q)}\|u_1-u_2\|_{X^s_T}^2,
\end{align}
and for $(V)'''_c$
\begin{align}
    (V)'''_c&\lesssim T \|\lambda(\abs{x})^{-1}(u_1-u_2)\|^2_{L^\infty_t H_x^{s/2+1/2+n/2+\varepsilon}}\big(\|u_2\|^{2(p+q)}_{L^\infty_t H_x^{s-1}} + \|u_0\|^{2(p+q)}_{H_x^{s-1}} \big)\\
    &\lesssim T \|u_1-u_2\|_{X^s_T}^2\big(\|u_2\|^{2(p+q)}_{X^s_T} + \|u_0\|^{2(p+q)}_{ H_x^{s}} \big)\\
    &\lesssim T R^{2(p+q)}\|u_1-u_2\|_{X^s_T}^2+T \|u_0\|^{2(p+q)}_{ H_x^{s}} \|u_1-u_2\|_{X^s_T}^2 .
\end{align}
For $(V)'''_d$ we can proceed as for $(V)'''_c$. 
Therefore, collecting the estimates for each term and using that $u_1,u_2\in B_R$, we get
\begin{align}
    (V)&\lesssim T (\|u_0\|^{2(p+q)}_{ H_x^{s}} + R^{2(p+q)}) \|u_1-u_2\|_{X^s_T}^2\\
    & \lesssim T (\|u_0\|^{2(p+q)}_{ H^{s}_x}+\|\lambda(\abs{x})^{-1}u_0\|^{2(p+q)}_{ H_x^{s-2}} + R^{2(p+q)})\|u_1-u_2\|^2_{X^s_T}.
\end{align}

At this point, since $(VI)$ and $(VII)$ satisfy the same estimate as $(V)$, we obtain
\begin{align}
    \norm{\Phi_{u_0}(u_1)-\Phi_{u_0}(u_2)}_{X_T^s}\leq  \sqrt{T} C(R)\norm{u_1-u_2}_{X_T^s}
\end{align}
for some positive constant $C(R)$  depending only on  $R=R(\|u_0\|_{ H_x^{s}},\|\lambda(\abs{x})^{-1}u_0\|_{ H_x^{s-2}})$. 
Finally,  taking $T$ possibly smaller than before (depending on $R$, hence on $\|u_0\|_{ H_x^{s}}$ and $\|\lambda(\abs{x})^{-1}u_0\|_{ H_x^{s-2}}$),  we conclude that
\begin{equation}\label{eqcont2}
\norm{\Phi_{u_0}(u_1)-\Phi_{u_0}(u_2)}_{X_T^s}\leq C \norm{u_1-u_2}_{X_T^s},
\end{equation}
for some $C<1$. This shows that $\Phi_{u_0}:B_R\rightarrow B_R$ is a contraction, which, in particular, gives the first part of the theorem by the  fixed point theorem. 

We are now left with the proof of the second part of the theorem, that is, the continuity property in the statement.
Let $u_0\in \mathscr{S}(\R^n)$, and let $U_{u_0}$ be a suitable neighborhood of $u_0$ in $\mathscr{S}(\R^n)$. Then, for $v_0 \in U_{u_0}$ and $u,v$ solutions to \eqref{NLIVP3} with initial values $u_0$ and $v_0$ respectively, 
we have that $u$ and $v$ solve  
\[
\begin{cases}
\partial_tu-i\tilde{A}u=(u^p\bar{u}^q-u_0^p\bar{u}_0^q)D^\alpha u \\
u(0,\cdot)=u_0
\end{cases}
\]
and 
\[
\begin{cases}
\partial_tv-i\tilde{A}v=(v^p\bar{v}^q-v_0^p\bar{v}_0^q)D^\alpha v-(u_0^p\bar{u}_0^q-v_0^p\bar{v}_0^q)D^\alpha v \\
v(0,\cdot)=v_0
\end{cases}
\]
respectively, where, recall, $-i\tilde{A}=-iA-u_0^p\bar{u}_0^qD^\alpha$.

Hence, denoting $w=u-v$ and $w_0=u_0-v_0$, then $w$ solves 
\[
\begin{cases}
\partial_tw-i\tilde{A}w=\Bigl((u^p\bar{u}^q-u_0^p\bar{u}_0^q)-(v^p\bar{v}^q-v_0^p\bar{v}_0^q)\Bigr)D^\alpha u+(v^p\bar{v}^q-v_0^p\bar{v}_0^q)D^\alpha w +(u_0^p\bar{u}_0^q-v_0^p\bar{v}_0^q)D^\alpha v \\
w(0,\cdot)=w_0.
\end{cases}
\]
Therefore, calling $W(t)$ the solution operator of the problem
\begin{equation}
    \left\{
\begin{array}{ll}
 \partial_t g -i\tilde{A}g=0 \quad \quad  \mathrm{on} \quad \R^n \times ]0,T[, \quad T>0,\\
 g(0,\cdot)=g_0 \in \mathscr{S}(\R^n), 
\end{array}
    \right.
\end{equation}
we get that 
\[
\begin{split}
w(t)&=W(t)w_0+\int_0^t W(t-t')\left(\Bigl((u^p\bar{u}^q-u_0^p\bar{u}_0^q)-(v^p\bar{v}^q-v_0^p\bar{v}_0^q)\Bigr)D^\alpha u+(v^p\bar{v}^q-v_0^p\bar{v}_0^q)D^\alpha w \right)(t')dt'\\
&+\int_0^t W(t-t')(u_0^p\bar{u}_0^q-v_0^p\bar{v}_0^q)D^\alpha v(t')dt',\\
\end{split}
\]
that we may rewrite as 
\[
\begin{split}
u(t)-v(t)&=W(t)(u_0-v_0)\\
&+\int_0^t W(t-t')\left(\Bigl((u^p\bar{u}^q-u_0^p\bar{u}_0^q)-(v^p\bar{v}^q-v_0^p\bar{v}_0^q)\Bigr)D^\alpha u+(v^p\bar{v}^q-v_0^p\bar{v}_0^q)D^\alpha (u-v) \right)(t')dt'\\
&+\int_0^t W(t-t')(u_0^p\bar{u}_0^q-v_0^p\bar{v}_0^q)D^\alpha v(t')dt',
\end{split}
\]
where, in particular, 
\begin{align}
    u^p\bar{u}^q- u_0^p\bar{u}_0^q&=\int_0^t\partial_s(u^p\bar{u}^q)(s) ds, \label{eqn09031848}\\
    {v}^p\bar{v}^q-{v}_0^p\bar{v}_0^q&=\int_0^t\partial_s{v}^p\bar{v}^q(s) ds \label{eqn09031849}\\
    u_0^p\bar u_0^q-v_0^p\bar v_0^q
&=(u_0-v_0)\bar u_0^q
\sum_{j=0}^{p-1}u_0^{p-1-j}v_0^j
+(\bar u_0-\bar v_0)v_0^p
\sum_{k=0}^{q-1}\bar u_0^{q-1-k}\bar v_0^k.\label{eqn1809-1}
\end{align}
Therefore, since $u$ and $v$ belong to a ball (centered at $0$) in $ X^s_{T'}$ with $T'$ to be determined, the estimates above, especially those leading to \eqref{eqn.contraction} which exploit \eqref{eqn09031848}, \eqref{eqn09031849}, \eqref{eqn1809-1} and once again Leibniz's rule, give that there exist  $C', C''>0$, with $C',C''$ depending on the radius of the ball, such that
\[
\norm{u-v}^2_{X_{T'}^s}\leq C' (\norm{u_0-v_0}_{H_x^s}^2+\norm{\lambda(\abs{x})^{-1}(u_0-v_0)}^2_{H_x^{s-2}})+C''T'\norm{u-v}^2_{X_{T'}^s}.
\]
Hence, for $T'>0$ satisfying $C''T'<1/2$, we get
\begin{equation}
\norm{u-v}_{X_{T'}^s}\leq \sqrt{2C'} (\norm{u_0-v_0}_{H_x^s}+\|\lambda(\abs{x})^{-1}(u_0-v_0)\|_{H^{s-2}_x}).\label{09031034}
\end{equation}
Finally, since the right hand side of \eqref{09031034} is bounded by finitely many seminorms in $\mathscr{S}$, we can conclude that the map 
\[
U_{u_0} \ni v_0 \mapsto v \in X_{T'}^s
\]
is continuous.
\end{proof}

\section{Non-trapping for some admissible operators}\label{sec.nontrapping}
Here we investigate the so called \textit{nontrapping} properties of a class of operators of order $m$ with principal symbol of the form \eqref{amdiff} below. The nontrapping property of an operator $A$, or, equivalently, of its principal symbol $a_m$, at a point $(x_0,\xi_0)$ of the phase space, refers to the following behavior of the bicharacteristics: the bicharacteristic curves of the principal symbol $a_m$ of $A$ starting from $(x_0,\xi_0)$ escape from any compact set (see Definition \ref{nontrapped}).

Smoothing estimates for variable coefficient operators of order two, that is for Schr\"odinger and ultrahyperbolic Schr\"odinger operators with variable coefficients,  are strongly related to the non-trapping property of the principal symbol of the space-dependent leading order operator $A$ in $D_t-A$. 
In \cite{D96} and \cite{KPRV05} the \textit{nontrapping condition}, which amounts to the non-trapping property of $A$ at each point of the cotangent space, is used to prove the suitable versions of Lemma \ref{newDoilem}, which, in turn, allowed the authors to establish the smoothing estimates contained therein. Moreover, even for the microlocal smoothing estimates derived in \cite{CKS}, which are proved by using a different strategy than that in \cite{D96}, the nontrapping condition is again a fundamental requirement to get the result. 

In the other direction, that is whether smoothing estimates for $D_t-A$ imply nontrapping properties of $A$, a positive result was established by Doi in \cite{D96(1)} when $A$ is an elliptic operator of order two with a real homogeneous symbol  $a\in S^2(\R^n)$.
More precisely, Doi showed that if the microlocal smoothing estimate for $D_t-A$ holds at a point $(x_0,\xi_0)$ of the cotangent space, then the point is nontrapped. We remark that study of the set of trapped points $S^A_\mathrm{trap}$ (see Definition \ref{nontrapped} below), for a given operator $A$, is of interest independently of its implications in the validity of smoothing estimates. In this respect, a result is due to Doi in \cite{D96}, who studied $S^A_\mathrm{trap}$ when $A$ is elliptic of order two without assuming any smoothing-type property for $D_t-A$.
Doi's result (see Lemma 1.3 in \cite{D96}) shows that if the (real) principal symbol $a_2 $ of $A\in \Psi_{\mathrm{cl}}^2(\R^n)$ satisfies suitable conditions (additional to ellipticity), then the set $S^A_{\mathrm{trap}}$ is compact.

Note that all the results listed above refer to the case $A\in\Psi_{\mathrm{cl}}^m(\R^n)$ with $m=2$.  In general, when the operator has variable coefficients, it is very hard to describe the behavior of the bicharacteristics, which is the main reason why we avoided any use of these objects in our previous analysis. Indeed, our smoothing result in Section \ref{sec.smoothing.estimates} does not rely on the non-trapping condition $S^A_\mathrm{trap}=\emptyset$, since our Lemma \ref{newDoilem} is not based on this property.  
Nevertheless, due to the aforementioned connection between non-trapping conditions and smoothing effect, and to the lack of results on the non-trapping properties of operators of order greater than two at any or some point, even elliptic ones, we decided to investigate this problem here. 
To the best of our knowledge, for differential operators with variable coefficients of order $m\geq 3$, no previous result establishes global existence and spatial escape of bicharacteristics, not even under the assumptions \eqref{ellderivatives} and \eqref{smallcoeffa2a3}.

Let us consider here $A=\Op^\w(a) \in \Psi_{\mathrm{cl}}^m(\R^n)$, with $m\in \mathbb{N}$, such that
its principal symbol $a_m$ is of the form
\begin{equation}\label{amdiff}
a_m(x,\xi)=\sum_{\abs{\alpha}=m}a_\alpha(x)\xi^\alpha,
\end{equation}
 with $a_\alpha \in C_b^\infty(\R^n;\R)$ for all $\alpha \in \mathbb{N}_0^n$, $\abs{\alpha}=m$. 

Our goal is to show that under our conditions \eqref{ellderivatives} and \eqref{smallcoeffa2a3} on $a_m$, the nontrapping property of the so called \textit{strongly elliptic points} of $A$ (see Definition \ref{def.elliptic.cosphere} below) holds; in other words, the bicharacteristic curves of the principal symbol $a_m$ of $A$ starting from such points escape from any compact set.

Before going into details, we first fix the notation. For $(x_0,\xi_0) \in \R^n \times (\R^n \setminus \lbrace 0 \rbrace)$, let $\gamma_{(x_0,\xi_0)}:I \rightarrow \R^n \times \R^n$ be the maximally extended bicharacteristic curve of $a_m$ starting from $(x_0,\xi_0)$ (see Definition \ref{biccurve}), where $I\subseteq \R$ is the maximal interval. We shall denote such a curve by $(x(t;x_0,\xi_0),\xi(t;x_0,\xi_0))$ (or simply $(x(t),\xi(t)$) and  say that it is \textit{globally defined} when $I=\R$.

\begin{definition}\label{def.elliptic.cosphere}
Let $A \in \Psi^m(\R^n)$. We define \textit{the elliptic co-sphere associated with $A$} as
\[
S_{A}^\ast \R^n= \lbrace (x,\xi) \in \R^n \times (\R^n \setminus \lbrace 0 \rbrace); \ a_m(x,\xi)=1 \rbrace,
\]
where $a_m$ is the principal symbol of $A$. If $(x_0,\xi_0) \in S_A^\ast \R^n$, we say that $(x_0,\xi_0)$ is \textit{an elliptic point} for $A$. Moreover, we say that a point $(x_0,\xi_0) \in S_A^\ast(\R^n)$ is \textit{a strongly elliptic point} for $A$ if there exists a constant $C=C(x_0,\xi_0)>0$ such that 
\begin{equation}\label{ellsymbol}
C^{-1}|\xi(t)|^m\leq |a_m(x(t),\xi(t))|\leq C |\xi(t)|^m, \quad \forall t \in I.
\end{equation}
\end{definition}
\begin{remark}
Note that, by Definition \ref{def.elliptic.cosphere}, the standard co-sphere 
\[
S^\ast \R^n:=\lbrace (x,\xi) \in \R^n \times (\R^n \setminus  \lbrace 0 \rbrace); \ \abs{\xi}=1 \rbrace,
\]
(where here $|\cdot|$ denotes the standard Euclidean norm) can be seen as the elliptic co-sphere associated with the (positive) Laplacian, i.e. $S^\ast \R^n=S_\Delta^\ast \R^n$.
\end{remark}
For the reader's convenience, we recall here the definition of \textit{trapped} and \textit{nontrapped points} (see Craig, Kappeler and Strauss \cite{CKS} or Doi \cite{D96}).
\begin{definition}\label{nontrapped}
Let $A \in \Psi^m(\R^n)$. We say that $(x_0,\xi_0)\in \R^n \times (\R^n \setminus \lbrace 0 \rbrace)$ is \textit{forward (backward) nontrapped} (by the bicharacteristics of the principal symbol $a_m$), if $(x(t;x_0,\xi_0),\xi(t;x_0,\xi_0))$ is globally defined, and 
\[
\lim_{t\rightarrow +\infty}\abs{x(t;x_0,\xi_0)}=+\infty \quad (\lim_{t\rightarrow -\infty}\abs{x(t;x_0,\xi_0)}=+\infty).
\]
A point $(x_0,\xi_0) \in \R^n \times (\R^n \setminus \lbrace 0 \rbrace)$ is said to be \textit{nontrapped} if it is either forward or backward nontrapped. Finally, we say that a point is \textit{trapped} if it is not nontrapped, and we denote by 
\[
S^A_{\mathrm{trap}}=\Bigl\lbrace (x_0,\xi_0) \in \R^n \times (\R^n \setminus \lbrace 0 \rbrace); \ (x_0,\xi_0) \ \mathrm{is} \ \mathrm{trapped }\Bigr\rbrace.
\]
\end{definition}
\begin{remark}
    Note that a point $(x_0,\xi_0) \in \R^n\times (\R^n \setminus \lbrace 0 \rbrace)$ is nontrapped if and only if the bicharacteristic through $(x_0,\xi_0)$ is globally defined and 
    \[
\forall R>0\ \exists T_R>0\ \text{such that}\
 \abs{x(t;x_0,\xi_0)}\geq R,\,\, \forall t\ge T_R\, \text{or}\ \forall t\leq -T_R.
    \]
\end{remark}
Our goal is now to prove the following theorem showing a non-trapping property at \textit{strongly elliptic points} for operators (not necessarily elliptic) whose principal symbol satisfies \eqref{ellderivatives} and \eqref{smallcoeffa2a3}.  
Let us emphasize that the strong ellipticity condition \eqref{ellsymbol}, although formulated as a property of the initial point, requires a uniform bound along the entire maximal bicharacteristic uniquely determined by that point.
\begin{theorem}
Let $A \in \Psi_{\mathrm{cl}}^m(\R^n)$, $m\in\mathbb{N},$ be such that its principal symbol $a_m$ is of the form \eqref{amdiff} and satisfies \eqref{ellderivatives} and \eqref{smallcoeffa2a3}. Then, if $(x_0,\xi_0) \in S_A^\ast(\R^n)$ is a strongly elliptic point, the bicharacteristic curve $\gamma_{(x_0,\xi_0)}$ is globally defined (i.e. $I=\R$) and $(x_0,\xi_0) \notin S^A_{\mathrm{trap}}$.
\end{theorem}
\begin{proof}
We start the proof by showing that the bicharacteristic curve starting from a strongly elliptic point $(x_0,\xi_0) \in S^\ast_A \R^n$ is globally defined. To see this it suffices to note that $a_m$ is constant and satisfies \eqref{ellsymbol} on $S^\ast_A \R^n$. Consequently, $1/c \leq \abs{\xi(t)} \leq c$ for all $t \in I$, for some constant $c=c(x_0,\xi_0)>0$. Thus, using a Gronwall-type argument, we can conclude that the bicharacteristic curve starting from $(x_0,\xi_0)$ is globally defined. 

We now show that $(x_0,\xi_0)$ is nontrapped. We define $q_0 \in C^\infty(\R^n \times (\R^n\setminus \lbrace 0 \rbrace))$ as 
\begin{equation}
q_0(x,\xi):=|\xi|^{-(m-1)}\sum_{j=1}^n x_j\partial_{\xi_j}a_m(x,\xi), \qquad (x,\xi) \in \R^n \times (\R^n\setminus \lbrace 0 \rbrace).
\end{equation}
Note that, as observed before $|\xi(t)|>0$ for all $t \in \R$, hence $q_0$ is well defined on $(x(t),\xi(t))$ for all $t \in \R$ and, moreover, 
\begin{equation}\label{eq.boundq0}
|q_0(x,\xi)|\lesssim |x|\lesssim \langle x \rangle, \quad \forall (x,\xi) \in \R^n \times (\R^n\setminus \lbrace 0 \rbrace).
\end{equation}
By the fundamental theorem of integral calculus we have
\begin{equation}\label{geomdelta}
q_0(x(t),\xi(t))=q_0(x_0,\xi_0)+\int_0^t H_{a_m}q_0(x(s),\xi(s)) \, ds.
\end{equation}
Now, by repeating the argument in \eqref{ellderimplyadmiss}, we have 
\begin{equation}
H_{a_m} q_0(x,\xi) \geq \abs{\xi}^{-(m-1)}|\nabla_\xi a_m(x,\xi)|^2-|\xi|^{-(m-1)}\sum_{j,k=1}^n| x_j(\partial_{\xi_k}a_m)(x,\xi)(\partial_{\xi_j}\partial_{x_k}a_m)(x,\xi)| - |\nabla_x a_m(x,\xi)||\nabla_\xi q_0(x,\xi)|,
\end{equation}
for all $(x,\xi) \in \R^n \times (\R^n\setminus \lbrace 0 \rbrace)$,
and hence, by \eqref{ellderivatives}, \eqref{smallcoeffa2a3} and by homogeneity of $a_m$ and $q_0$ (in the $\xi$ variables), we get 
\begin{equation}\label{eq.Hamq0}
H_{a_m}q_0(x,\xi) \geq (C-\varepsilon C')|\xi|^{m-1}, \quad \forall \R^n \times (\R^n\setminus \lbrace 0 \rbrace),
\end{equation}
for some $C,C'>0$ universal constants. 

When $m=1$ then, for $\varepsilon$ sufficiently small, we have
\begin{equation}\label{eq.m1}
H_{a_m} q_0(x,\xi) \geq C,
\end{equation}
for some $C>0$. Therefore, when $m=1$, from \eqref{eq.boundq0}, \eqref{geomdelta} and \eqref{eq.m1},
\[
\langle x(t) \rangle \gtrsim \abs{q_0(x(t),\xi(t))} \gtrsim q_0(x_0,\xi_0)+C t,\quad \forall t\geq 0,
\]
which shows that $(x_0,\xi_0)$ is nontrapped.

For $m>1$, by applying \eqref{eq.Hamq0} into \eqref{geomdelta}, we obtain
\begin{equation}\label{eq.newq0}
q_0(x(t),\xi(t))\geq q_0(x_0,\xi_0)+\int_0^t(C-\varepsilon C') \abs{\xi(s)}^{m-1}ds.
\end{equation}
At this point, since $|a_m(x_0,\xi_0)|=|a_m(x(s),\xi(s))|$ for all $s \in \R$, we have that $(x(s),\xi(s)) \in S^\ast_A\R^n$ for all $s \in \R$, and that
\begin{equation}\label{eq.newboundbic}
( C-\varepsilon C') \abs{\xi(s)}^{m-1}\geq (C-\varepsilon C') c^{-(m-1)/m}|a_m(x(0),\xi(0))|^{(m-1)/m}>\mu,
\end{equation}
for some $\mu>0$, for all $s \in \R$, if $\varepsilon>0$ is sufficiently small.
Therefore, by \eqref{eq.boundq0}, \eqref{eq.newq0} and \eqref{eq.newboundbic}, we obtain 
\[
\langle x(t) \rangle \gtrsim \abs{q_0(x(t),\xi(t))} \gtrsim q_0(x_0,\xi_0)+\mu t,\quad \forall t\geq 0,
\]
which shows that $(x_0,\xi_0)$ is nontrapped and concludes the proof. 
\end{proof}


\appendix

\section{Technical lemmas}\label{sec.teclem}

\begin{lemma}\label{lemmatec0}
    Let $p \in S^m(\R^n)$ and let $\gamma \in \bN_0^n$. Then 
    \[
    [x^\gamma,\Op(p)]f=\sum_{0<\alpha\leq\gamma}c_{\alpha}\Op(\partial_\xi^\alpha p)x^{\gamma-\alpha}f, \quad f \in \mathscr{S}(\R^n),
    \]
    for some $c_\alpha\in \mathbb{C}$.
\end{lemma}

\begin{lemma}\label{lemmatec1}
    Let $p \in S^m(\R^n)$ and $N \in \mathbb{N}$. Then
    \[
    \begin{split}
    \langle x \rangle^{2N}\Op(p)f&=\Op(p)\langle x \rangle ^{2N}f+2N \sum_{j=1}^n \Op(i\partial_{\xi_j}p)x_j \langle x \rangle ^{2N-2}f \\
    &+ \sum_{\substack{\abs{\alpha+\beta} \leq 2N, \\ \abs{\alpha} \geq 2, \abs{\beta}\leq 2N-2}}C_{\alpha \beta}\Op(\partial_\xi^\alpha p)x^\beta f, \quad \quad f \in \mathscr{S}(\mathbb{R}^n).
    \end{split}
    \]
\end{lemma}
\begin{lemma}\label{lemmatec2}
Let $N \in \mathbb{N}$ and $s \in \mathbb{R}$. Suppose $\langle x \rangle^{2N}u_0 \in H^{s+4N}(\R^n)$.
Then, there exist $c_0,\dots,c_{2N},c>0$ constants, such that for $T \in [0,1]$
\[
    \sup_{0 \leq t \leq T} \norm{\langle x \rangle^{2N}\widetilde{W}(t)u_0}^2_{H_x^s}\leq \sum_{j=0}^{2N}c_jT^j \norm{\langle x \rangle^{2N-j}u_0}_{H_x^{s+2j}}^2
\]
and
\begin{equation}\label{eq.lemmatecenu}
    \sup_{0 \leq t \leq T} \norm{\langle x \rangle^{2N}\widetilde{W}(t)u_0}^2_{H_x^s}\leq c(1+T^{2N}) \norm{\langle x \rangle^{2N}u_0}_{H_x^{s+4N}}^2,
\end{equation}
where $\widetilde{W}(t)$ denotes the solution operator of the linear homogeneous initial value problem \eqref{homLIVP}.
\end{lemma}
\begin{proof}
Let us denote by $u=\widetilde{W}(t)u_0$ the solution of \eqref{homLIVP}. We have
\begin{equation}\label{eq.lemmatec2}
\partial_t \langle x \rangle^{2N}u=\langle x \rangle^{2N}i\tilde{A}u.
\end{equation}
We now write $\tilde{A}=\Op(\tilde{a}_{KN})$, for some $\tilde{a}_{KN} \in S^3(\R^n)$. By Lemma \ref{lemmatec0} we have
\[
[\langle x \rangle^{2N}, \Op(\tilde{a}_{KN})]=\sum_{1\leq |\alpha|\leq 2N}c_\alpha \Op(\partial_\xi^\alpha \tilde{a}_{KN})\partial_x^\alpha \langle x \rangle^{2N},
\]
for some $c_\alpha\in \mathbb{C}$. Hence, from \eqref{eq.lemmatec2} we have 
\[
\partial_t \langle x \rangle^{2N} u=i\tilde{A}\langle x \rangle^{2N}u+\sum_{j=1}^nB_jx_j\langle x \rangle^{2N-2}u+B_0\langle x \rangle^{2N-2}u,
\]
for some $B_j \in \Psi^2(\R^n)$, for $j=0,\dots,n$. Therefore, by Theorem \ref{thm.LIVP} we have 
\begin{equation}\label{eq.rhslemmatec2}
\begin{split}
\sup_{0\leq t \leq T}\|\langle x \rangle^{2N}\widetilde{W}(t)u_0\|_{H^s}^2 &\lesssim \|\langle x \rangle^{2N}u_0\|_{H^s}^2+\int_0^T \Bigl(\sum_{j=1}^n \|x_j \langle x \rangle^{2N-2}\widetilde{W}(t)u_0\|^2_{H^{s+2}}+ \| \langle x \rangle^{2N-2}\widetilde{W}(t)u_0\|^2_{H^{s+2}}\Bigr)dt \\
& \lesssim \|\langle x \rangle^{2N}u_0\|_{H^s}^2+T \sup_{0\leq t \leq T}\Bigl(\sum_{j=1}^n \|x_j \langle x \rangle^{2N-2}\widetilde{W}(t)u_0\|^2_{H^{s+2}}+ \| \langle x \rangle^{2N-2}\widetilde{W}(t)u_0\|^2_{H^{s+2}}\Bigr).
\end{split}
\end{equation}
Next, by arguing as before, applying again Lemma \ref{lemmatec0}, we have 
\[
\partial_t\langle x \rangle^{2N-2}u=i\tilde{A}\langle x \rangle^{2N-2}u+C_0 \langle x \rangle^{2N-2}u, \quad C_0 \in \Psi^2(\R^n),
\]
and 
\[
\partial_t x_j\langle x \rangle^{2N-2}u=i \tilde{A}x_j\langle x \rangle^{2N-2}u+C_j\langle x \rangle^{2N-2}u, \quad C_j \in \Psi^2(\R^n), \ \forall j=1,\dots,n.
\]
Hence, by using the linear estimates of Theorem \ref{thm.LIVP} on the RHS of \eqref{eq.rhslemmatec2} we obtain 
\[
\sup_{0\leq t \leq T}\|\langle x \rangle^{2N}\widetilde{W}(t)u_0\|_{H^s}^2\lesssim \|\langle x \rangle^{2N}u_0\|^2_{H^s}+T\|\langle x \rangle^{2N-1}u_0\|^2_{H^{s+2}}+T^2\sup_{0\leq t \leq T}\|\langle x \rangle^{2N-2}\tilde{W}(t)u_0\|^2_{H^{s+4}}.
\]
We now iterate this argument to estimate the last term $\|\langle x \rangle^{2N-2}\tilde{W}(t)u_0\|^2_{H^{s+4}}$ by replacing $(N,s)$ with $(N-1,s+4), \dots, (1,s+4N-4)$ and obtain 
\begin{equation}\label{eq.lemmatec2use}
    \sup_{0 \leq t \leq T} \norm{\langle x \rangle^{2N}\widetilde{W}(t)u_0}^2_{H_x^s}\leq \sum_{j=0}^{2N}c_jT^j \norm{\langle x \rangle^{2N-j}u_0}_{H_x^{s+2j}}^2,
\end{equation}
for some $c_0,\dots,c_{2N}>0$. Finally, to obtain \eqref{eq.lemmatecenu} it is sufficient to observe that 
\[
\|\langle x \rangle^{2N-j} u_0\|_{H^{s+2j}}\lesssim \|\langle x \rangle^{2N} u_0\|_{H^{s+2j}} \lesssim \|\langle x \rangle^{2N} u_0\|_{H^{s+4N}}, \quad \forall j=0,\dots,2N,
\]
and thus, by \eqref{eq.lemmatec2use}, we have 
\[
\sup_{0 \leq t \leq T} \norm{\langle x \rangle^{2N}\widetilde{W}(t)u_0}^2_{H_x^s}\leq c(1+T^{2N}) \norm{\langle x \rangle^{2N}u_0}_{H_x^{s+4N}}^2,
\]
for some $c>0$.
\end{proof}

\end{document}